\documentclass[a4paper,9pt]{amsart}
\usepackage[english]{babel}
\usepackage[T1]{fontenc}
\usepackage{graphicx}
\usepackage{amsmath, amsfonts,amsthm, mathrsfs, amssymb}
\usepackage{verbatim}

\newcommand{\Ham}{\mathsf}

\usepackage{enumerate}
\usepackage[latin1]{inputenc}
\usepackage{geometry}

\usepackage{soul,xcolor,empheq}
\usepackage[many]{tcolorbox}
\tcbset{
  myhlight/.style={
    colback=yellow,
    arc=0pt,
    outer arc=0pt,
    boxrule=0pt,
    top=0pt,
    bottom=0pt,
    left=0pt,
    right=0pt,
  },
  highlight math style={myhlight}
}
\newtcolorbox{hllong}{breakable,myhlight}

\newcommand{\N}{{\mathbb N}}

\newcommand{\R}{{\mathbb R}}

\newcommand{\T}{{\mathbb T}}
\newcommand{\Z}{{\mathbb Z}}

\newcommand{\cB}{{\mathcal B}}

\newcommand{\cE}{{\mathcal E}}
\newcommand{\cF}{{\mathcal F}}
\newcommand{\cG}{{\mathcal G}}
\newcommand{\cH}{{\mathcal H}}

\newcommand{\cO}{{\mathcal O}}

\newcommand{\cR}{{\mathcal R}}

\newcommand{\cT}{{\mathcal T}}
\newcommand{\cU}{{\mathcal U}}

\newcommand{\cZ}{{\mathcal Z}}

\renewcommand{\d}{\partial}

\newcommand{\ii}{\mathrm{i}}

\newcommand{\la}{\left\langle}
\newcommand{\ra}{\right\rangle}

\def\sleq{\lesssim}

\def\di{{\rm d}}

\newtheorem{theorem}{Theorem}[section]
\newtheorem{lemma}[theorem]{Lemma}
\newtheorem{corollary}[theorem]{Corollary}
\newtheorem{proposition}[theorem]{Proposition}
\newtheorem{definition}[theorem]{Definition}
\newtheorem{remark}[theorem]{Remark}

\title[]{Metastability phenomena in two-dimensional rectangular lattices with nearest-neighbour interaction}

\author{M. Gallone$^{(\dagger)}$}
\address[$\dagger$]{Dipartimento di Matematica, Universit\`a degli Studi di Milano, Via Cesare Saldini 50, 20133, Milan, Italy}
\email[$\dagger$]{matteo.gallone@unimi.it}
\author{S. Pasquali$^{(\ast)}$}
\address[$\ast$]{Matematikcentrum, Lunds Universitet, S\"olvegatan 18, 223 62 Lund, Sweden.}
\email[$\ast$]{stefano.pasquali@math.lu.se}

\begin{document}

\begin{abstract}
We study analytically the dynamics of two-dimensional rectangular lattices with periodic boundary conditions. We consider anisotropic initial data supported on one low-frequency Fourier mode. We show that, in the continuous approximation, the resonant normal form of the system is given by integrable PDEs. We exploit the normal form in order to prove the existence of metastability phenomena for the lattices. More precisely, we show that the energy spectrum of the normal modes attains a distribution in which the energy is shared among a packet of low-frequencies modes; such distribution remains unchanged up to the time-scale of validity of the continuous approximation. 
\end{abstract}
\maketitle
\tableofcontents

\section{Introduction} \label{intro}

In this paper we present an analytical study of the dynamics of two-dimensional rectangular lattices with nearest-neighbour interaction and periodic boundary conditions, for initial data with only one low-frequency Fourier mode initially excited. We give some rigorous results concerning the relaxation to a metastable state, in which energy sharing takes place among low-frequency modes only. \\

The study of metastability phenomena for lattices started with the numerical result by Fermi, Pasta and Ulam (FPU) \cite{fermi1995studies}, who investigated the dynamics of a one-dimensional chain of particles with nearest neighbour interaction. In the original simulations all the energy was initially given to a single low-frequency Fourier mode with the aim of measuring the time of relaxation of the system to the `thermal equilibrium' by looking at the evolution of the Fourier spectrum. Classical statistical mechanics prescribes that the energy spectrum corresponding to the thermal equilibrium is a plateau (the so-called theorem of equipartition of energy). Despite the authors believed that the approach to such an equilibrium would have occurred in a short time-scale, the outcoming Fourier spectrum was far from being flat and they observed two features of the dynamics that were in contrast with their expectations: the lack of thermalization displayed by the energy spectrum and the recurrent behaviour of the dynamics.

Both from a physical and a mathematical point of view, the studies on FPU-like systems have a long and active history: a concise survey of this vast literature is discussed in the monograph \cite{gallavotti2007fermi}. For a more recent account on analytic results on the `FPU paradox' we refer to \cite{bambusi2015some}. \\

In particular, we mention the papers \cite{bambusi2006metastability} and \cite{bambusi2008resonance}, in which the authors used the techniques of canonical perturbation theory for PDEs in order to show that the FPU $\alpha$ model (respectively, $\beta$ model) can be rigorously described by a system of two uncoupled KdV (resp. mKdV) equations, which are obtained as a resonant normal form of the continuous approximation of the FPU model; moreover, this result allowed to deduce a rigorous result about the energy sharing among the Fourier modes, up to the time-scales of validity of the approximation. If we denote by $N$ the number of degrees of freedom for the lattice and by $\mu \sim \frac{1}{N} \ll 1$ the wave-number of the initially excited mode, if we assume that the specific energy $\epsilon \sim \mu^4$ (resp. $\epsilon \sim \mu^2$ for the FPU $\beta$ model), then the dynamics of the KdV (resp. mKdV) equations approximates the solutions of the FPU model up to a time of order $\cO(\mu^{-3})$. However, the relation between the specific energy and the number of degrees of freedom implies that the result does not hold in the thermodynamic limit regime, namely for large $N$ and for fixed specific energy $\epsilon$ (such a regime is the one which is relevant for statistical mechanics). \\

Unlike the extensive research concerning one-dimensional systems, it seems to the authors that the behaviour of the dynamics of two-dimensional lattices is far less clear; it is expected that the interplay between the geometry of the lattice and the specific energy regime could lead to different results.

Benettin and collaborators \cite{benettin1980stochastic} \cite{benettin2005time} \cite{benettin2008study} studied numerically a two-dimensional FPU lattice with triangular cells and different boundary conditions in order to estimate the equipartition time-scale. They found out that in the thermodynamic limit regime the equipartition is reached faster than in the one-dimensional case. The authors decided not to consider model with square cells in order to have a spectrum of linear frequencies which is different with respect to the one of the one-dimensional model; they also added 
(see \cite{benettin2008study}, Section B.(iii) )\\

\emph{There is a good chance, however, that models with square lattice, and perhaps a different potential so as to avoid instability, behave differently from models with triangular lattice, and are instead more similar to one-dimensional models. This would correspond to an even stronger lack of universality in the two-dimensional FPU problem.} \\

Up to the authors' knowledge, the only analytical results  on the dynamics of two-dimensional lattices in this framework concern the existence of breathers \cite{wattis1994solitary} \cite{butt2006discrete} \cite{butt2007discrete} \cite{yi2009discrete} \cite{bambusi2010existence}.

In this paper we study two-dimensional rectangular lattices with $(2N_1+1) \times (2 N_2+1)$ sites, square cell, nearest-neighbour interaction and periodic boundary conditions, and we show the existence of metastability phenomena as in \cite{bambusi2006metastability}. More precisely, if we denote by $\mu \ll 1$ the wave-number of the Fourier mode initially excited and by $\sigma$ the ratio between the sides of the lattice, we obtain for a 2D Electrical Transmission lattice (ETL) either a system of two uncoupled KP-II equations for $\mu \ll 1$ and $\sigma=2$, or a system of two uncoupled KdV equations for $\mu \ll 1$ and $2 < \sigma < {7}$ as a resonant normal form for the continuous approximation of the lattice, while for the 2D Klein-Gordon lattice with quartic defocusing nonlinearity we obtain a one-dimensional cubic defocusing NLS equation for $\mu \ll 1$ and $1 < \sigma < 7$. Since all the above PDEs are integrable, we can exploit integrability to deduce a mathematically rigorous result on the formation of the metastable packet. 

Up to the authors' knowledge, this is the first analytical result about metastable phenomena in two-dimensional Hamiltonian lattices with periodic boundary conditions; in particular, this is the first rigorous result for two-dimensional lattices in which the dynamics of the lattice in a two-dimensional regime is described by a system of two-dimensional integrable PDEs. \\

Some comments are in order:
\begin{enumerate}
\item[i.] 
the time-scale of validity of our result is of order $\cO(\mu^{-3})$ for the 2D ETL lattice, and of order $\cO(\mu^{-2})$ for the 2D Klein-Gordon lattice;
\item[ii.] the ansatz about the small amplitude solutions gives a relation between the specific energy of the system $\epsilon$ and the wave-number $\mu \sim \frac{1}{N_1}$ of the Fourier mode initially excited. More precisely, we obtain $\epsilon \sim \mu^4$ for the 2D ETL lattice as in \cite{bambusi2006metastability}, and $\epsilon \sim \mu^2$ for the 2D Klein-Gordon lattice. This implies that the result does not hold in the thermodynamic limit regime;
\item[iii.] our result can be easily generalized to higher-dimensional lattices, such as the physical case of three-dimensional rectangular lattices with cubic cells;
\item[iv.] depending on the geometry of the lattice which is encoded in the parameter $\sigma$, the effective dynamics is described by one-dimensional PDEs for highly anisotropic lattices and by two-dimensional PDEs for low values of $\sigma$. 
{ The normal form equation of the ETL lattices are not integrable for $1 \leq \sigma < 2$ and they are integrable if $\sigma > 2$. The edge case $\sigma=2$ is very sensitive to the potential: if the cubic term of the potential is present, then the normal form equation it is the integrable two-dimensional KP-II equation, otherwise it is a non-integrable modification of the mKdV equation. On the other hand, the normal form equation for KG lattices is integrable for $\sigma > 1$, thus also for less anisotropic lattices; }
\item[v.] the upper bounds for $\sigma$ in the KdV regime and in the NLS regime come from a technical assumption in the approximation results (see Proposition \ref{KdVxietaProp}, Proposition \ref{KPxietaProp} and Proposition \ref{1DNLSpsiProp}). The approximation of solutions for the lattice with solutions of integrable PDEs in one-dimensional lattices was obtained through a detailed analysis in order to bound the error, and this is also the case for two-dimensional lattices, (see Proposition \ref{KdVxietaProp}, Proposition \ref{1DNLSpsiProp}, Appendix \ref{ApprEstSec12} and Appendix \ref{ApprEstSec22}), where one has to do very careful estimates in order to bound the different contributions to the error.
\end{enumerate}

To prove our results we follow the strategy of \cite{bambusi2006metastability}. The first step consists in the approximation of the dynamics of the lattice with the dynamics of a continuous system. {This step gives also a natural perturbative order and, since $\sigma > 1$, the leading term is given by a PDE in only one space variable. The effect of the second dimension is of order $\mu^\sigma$, and thus it appears at the second perturbative order for low values of $\sigma$ and at higher perturbative orders for higher values. In this sense one expects that for higher values of $\sigma$ the normal form equations are one-dimensional, but it is not trivial that the effect of the second dimension do not destroy integrability. For this reason,} as a second step we perform a normal form canonical transformation and we obtain that the effective dynamics is given by a system of integrable PDEs (KdV, KP-II, NLS depending on the lattice and the relation between $N_1$ and $N_2$). Next, we exploit the dynamics of these integrable PDEs in order to construct approximate solutions of the original discrete lattices, and we estimate the error with repect to a true solution with the corresponding initial datum. Finally, we use the known results about the dynamics of the above mentioned integrable PDEs in order to estimate the specific energies for the approximate solutions of the original lattices. \\

The novelties of this work are: on the one side, a mathematically rigorous proof of the approximation of the dynamics of the 2D ETL lattice by the dynamics of certain integrable PDEs (among these integrable PDEs, there is one which is \emph{genuinely} two-dimensional, the KP-II equation) and of the dynamics of the 2D KG lattice by the dynamics of the one-dimensional nonlinear Schr\"odinger equation; on the other side, there are two technical differences with respect to previous works, namely the normal form theorem (which is a variant of the technique used in \cite{bambusi2002nonlinear} \cite{bambusi2005galerkin} \cite{pasquali2018dynamics}) and the estimates for bounding the error between the approximate solution and the true solution of the lattice (which need a more careful study than the ones appearing in \cite{schneider2000counter} \cite{bambusi2006metastability} for the one-dimensional case). \\

The paper is organized as follows: in Section \ref{results} we introduce the mathematical setting of the models and we state our main results, Theorem \ref{KPrThm}, Theorem \ref{KdVrThm} and Theorem \ref{1DNLSrThm}. In Section \ref{galavsec} we state an abstract Averaging Theorem, which we prove in Section \ref{BNFprsubsec}. In Section \ref{2Dsec} we apply the averaging Theorem to the two-dimensional lattices, deriving the integrable approximating PDEs in the different regimes. In Section \ref{BNFdyn} we review some results about the dynamics of the normal form equation. In Section \ref{ApprSec} we use the normal form equations in order to construct approximate solutions (see Proposition \ref{KdVxietaProp}, Proposition \ref{KPxietaProp} and Proposition \ref{1DNLSpsiProp}), and we estimate the difference with respect to the true solutions with corresponding initial data in Proposition \ref{ApprPropKdV}, Proposition \ref{ApprPropKP} and Proposition \ref{ApprProp1DNLS}. In Appendix \ref{BNFest} we prove the technical Lemma \ref{NFest}; in Appendix \ref{ApprEstSec11} we prove Propositions \ref{KPxietaProp} and \ref{KdVxietaProp}; in Appendix \ref{ApprEstSec12} we prove Propositions \ref{ApprPropKdV} and \ref{ApprPropKP}; in Appendix \ref{ApprEstSec21} we prove Proposition \ref{1DNLSpsiProp}; in Appendix \ref{ApprEstSec22} we prove Proposition \ref{ApprProp1DNLS}. \\

{\color{black}
	\emph{Relevant notations.} For the sake of clarity we provide a short explanation of some of the symbols we use in the paper. With $\mathbb{Z}_{N_1,N_2}^2$ we denote $\frac{\mathbb{Z}}{2N_1+1} \times \frac{\mathbb{Z}}{2N_2+1}=\{(j_1,j_2) \in \mathbb{Z}^2 \,: \, |j_1| \leq N_1, \, |j_2| \leq N_2\}$. We will denote by $k=(k_1,k_2) \in \mathbb{Z}^2$ the index of the wave-vector; we denote by $\kappa=(\kappa_1,\kappa_2)$ the specific wave vector defined in \eqref{kappa}. In Section \ref{ApprSec} and in the appendices we will denote by $K=(K_1,K_2)$ the index of the wave-vector with $| K_1| \leq N_1$ and $ |K_2| \leq N_2$. With $N_r$ ($r=1,2$) we denote the half-length of the rectangular lattice and with $N:=(2N_1+1)(2N_2+1)$ we denote the total number of sites.
	
	Space variables are denoted by $x_r$ (with $r=1,2$) with $x_r \in [-N_r,N_r]$ and the rescaled space variables are denoted by $y_r$ with $y_r \in [-1,1]$. 
	
	The Hilbert space $\ell_{\rho,s}^2$ is defined in Definition \ref{defsigk} and $B_{\rho,s}(R)$ denotes the open ball of radius $R$ centred in $0$, $B_{\rho,s}(R) \subset \ell_{\rho,s}^2$. We also use $\cH^{\rho,s}:=\ell^2_{\rho,s} \times \ell^2_{\rho,s}$ and $\cB_{\rho,s}(R):= B_{\rho,s}(R) \times B_{\rho,s}(R) \subset \cH^{\rho,s}$.
	
	We also denote with $|(a,b)|:=\sqrt{a^2+b^2}$ the Euclidean norm of the vector $(a,b)$. We denote with $[f]:=\int_0^T f \circ \Phi_{\Ham h}^t(y) \, \frac{\di t}{T}$ the time-average of $f$ with respect to $\Phi_{\Ham h}^t$.
	
	In Section \ref{2Dsec} we introduce interpolating functions and we denote with $Q(t,x)$ the interpolating function for the lattice variable $Q_j(t)$ and with $P(t,x)$ the interpolating function for the lattice variable $P_j(t)$. We use this little abuse of notation because the former is actually an extension of the latter in the sense that $P_j(t)=P(t,j)$ for every $t$ and $j$.
	
	The component $k$ of the Fourier transform of the continuous approximation at time $t$ will be denoted by $\hat{Q}(t,k)$ to distinguish it from the $k$-th component of the Fourier transform of the reticular variable $\hat{Q}_k(t)$. Since for the rescaled functions such as $q$ and $p$ this ambiguity do not hold, we prefer the short notation $\hat{q}_k(\tau)$ and $\hat{p}_k(\tau)$ for their Fourier transform.
	
{	In Section \ref{ApprSec} we make use of the following abuse of notation: $Q(t)=\xi(\tau)$. This is to avoid the long-writing $Q(t)=\xi(\tau(t))$.}
}

%
%
%

\section{Main Results} \label{results}

We consider a periodic two-dimensional rectangular lattice, called ETL lattice, which in the non-periodic setting has been studied in \cite{butt2006discrete}, and which can be regarded as a simpler version of a 2D rectangular FPU model. We denote
\begin{align} \label{Z2}
\Z^2_{N_1,N_2} &:= \{ (j_1,j_2): j_1,j_2 \in \Z, |j_1| \leq N_1, |j_2| \leq N_2 \};
\end{align}
we also write $e_1:=(1,0)$, $e_2:=(0,1)$ and {we denote by $N:=(2N_1+1)(2N_2+2)$ the total number of sites of the lattice.}

The Hamiltonian describing the ETL lattice is given by
\begin{align}
H_{\color{black}\mathrm{ETL}}(Q,P) &= \sum_{j \in \Z^2_{N_1,N_2}} - \frac{1}{2} P_j \, (\Delta_1P)_j + ({\color{black}V}(Q))_j, \label{HamQ} \\
(\Delta_1 P)_j &:= (P_{j+e_1}-2P_j+P_{j-e_1}) + (P_{j+e_2}-2P_j+P_{j-e_2}), \label{Delta1} \\
({\color{black}V}(Q))_j &= \frac{Q_j^2}{2} + \alpha \frac{Q_j^3}{3} + \beta \frac{Q_j^4}{4}. \label{FQ}
\end{align}
We refer to \eqref{HamQ} as $\alpha+\beta$ model (respectively, $\beta$ model) if $\alpha \neq 0$ (respectively $\alpha =0$). With the above Hamiltonian formulation the equations of motion associated to \eqref{HamQ} are given by
\begin{equation*} \label{eqQ4}
\begin{cases}
\dot{Q_j} = -(\Delta_1P)_j \\
\dot{P_j} = -({\color{black}V}'(Q))_j
\end{cases}
;
\end{equation*}
\begin{align} \label{2DETLeq}
\ddot{Q_j} &= (\Delta_1 {\color{black}V}'(Q))_j.
\end{align}

We also introduce the Fourier coefficients of $Q$ via the following standard relation,
{
\begin{equation} \label{fourierQ}
Q_j(t) \;:=\; \frac{1}{ \sqrt{N} } \sum_{k \in \Z^2_{N_1,N_2}} \hat{Q}_k(t) \, e^{2 \pi \ii \frac{j \cdot k}{ N }  } \, , \; \qquad j \in \Z^2_{N_1,N_2},
\end{equation}}
and similarly for $P_j$. We denote by
\begin{align}
E_k \;&:=\; \frac{\omega_k^2|\hat{P}_k|^2+ |\hat{Q}_k|^2}{2}, \label{EnNormMode} \\
\omega_k^2 \;&:=\;  4 \sin^2\left(\frac{k_1 \, \pi}{2N_1+1} \right) + 4 \sin^2\left(\frac{k_2 \, \pi}{2N_2+1} \right), \label{FreqNormMode}
\end{align}
the energy and the square of the frequency of the mode at site $k=(k_1,k_2) \in \Z^2_{N_1,N_2}$ (see Figure \ref{fig:freq_ETL}). For states described by real functions, one has $E_{(k_1,k_2)} = E_{(-k_1,k_2)}$ and $E_{(k_1,k_2)} = E_{(k_1,-k_2)}$ for all $k=(k_1,k_2)$, so we will consider only indexes in 
\begin{align*}
\Z^2_{N_1,N_2,+}&:= \{(k_1,k_2) \in \Z^2_{N_1,N_2} : k_1,k_2 \geq 0\}.
\end{align*}

\begin{figure} 
\begin{center}
\includegraphics[width=0.5\textwidth]{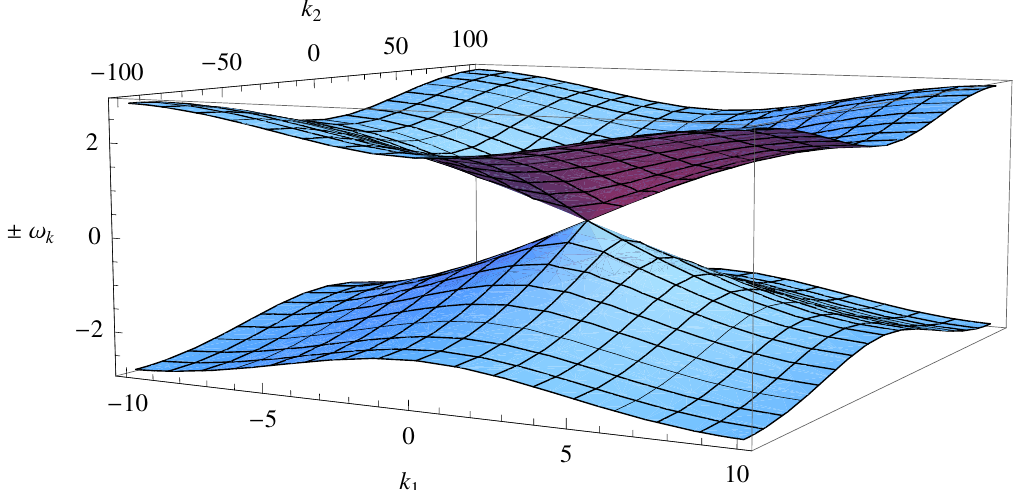}
\end{center}
\caption{The dispersion relation for the ETL lattice \eqref{HamQ}, namely $\pm \; \omega_k$, $k=(k_1,k_2)$, vs the integer coordinates, for $N_1=10$ and $N_2=100$.}
\label{fig:freq_ETL}
\end{figure}

As is customary in lattices with a large number of degrees of freedom, especially in relation with statistical mechanics, {we introduce the \emph{specific wave vector} $\kappa$ as
\begin{equation}\label{kappa}
\kappa \;:=\; \kappa(k) = \left( \frac{k_1}{ N_1+\frac{1}{2} }, \frac{k_2}{ N_2+\frac{1}{2} } \right) \, , 
\end{equation} 
{and the \emph{specific energy of the specific normal mode $\kappa=\kappa(k)$} as}
\begin{equation}\label{enkappa}
\cE_\kappa \;:=\; \frac{E_k}{\left(N_1+\frac{1}{2}\right) \left(N_2+\frac{1}{2}\right)}\, .
\end{equation}}


We want to study the behaviour of small amplitude solutions of \eqref{2DETLeq}, with initial data in which only one low-frequency Fourier mode is excited. \\

We assume $N_1 \leq N_2$, and we introduce the quantities
\begin{align} 
\mu &:= \frac{2}{2N_1+1}, \label{mu} \\
\sigma &:= \, \log_{N_1+\frac{1}{2}} \left( N_2+\frac{1}{2} \right), \label{sigma}
\end{align}
which play the role of parameters in our construction: we will use them in the asymptotic expansion of the dispersion relation of the continuous approximation of the lattice (see \eqref{exDelta1sigma}-\eqref{ex2Delta1sigma}, and \eqref{Ham2KGc1DNLS}-\eqref{1DKGsc2}) in order to derive the integrable approximating PDEs in the regimes we are considering.

We study the $\alpha+\beta$ model of \eqref{2DETLeq} in the following regime:
\begin{itemize}
\item[(KP)] the \emph{weakly transverse regime}, where the effective dynamics is described by a system of two \emph{uncoupled} Kadomtsev-Petviashvili (KP) equation. This corresponds to taking $\mu \ll 1$ and $\sigma=2$.
\end{itemize}

From now on, we denote by $\kappa_0 := \left( \frac{1}{N_1+\frac{1}{2}}, \frac{1}{ (N_1+\frac{1}{2})^\sigma } \right) = (\mu,\mu^\sigma)$. Our main result is the following:

\begin{theorem} \label{KPrThm}
Consider \eqref{2DETLeq} with $\alpha \neq 0$, $\sigma = 2$.

Fix {$0< \gamma  < \frac{1}{2}$} and two positive constants $C_0$ and $T_0$, then there exist positive constants $\mu_0$, $C_1$ and $C_2$ (depending only on $\gamma$, $C_0$ and on $T_0$) such that the following holds. Consider an initial datum with
\begin{align} \label{KPData}
\cE_{\kappa_0}(0) \, = C_0\mu^{4}, &\qquad \cE_{\kappa}(0)=0 \qquad \forall \kappa=(\kappa_1,\kappa_2) \neq \kappa_0,
\end{align}
and assume that $\mu < \mu_0$. Then there exists $\rho>0$ such that along the corresponding solution one has
\begin{align} \label{EnModesKP}
\cE_\kappa(t) &\leq C_1 \, \mu^{4} e^{-\rho |(\kappa_1/\mu, \kappa_2/\mu^\sigma )| } + C_2 \, \mu^{4+\gamma }, \qquad {{\forall \, }} |t| \, \leq \, \frac{T_0}{\mu^{3} }
\end{align}
for all $\kappa$.
\end{theorem}

\begin{remark} \label{ActAngKPrem}
Theorem \ref{KPrThm} is the first rigorous result for two-dimensional lattices in which the dynamics of the lattice in a genuinely two-dimensional regime is described by a system of two-dimensional integrable PDEs. Moreover, in Theorem \ref{KPrThm} we do not mention the existence of a sequence of almost-periodic functions approximating the specific energies of the modes, and this is a difference with respect to Theorem 5.3 in \cite{bambusi2006metastability}. This is related to the construction of action-angle/Birkhoff coordinates for the KP equation, which is an open problem in the theory of integrable PDEs.
\end{remark}

\begin{remark} \label{InDatarem}
For the sake of simplicity, we have proved Theorem \ref{KPrThm} for initial data in which only one low-frequency Fourier mode is excited. One can also prove that a variant of Theorem \ref{KPrThm} holds also in the case the higher harmonics of a low-frequency Fourier mode are excited, provided that the energy decreases exponentially with respect to $|(\kappa_1/\mu, \kappa_2/\mu^\sigma )|$, and also for initial data in which the symmetrical modes of a given low-frequency Fourier mode are excited. To summarize, we are only able to prove stability of the solutions we constructed for initial data with vanishing specific energy for a time-scale $\cO(\mu^{-3})$. 
\end{remark}

{
\begin{remark}
	Here and in the following theorems we decided to write the bound on the specific energy with respect to the intensive variables $\kappa=(\kappa_1,\kappa_2)$. This writing has the advantage to be consistent with the previous literature (e.g. compare \eqref{KPData} above with (3.8) in \cite{bambusi2006metastability}). Anyway, let us emphasize that, using the definition of $\kappa$, \eqref{KPData} can be written equivalently as
\begin{equation}
\cE_\kappa(t) \leq C_1 \mu^4 e^{-\rho|(k_1,k_2)|}+C_2 \mu^{4+\gamma} \, , \qquad \forall t \leq \frac{T_0}{\mu^3} \, .
\end{equation} 
This implies that, as the number of sites increase, the Fourier spectrum is exponentially localised (and remains so, as $\mu \to 0$) around $\kappa=0$, apart from an error of order $\mu^{4+\gamma}$.
\end{remark}}

We also point out that there are also other regimes in which the dynamics of a two-dimensional lattice can be approximated by integrable PDEs. For example, we can consider $\alpha+\beta$ model of \eqref{2DETLeq} in the following regime:

\begin{itemize}
\item[(KdV)] the \emph{very weakly transverse regime}, where the effective dynamics is described by a system of two \emph{uncoupled} Korteweg-de Vries (KdV) equations. This corresponds to taking $\mu \ll 1$ and $2 < \sigma < {7}$.
\end{itemize}

The corresponding result one can prove in such a regime is the following.

\begin{theorem} \label{KdVrThm}
Consider \eqref{2DETLeq} with $\alpha \neq 0$, {$2< \sigma < 7$}. {Define $\gamma_0(\sigma) := \frac{3}{2} \left(\sigma-\frac{5}{3}\right)$ for $2<\sigma<3$, and $\gamma_0(\sigma):= \frac{1}{2} (7-\sigma)$ for $3 \leq \sigma < 7$.}

Fix {$0 < \gamma  < \gamma_0(\sigma)$} and two positive constants $C_0$ and $T_0$, then there exist positive constants $\mu_0$, $C_1$ and $C_2$ (depending only on $\gamma$, $C_0$, ${\sigma}$ and on $T_0$) such that the following holds. Consider an initial datum with
\begin{align} \label{KdVData}
\cE_{\kappa_0}(0) \, = C_0\mu^{4}, &\; \qquad \cE_{\kappa}(0)=0, \; \qquad \forall \kappa=(\kappa_1,\kappa_2) \neq \kappa_0,
\end{align}
and assume that $\mu < \mu_0$. Then there exists $\rho>0$ such that along the corresponding solution one has
\begin{align} \label{EnModesKdV}
\cE_\kappa(t) &\leq C_1 \, \mu^{4} e^{-\rho |(\kappa_1/\mu, \kappa_2/\mu^\sigma )| } + C_2 \, \mu^{4+\gamma}, \; \qquad {{\forall \, }} |t| \, \leq \, \frac{T_0}{\mu^{3} }
\end{align}
for all $\kappa$. Moreover, for any $n_2$ with $0 \leq n_2 \leq N_2$ there exists a sequence of almost-periodic functions $(F_n)_{n=(n_1,n_2) \in \Z^2_{N_1,N_2,+}}$ such that, if we define
\begin{align} \label{AlmostPerKdV}
\cF_{\kappa_0} \, = \, \mu^{4} \, F_{n}, &\; \qquad \cF_\kappa \, = \, 0 \; \qquad \forall \kappa \neq n\kappa_0
\end{align}
{then one has, for the \emph{specific energy distribution} $\cE_\kappa(t)$,}
\begin{align} \label{ApprEnModesKdV}
\left| \cE_\kappa(t) - \cF_\kappa(t) \right| &\leq C_2 \, \mu^{4+\gamma}, \; \qquad {{\forall \, }} |t| \, \leq \, \frac{T_0}{\mu^{ 3 }}.
\end{align}
\end{theorem}

\begin{remark} \label{Scalingrem}

We point out that in the statement of Theorem \ref {KdVrThm} the assumption $\sigma > 2$ comes from an asymptotic expansion of the dispersion relation of the continuous approximation of the lattice (see \eqref{exDelta1sigma}-\eqref{ex2Delta1sigma}), while the assumption {$\sigma < 7$} comes from a technical assumption under which we can approximate the dynamics of the lattice with the dynamics of the system of uncoupled KdV equations (see the statement of Theorem \ref{ApprPropKdV}).
\end{remark}

We can also consider two-dimensional KG lattices, which combine the nearest-neighbour potential with an on-site one: the scalar model is described by
\begin{align}
H_{\color{black}\mathrm{KG}}(Q,P) &= \sum_{j \in \Z^2_{N_1,N_2}} \frac{P_j^2}{2} + \frac{1}{2} \sum_{j \in \Z^2_{N_1,N_2}} Q_j \; (-\Delta_1 Q)_j + \sum_{j \in \Z^2_{N_1,N_2}} U(Q_j),, \label{Ham2KGs} \\
U(x) &= m^2 \frac{x^2}{2} + \beta \frac{x^{2p+2}}{2p+2}, \qquad m>0, \qquad \beta>0, \; p \geq 1, \label{potKG2}
\end{align}
(see \cite{rosenau2003hamiltonian} for a physical interpration of the model).
The associated equations of motion are
\begin{align}
\ddot{Q_j} &= (\Delta_1 Q)_j - m^2 Q_j - \beta Q_j^{2p+1}, \qquad j \in \Z^2_{N_1,N_2}. \label{2DKGseq} 
\end{align}
If we take $p=1$, we obtain a generalization of the one-dimensional $\phi^4$ model. \\

{We} now introduce the Fourier coefficients of $Q$ as in \eqref{fourierQ}, 
and similarly for $P$, and denote by
\begin{align}
E_k &:= \frac{|\hat{P}_k|^2+\omega_k^2 |\hat{Q}_k|^2}{2}, \label{EnNormModeKG} \\
\omega_k^2 &:=  m^2+4 \sin^2\left(\frac{k_1 \, \pi}{2N_1+1} \right) + 4 \sin^2\left(\frac{k_2 \, \pi}{2N_2+1} \right), \label{FreqNormModeKG}
\end{align}
the energy and the square of the frequency of the mode at site $k=(k_1,k_2) \in \Z^2_{N_1,N_2}$ (see Figure \ref{fig:freq_KG}).  \\
In the rest of the paper we will assume that $m=1$. \\

\begin{figure} 
\begin{center}
\includegraphics[width=0.5\textwidth]{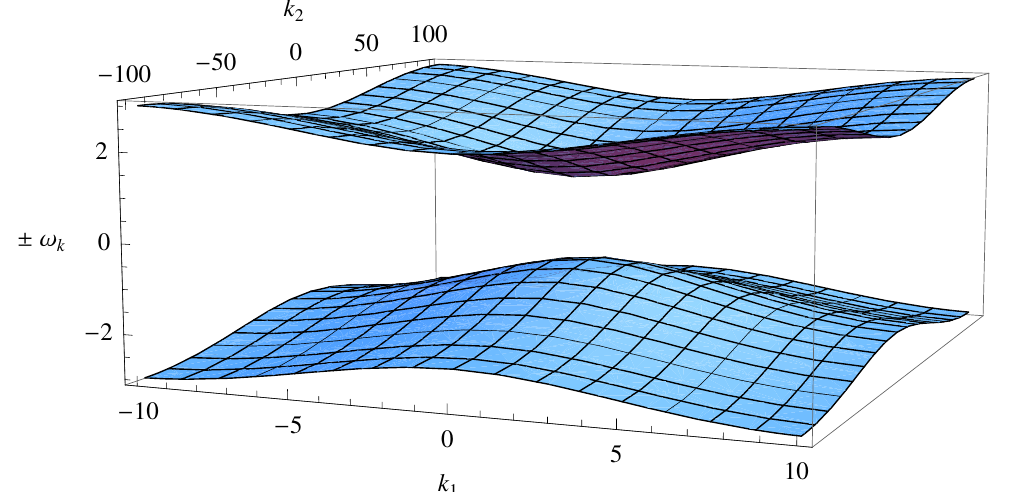}
\end{center}
\caption{The dispersion relation for the KG lattice \eqref{Ham2KGs}, namely $\pm \; \omega_k$, $k=(k_1,k_2)$, vs the integer coordinates, for $N_1=10$ and $N_2=100$.}
\label{fig:freq_KG}
\end{figure}

We consider the two-dimensional KG lattice \eqref{Ham2KGs} in the following regime:
\begin{itemize}
\item[(1D NLS)] the \emph{very weakly transverse regime}, where the effective dynamics is described by a cubic one-dimensional nonlinear Schr\"odinger (NLS) equation. This corresponds to taking $\mu \ll 1$ and $1 < \sigma < 7$.
\end{itemize}

\begin{theorem} \label{1DNLSrThm}
Consider \eqref{Ham2KGs} with $\beta > 0$,  $1 < \sigma < 7$. {Define $\gamma_1(\sigma) := \frac{3}{2} \left(\sigma-\frac{1}{3}\right)$ for $1<\sigma<2$, and $\gamma_1(\sigma):= \frac{1}{2} (7-\sigma)$ for $2 \leq \sigma < 7$.}

Fix $0 < \gamma  \leq {\gamma_1(\sigma)}$ and two positive constants $C_0$ and $T_0$, then there exist positive constants $\mu_0$, $C_1$ and $C_2$ (depending only on $\gamma$, $C_0$, {$\sigma$} and on $T_0$) such that the following holds. Consider an initial datum with
\begin{align} \label{1DNLSData}
\cE_{\kappa_0}(0) \, = C_0\mu^{2}, &\qquad \cE_{\kappa}(0)=0, \, \, \forall \kappa=(\kappa_1,\kappa_2) \neq \kappa_0,
\end{align}
and assume that $\mu < \mu_0$. Then there exists $\rho>0$ such that along the corresponding solution one has
\begin{align} \label{EnModes1DNLS}
\cE_\kappa(t) &\leq C_1 \, \mu^{2} e^{-\rho |(\kappa_1/\mu, \kappa_2/\mu^\sigma )| } + C_2 \, \mu^{2+\gamma}, \qquad |t| \, \leq \, \frac{T_0}{\mu^{2} }
\end{align}
for all $\kappa$. Moreover, for any $n_2$ with $0 \leq n_2 \leq N_2$ there exists a sequence of almost-periodic functions $(F_n)_{n=(n_1,n_2) \in \Z^2_{N_1,N_2,+}}$ such that, if we denote
\begin{align} \label{AlmostPer1DNLS}
\cF_{\kappa_0} \, = \, \mu^{2} \, F_{n}, &\qquad \cF_\kappa \, = \, 0 \qquad \forall \kappa \neq n\kappa_0
\end{align}
then {we have for the \emph{specific energy distribution}}
\begin{align} \label{ApprEnModes1DNLS}
\left| \cE_\kappa(t) - \cF_\kappa(t) \right| &\leq C_2 \, \mu^{2+\gamma}, \qquad |t| \, \leq \, \frac{T_0}{\mu^{ 2 }}.
\end{align}
\end{theorem}

\begin{remark} \label{InDataKGrem}
In Theorem \ref{1DNLSrThm} we are able to prove stability of the solutions we constructed for initial data with vanishing specific energy for a time-scale $\cO(\mu^{-2})$.
\end{remark}

\begin{remark} \label{ScalingKGrem}

As for Theorem \ref{KdVrThm}, in the statement of Theorem \ref{1DNLSrThm} the assumption $\sigma > 1$ comes from an asymptotic expansion of the dispersion relation of the continuous approximation of the lattice (see \eqref{Ham2KGc1DNLS}-\eqref{1DKGsc2}), while the assumption $\sigma < 7$ comes from a technical assumption under which we can approximate the dynamics of the lattice with the dynamics of the system of uncoupled NLS equations (see the statement of Theorem \ref{ApprProp1DNLS}).
\end{remark}

\subsection{Further remarks}

\begin{remark} \label{Anisorem}
The specific choice of the direction of longitudinal propagation in the regimes that we have considered is not relevant.
\end{remark}

\begin{remark} \label{FPUNLKGrem}
We point out that the time of validity of Theorem \ref{1DNLSrThm} for the KG lattice is of order $\cO(\mu^{-2})$, which is different from the time of validity of Theorem \ref{KdVrThm} and Theorem \ref{KPrThm} for the FPU lattice. In the one-dimensional case it has been observed that, for a fixed value of specific energy $\epsilon$ and for long-wavelength modes initially excited, the $\phi^4$ model reached equipartition faster than the FPU $\beta$ model (see \cite{lichtenberg2007dynamics}, sec. 2.1.8). 
\end{remark}

\begin{remark} \label{HigherDimRem}
Theorem \ref{KPrThm}, Theorem \ref{KdVrThm} and Theorem \ref{1DNLSrThm} can be generalized to higher dimensional lattices. Indeed, let $d \geq 2$, define
\begin{align} \label{Zd}
\Z^d_{N_1,\ldots,N_d} &:= \{ (j_1,\ldots,j_d): j_1,\ldots,j_d \in \Z, |j_1| \leq N_1,\ldots, |j_d| \leq N_d \},
\end{align}
and consider the $d$-dimensional ETL
\begin{align} 
H_{\color{black}\mathrm{ETL}}^{\color{black}(d)}(Q,P) &= \sum_{j \in \Z^d_{N_1,\ldots,N_d}}\left( - \frac{1}{2} P_j \, (\Delta_1P)_j + ({\color{black}V}(Q))_j\right), \; \; \; ({\color{black}V}(Q))_j = \frac{Q_j^2}{2} + \alpha \frac{Q_j^3}{3} + \beta \frac{Q_j^4}{4}, \; \; j \in \Z^d_{N_1,\ldots,N_d}, \label{dDHam} 
\end{align}
and the $d$-dimensional NLKG lattice
\begin{align}
H_{\color{black}\mathrm{KG}}^{\color{black}(d)}(Q,P) &= \sum_{j \in \Z^d_{N_1,\ldots,N_d}} \frac{P_j^2}{2} + \frac{1}{2} \sum_{\substack{j,k \in \Z^d_{N_1,\ldots,N_d} \\ |j-k|=1}} \frac{(Q_j-Q_k)^2}{2} + \sum_{j \in \Z^d_{N_1,\ldots,N_d}} U(Q_j), \; \; \; U(x) = \frac{x^2}{2} + \beta \frac{x^{4}}{4}, \;\; \beta>0. \label{HamdKGs} 
\end{align}

We assume $N_1 \leq N_2,\ldots,N_d$, and we introduce the quantities
\begin{align}
\mu := \frac{2}{2N_1 +1}, \; &\; \; \sigma_i := \log_{N_1 + \frac{1}{2} } \left( N_{i+1} + \frac{1}{2} \right), \qquad i=1,\ldots,d-1. \label{dmudsigma}
\end{align}
Then we can describe the following regimes:
\begin{itemize}
\item[(KdV-d)] the model \eqref{dDHam} for $d \leq 4$, in the very weakly transverse regime with $\mu \ll 1$ and $2 < \sigma_1,\ldots,\sigma_{d-1} < {7}$;

\item[(KP-d)] the  model \eqref{dDHam} for $d \leq 4$, in the weakly transverse regime with $\mu \ll 1$ and $\sigma_1=2$, $2 < \sigma_2,\ldots,\sigma_{d-1} < {7}$;

\item[(1DNLS-d)] the model \eqref{HamdKGs} for $d \leq 6$, in the very weakly transverse regime, with $\mu \ll 1$, $1 < \sigma_1,\ldots,\sigma_{d-1} < 7$.
\end{itemize}

\end{remark}

\begin{remark}
There are other interesting regimes for \eqref{2DETLeq} and \eqref{2DKGseq} especially for their relation with the modified KdV equation and two-dimensional Non-Linear Schr\"odinger equation respectively. These will be discussed in Remark \ref{mKdVrem} and Remark \ref{2DNLSrem} respectively.
\end{remark}

\section{Galerkin Averaging} \label{galavsec}

{In the next Section we will show that for large $N$ the dynamics of both ETL and KG lattices can be approximated by an infinite dimensional Hamiltonian, which can be written as the sum of an integrable part and a non-integrable perturbation. For this kind of systems it is often possible to analyse the dynamics taking into account only the leading terms of the integrable part and the `average effect' of the perturbation, which describe the relevant qualitative behaviour of the system for a sufficiently long time-scale.}

{In this Section we prove an abstract averaging theorem whose assumptions are satisfied by the systems we have introduced in Section \ref{results}. This is the crucial technical result that allows to rigorously approximate the infinite dimensional Hamiltonian system with the leading terms of the integrable part and the average of the perturbation, up to the time-scales we are interested in. Since the average has to be computed along the solutions of the unperturbed system (see \eqref{average} below), the vector field of the averaged perturbation commutes with the vector field of the unperturbed system, thus resulting in a system in normal form. }

{ The idea of its proof (following \cite{bambusi2005galerkin}, \cite{bambusi2006metastability} and \cite{pasquali2018dynamics}) is to make a Galerkin cutoff, namely to approximate the original infinite dimensional system by a finite dimensional one, to put in normal form the cutoffed system, and then to choose the dimension of the cutoffed system in such a way that the error due to the Galerkin cutoff and the error due to the truncation in the normalization procedure are of the same order of magnitude. The system one gets is composed by a part which is in normal form, and by a remainder which is a smaller singular perturbation. }

If we neglect the remainder, we obtain a system whose solutions are approximate solutions of the original system. In Sec. \ref{ApprSec} we will show how to control the error with respect to a true solution of the original system.

{This Section is divided in two parts. In the first part we introduce the analytic setting we are working with. This includes the definition functional setting of the problem and the average Theorem \ref{gavthm}. In the second part we give a concise proof of the average Theorem, deferring the proof of the technical Lemma \ref{NFest} to Appendix \ref{BNFest}.}

\subsection{An Averaging Theorem} \label{BNFsubsec}

{To define the function spaces we are working with, we introduce} a topology in the phase space. This is conveniently done in terms of Fourier coefficients.

\begin{definition} \label{defsigk}
Fix two constants $\rho \geq 0$ and $s \geq 0$. We will denote by $\ell^2_{\rho,s}$
the Hilbert space of complex sequences $v = (v_n)_{n \in \Z^2 \setminus\{0\}}$ with obvious vector space structure and with scalar product
\begin{align}
\langle v, w \rangle_{\rho,s} \; := \; \sum_{n \in \Z^2 \setminus\{0\}} \overline {v_n} w_n e^{2 \rho |n|} |n|^{2 s} \, .
\end{align} 
and such that
\begin{align} \label{normsigk}
\|v\|_{\rho,s}^2 &:= \langle v,v \rangle_{\rho,s} = \sum_{n \in \Z^2 \setminus\{0\}} |v_n|^2 e^{2\rho |n|} |n|^{2s}
\end{align}
is finite. We will denote by $\ell^2$ the space $\ell^2_{0,0}$.
\end{definition}

We will identify a 2-periodic function {\color{black}$f$} with the sequence of its Fourier coefficients {$(\hat{f}_n)_{n \in \mathbb{Z}^2 \setminus\{0\}}$},
\begin{align*}
{\color{black}f(y)} &= \frac{1}{2} \sum_{n \in \Z^2} {\color{black}\hat{f}_n} e^{\ii \pi \, n \cdot y},
\end{align*}
and, with a small abuse of notation, we will say that ${f} \in \ell^2_{\rho,s}$ if the sequence of its Fourier coefficients  belong to $\ell^2_{\rho,s}$.

Now fix $\rho \geq 0$ and $s \geq 1$, and consider the scale of Hilbert spaces 
$\cH^{\rho,s}:= \ell^2_{\rho,s} \times \ell^2_{\rho,s} \ni \zeta=(\xi,\eta)$, 
endowed with one of the following symplectic forms:
\begin{align}\label{eq:SymplecticForms}
\Omega_1 := \begin{pmatrix}
0 & \ii \\ -\ii & 0
\end{pmatrix}, \qquad \Omega_2 := \begin{pmatrix}
-\partial_{y_1}^{-1} & 0 \\
0 & \partial_{y_1}^{-1}
\end{pmatrix} \, .
\end{align}
Observe that $\Omega_{{\Gamma}}: \cH^{\rho,s} \to \cH^{\rho,s+{{\Gamma}}-1}$ (${{\Gamma}}=1,2$) is a well-defined operator. Moreover, $\Omega_2$ is well-defined on the space of functions with zero-average with respect to the $y_1$-variable, i.e. on those functions $\zeta(y_1,y_2)$ such that for every $y_2$ we have $\int_{-1}^1 \zeta(y_1,y_2) \, \di y_1 =0$. 

If we fix $ {{\Gamma}} \in \{1,2\}$, $s$ and $U_{s} \subset \ell^2_{\rho,s}$ open, 
we define the gradient of $\Ham{K} \in C^\infty(U_{s},\R)$ with respect 
to $\xi \in \ell^2_{\rho,s}$ as the unique function s.t.
\begin{align*}
\la \nabla_\xi \Ham K , h \ra &= \di_{\xi} \Ham K  (h), \qquad \forall h \in \ell^2_{\rho,s}.
\end{align*}
Similarly, for an open set $\cU_{s} \subset \cH^{\rho,s}$  the Hamiltonian vector field of the Hamiltonian function $\Ham H \in C^\infty(\cU_{s},\R)$ is given by 
\[ X_{\Ham H}(\zeta)=\Omega^{-1}_{{{\Gamma}}} \nabla_{\zeta} \Ham H(\zeta). \]
The open ball of radius $R$ and center $0$ in $\ell^2_{\rho,s}$ will be denoted by 
$B_{\rho,s}(R)$; we write $\cB_{\rho,s}(R):=B_{\rho,s}(R) \times B_{\rho,s}(R) \subset \cH^{\rho,s}$. \\

Now, we introduce the Fourier projection operators $\hat{\pi}_j: \ell^2_{\rho,s} \to \ell^2_{\rho,s}$
\begin{align} \label{hatpi}
\hat{\pi}_j((v_n)_{ n\in\Z^2 \setminus\{0\} }) &:= 
\begin{cases}
v_n & \text{if} \qquad j-1 \leq |n| < j \\
0 & \text{otherwise}
\end{cases} \qquad , \qquad j \geq 1,
\end{align}
the operators $\pi_j: \cH^{\rho,s} \to \cH^{\rho,s}$
\begin{align} \label{smallpi}
\pi_j((\zeta_n)_{ n\in\Z^2 \setminus\{0\} }) &:= 
\begin{cases}
\zeta_n & \text{if} \qquad j-1 \leq |n| < j \\
0 & \text{otherwise}
\end{cases} \qquad , \qquad j \geq 1,
\end{align}
and the operators $\Pi_M: \cH^{\rho,s} \to \cH^{\rho,s}$
\begin{align} \label{bigpi}
\Pi_M((\zeta_n)_{ n\in\Z^2 \setminus\{0\} })  &:= 
\begin{cases}
\zeta_n & \text{if} \qquad |n| \leq M \\
0 & \text{otherwise}
\end{cases} \qquad , \qquad M \geq 0.
\end{align}
{Last, we define the operator $\overline{\pi_j}:=\mathrm{id}-\pi_j$ that will be used in Appendix \ref{ApprEstSec12} and \ref{ApprEstSec22}.}

\begin{lemma}
The projection operators defined in \eqref{smallpi} and \eqref{bigpi} satisfy 
the following properties for any $\zeta \in \mathcal{H}^{\rho,s}$:
\begin{enumerate}
\item[$i.$] for any $j \geq 0$
\begin{align*}
\zeta = \sum_{j \geq 0} \pi_j\zeta;
\end{align*}
\item[$ii.$] for any $j \geq 0$ 
\begin{align*}
\|\Pi_M\zeta\|_{\cH^{\rho,s}}&\leq  \, \|\zeta\|_{\cH^{\rho,s}} ;
\end{align*}
\item[$iii.$] the following equality holds 
\begin{align} \label{compnorms}
\| \zeta \|_{\cH^{\rho,s}} = \left\| \left[ \sum_{j \in \N} j^{2s} |\pi_j\zeta|^2 \right]^{1/2} \right\|_{\cH^{\rho,0}} 
\end{align}
where $| \zeta|$, for $\zeta \in \cH^{\rho,s}$ is the element $|\zeta| \in \cH^{\rho,s}$ whose $n$-th element is
\begin{align*}
|\zeta|_n &:= (|\xi_n|,|\eta_n|) 
\end{align*}
and $(\zeta^\alpha)_n:=(\xi_n^\alpha,\eta_n^\alpha)$.
\end{enumerate}
\end{lemma}

Now we consider a Hamiltonian system of the form
\begin{align} \label{Hamdecomp}
\Ham H &= \Ham h_0 + \delta \Ham F, 
\end{align}
where we assume that
\begin{itemize}
\item[(PER)]  $\Ham{h}_0$ generates a linear periodic flow $\Phi^\tau_{\Ham{h}_0}$ with period $T$, 
\begin{align*}
\Phi^{\tau+T}_{\Ham{h}_0} &= \Phi^\tau_{\Ham{h}_0} \qquad \forall \tau,
\end{align*}
which is analytic as a map from $\cH^{\rho,s}$ into itself for any $s \geq 1$. Furthermore, the flow is an isometry for any $s \geq 1 $. 

\item[(INV)] for any $s \geq 1$, 
$\Phi^\tau_{\Ham{h}_0}$ leaves invariant the space $\Pi_j\cH^{\rho,s}$ for any $j\geq0$. 
Furthermore, for any $j \geq 0$ 
\[ \pi_j \circ \Phi^\tau_{\Ham{h}_0} = \Phi^\tau_{\Ham{h}_0} \circ \pi_j. \]
\end{itemize}

Next, we assume that {there exists $\nu \in ]0,1]$ such that} the vector field of $\Ham F$ admits an asymptotic expansion in $\delta$ of the form
\begin{align}
\Ham F &\sim \sum_{j \geq 1} \delta^{{\nu}(j-1)} \Ham F_j, \label{expF} \\
X_{\Ham F} &\sim \sum_{j \geq 1} \delta^{{\nu}(j-1)} X_{\Ham F_j}, \label{expXF}
\end{align}
and that the following property is satisfied
\begin{itemize}

\item[(HVF)] There exists $R^\ast>0$ such that for any $j \geq 1$ 
\begin{itemize}
\item[$\cdot$] $X_{\Ham F_j}$ is analytic from $\cB_{\rho,s+2j+{\Gamma}}(R^\ast)$ to $\cH^{\rho,s}$.
\end{itemize}
Moreover, for any $r \geq 1$ we have that 
\begin{itemize}
\item[$\cdot$] $X_{\Ham F - \sum_{j=1}^r \delta^{{\nu}(j-1)}\Ham F_j}$ is analytic from $\cB_{\rho,s+2(r+1)+{\Gamma}}(R^\ast)$ to $\cH^{\rho,s}$.
\end{itemize}

\end{itemize}

The main result of this Section is the following theorem.

\begin{theorem} \label{gavthm}
Fix $R>0$, $s_1\gg 1$. 
Consider \eqref{Hamdecomp}, and assume (PER), (INV) and (HVF).
Then $\exists$ $s_0>0$ with the following properties: 
for any $s \geq s_1$ there exists $\delta_{s} \ll 1$ such that 
for any $\delta<\delta_{s}$ there exists 
$\cT_\delta:\cB_{\rho,s}(R/2) \to \cB_{\rho,s}(R)$ analytic canonical transformation 
such that
\begin{align} \label{TransfHam}
\Ham H_1 := \Ham H \circ \cT_\delta &= \Ham{h}_0 + \delta \cZ_1 + \delta^{{1+\nu}} \; \mathcal{R}^{(1)},
\end{align}
where $\cZ_1$ is in normal form, namely
\begin{align} \label{NFthm}
\{\cZ_1,\Ham{h}_0\} &= 0,
\end{align} 
and there exists a positive constant $C'_{s}$ {(that depends on $s$)} such that
\begin{align*} 
\sup_{\cB_{\rho,s+s_0}(R/2)} \|X_{\cZ_{1}}\|_{\cH^{\rho,s}} &\leq C'_{s},
\end{align*} 
\begin{align} \label{Remthm}
\sup_{\cB_{\rho,s+s_0}(R/2)} \|X_{\mathcal{R}^{(1)}}\|_{\cH^{\rho,s}} &\leq C'_{s},
\end{align} 
\begin{align} \label{CTthm}
\sup_{\cB_{\rho,s}(R/2)} \|\cT_\delta-\mathrm{id}\|_{\cH^{\rho,s}} &\leq C'_{s} \, \delta.
\end{align}
In particular,
\begin{align} \label{average}
\cZ_1(\zeta) = \la \Ham F_1 \ra(\zeta),
\end{align}
where $\la \Ham F_1 \ra(\zeta) := \int_0^{T} \Ham F_1\circ\Phi^\tau_{\Ham{h}_0}(\zeta) \frac{\di\tau}{T}$.
\end{theorem}

\subsection{Proof of the Averaging Theorem} \label{BNFprsubsec}

The proof of Theorem \ref{gavthm} is actually an application of the techniques used in \cite{pasquali2018dynamics} and \cite{bambusi2006metastability}).

First notice that by assumption (INV) the Hamiltonian vector field of $\Ham{h}_0$ 
generates a continuous flow $\Phi^\tau_{{\Ham h_0}}$ which leaves $\Pi_M\cH^{\rho,s}$ invariant. \\

Now we set $\Ham H = \Ham H_{1,M} + \cR_{1,M} + \cR_1$, where \\
\begin{align} \label{truncsys}
\Ham H_{1,M} &:= \Ham h_{0} + \delta \, \Ham F_{1,M}, \\ 
\Ham F_{1,M} &:= \Ham F_1 \circ \Pi_M,
\end{align}
and
\begin{align} \label{remsys}
\cR_{1,M} &:= \Ham{h}_0 + \delta \Ham F_1 - \Ham H_{1,M}, \\
\cR_1 &:= \delta \left( \Ham F - \Ham F_1 \right).
\end{align}

The system described by the Hamiltonian \eqref{truncsys} is the one that 
we will put in normal form.  \\
In the following we will use the notation $a \sleq b$ to mean: 
there exists a positive constant $K$ independent of $M$ and $R$ 
(but eventually on $s$), such that $a \leq Kb$. 
We exploit the following intermediate results:

\begin{lemma} \label{truncest}
For any $s \geq s_1$ there exists $R>0$ such that $\forall$ $\sigma >0$, $M>0$ 
\begin{align} \label{truncremt}
\sup_{\zeta \in  \cB_{\rho, s+ {\Gamma} +\sigma+2 }(R) } \|X_{\cR_{1,M}}(\zeta)\|_{\cH^{\rho,s}} &\sleq \; \frac{\delta}{(M+1)^{ \sigma }}, 
\end{align}
\begin{align} \label{expremest} 
\sup_{\zeta \in  \cB_{\rho, s+ {\Gamma} +4 }(R) } \|X_{\cR_1}(\zeta)\|_{\cH^{\rho,s}} &\sleq \delta^{ {1+\nu} }.
\end{align}
\end{lemma}

\begin{proof}
We recall that $\cR_{1,M} = \Ham{h}_0 + \delta \Ham F_1 - \Ham H_{1,M} { =\delta (F_1-F_{1,M}) }$. \\
We first notice that  $\|\mathrm{id}-\Pi_M\|_{ \cH^{\rho,s+\sigma} \to \cH^{\rho,s} } = (M+1)^{-\sigma}$: indeed, using \eqref{compnorms} we obtain
\begin{align*}
\left\| \sum_{j \geq M+1} \pi_jf \right\|_{\cH^{\rho,s}} &= \left\| \left[ \sum_{j \geq M+1} |j^{s} \pi_jf|^2 \right]^{1/2} \right\|_{\cH^{\rho,0}} \leq (M+1)^{-\sigma} \left\| \left[ \sum_{j \geq M+1} |j^{s+\sigma} \pi_jf|^2 \right]^{1/2} \right\|_{\cH^{\rho,0}} \\
&\leq (M+1)^{-\sigma} \|f\|_{\cH^{\rho,s+\sigma}},
\end{align*}
whereas the inequality $\Vert \mathrm{id} - \Pi_M \Vert_{\cH^{\rho,s+\sigma} \to \cH^{\rho,\sigma}} \leq (M+1)^{-\sigma}$ is obtained with a function which has non zero components only for $|j|=M+1$, i.e. $f=\pi_{M+1} f$. 

Inequality \eqref{truncremt} follows from 
\begin{align*}
\sup_{\zeta \in \cB_{\rho,s+{\Gamma}+2+\sigma}(R)} \; \| X_{\cR_{1,M}}(\zeta)\|_{\cH^{\rho,s}} & \sleq \; \|\mathrm{d} X_{ \delta \Ham F_1 }\|_{ L^\infty(\cB_{\rho,s+2+{\Gamma}}(R),\cH^{\rho,s}) } \|\mathrm{id}-\Pi_M\|_{ L^\infty(\cB_{\rho,s+2+{\Gamma}+\sigma}(R),\cB_{\rho,s+2+{\Gamma}}(R) ) } \\
&\sleq \delta \, (M+1)^{-\sigma},
\end{align*}
while estimate \eqref{expremest} is an immediate consequence of (HVF).
\end{proof}

\begin{lemma} \label{pertestlemma}
For any $s \geq s_1$ 
\begin{align*}
\sup_{\zeta \in  \cB_{\rho,s}(R^*) } \|X_{\Ham F_{1,M}}(\zeta)\|_{\cH^{\rho,s}} &\leq K^{(\Ham F)}_{1,s} M^{2+{\Gamma}} , 
\end{align*}
where 
\begin{align*}
K^{(\Ham F)}_{1,s} &:= \sup_{\zeta \in \cB_{\rho,s}(R^*) } \|X_{\Ham F_1}(\zeta)\|_{\cH^{\rho,s-2-{\Gamma}}} \; < \; + \infty.
\end{align*}
\end{lemma}

\begin{proof}
Using \eqref{compnorms} we have
\begin{align}
\sup_{\zeta \in \cB_{\rho,s}(R)} &\left\| \sum_{h \leq M} \pi_h X_{\Ham F_{1,M}} (\zeta) \right\|_{\cH^{\rho,s}} = 
\sup_{\zeta \in \cB_{\rho,s}(R)} \left\| \left[ \sum_{h \leq M} |h^{s} \pi_h X_{\Ham F_{1,M}}(\zeta)|^2 \right]^{1/2} \right\|_{\cH^{\rho,0}} \\
&\leq M^{2+{\Gamma}} \sup_{\zeta \in \cB_{\rho,s}(R)} \left\| \left[ \sum_{h \leq M} |h^{s-2-{\Gamma}} \pi_h X_{\Ham F_{1,M}}(\zeta)|^2 \right]^{1/2} \right\|_{\cH^{\rho,0}} \\
&\leq M^{2+{\Gamma}} \sup_{\zeta \in \cB_{\rho,s}(R)} \|X_{\Ham F_{1,M}}(\zeta)\|_{\cH^{\rho,s-2-{\Gamma}}} = K^{(\Ham F)}_{1,s} \, M^{2+{\Gamma}},
\end{align}
where the last quantity is finite for $R \leq R^\ast$ by property (HVF).
\end{proof}

To normalize \eqref{truncsys} we need {to prove a reformulation}
of Theorem 4.4 in \cite{bambusi1999nekhoroshev}. Here we report a statement of the result adapted to our context which is proved in Appendix  \ref{BNFest}.

\begin{lemma} \label{NFest}
Let $s \geq s_1+2+{{\Gamma}}$, $R>0$, and consider the 
system \eqref{truncsys}. Assume that $\delta < \frac{1}{30^{ {1/\nu} } }$, and that 
\begin{align} \label{smallcond}
12 \, T \, K_{1,s}^{(\Ham F)} M^{2+{{\Gamma}}} \delta <  R
\end{align}
where 
\begin{align*}
K^{(\Ham F)}_{1,s}&:=\sup_{\zeta \in \cB_{\rho,s}(R)} \|X_{\Ham F_1}(\zeta)\|_{\cH^{\rho,s-2-{{\Gamma}}}}.
\end{align*}
Then there exists an analytic canonical transformation 
$\cT^{(0)}_{\delta,M}: \cB_{\rho,s}(R/2) \to \cB_{\rho,s}(R)$ such that
\begin{align} \label{CTlemma}
\sup_{\zeta \in \cB_{\rho,s}(R/2)} \|\cT^{(0)}_{\delta,M}(\zeta)-\zeta\|_{\cH^{\rho,s}} &\leq 2T \, K^{(F)}_{1,s} M^{2+{{\Gamma}}} \delta,
\end{align}
and that puts \eqref{truncsys} in normal form up to a small remainder, 
\begin{align} \label{stepr}
\Ham H_{1,M} \circ \cT^{(0)}_{\delta,M} &= \Ham h_{0} + \delta \Ham Z^{(1)}_M + \delta^{ {1+\nu} } \cR^{(1)}_M, 
\end{align}
with $Z^{(1)}_M$ in normal form, namely $\{h_{0,M},Z^{(1)}_M\}=0$, and
\begin{align}
\sup_{\zeta \in \cB_{\rho,s}(R/2)} \|X_{ \Ham Z^{(1)}_M }(\zeta)\|_{\cH^{\rho,s}} &\leq 
K_{1,s}^{(\Ham F)} M^{2+{{\Gamma}}}
\label{itervf1}
\end{align} 
\begin{align}
&\sup_{\zeta \in \cB_{\rho,s}(R/2)} \|X_{ \cR^{(1)}_M }(\zeta)\|_{\cH^{\rho,s}} \leq 15 K_{1,s}^{(\Ham F)} M^{2+{{\Gamma}}}  \label{vecfrem}
\end{align}
\end{lemma}

Now we conclude with the proof of  Theorem \ref{gavthm}. 

\begin{proof}
If we define $\delta_s := \min\{\frac{1}{ 30^{{1/\nu}} },\frac{R}{12 \, T \, K_{1,s}^{(\Ham F)} M^{2+{{\Gamma}}}}\}$ and we choose
\begin{align*}
s_0 &= \sigma+2+{{\Gamma}}, \\
\sigma &\geq 2,
\end{align*}
then the transformation $\cT_\delta:=\cT^{(0)}_{\delta,M}$ defined by 
Lemma \ref{NFest} satisfies \eqref{TransfHam} because of \eqref{stepr}. 

Next, Eq. \eqref{NFthm} follows from Lemma \ref{NFest}, Eq. \eqref{Remthm} follows from \eqref{itervf1} and \eqref{vecfrem}, while \eqref{CTthm} is precisely \eqref{CTlemma}.
Finally, \eqref{average} can be deduced by applying Lemma \ref{homeqlemma} to $\Ham G=\Ham F_1$.
\end{proof}

\section{Applications to two-dimensional lattices} \label{2Dsec}

\subsection{The KP regime for the ETL lattice} \label{KPsubsec}
We want to study the behaviour of small amplitude solutions of \eqref{2DETLeq} with initial data in which only one low-frequency Fourier mode is excited. \\

As a first step, we introduce an interpolating function $Q=Q(t,x)$ such that 
\begin{itemize}
\item[(A1)] $Q(t,j)=Q_j(t)$, for all $j \in \Z^2_{N_1,N_2}$;
\item[(A2)] $Q$ is periodic with period $2N_1+1$ in the $x_1$-variable, and periodic with period $2N_2+1$ in the $x_2$-variable;
\item[(A3)] $Q$ has zero average, $\int_{ \left[-\left(N_1+\frac{1}{2}\right),N_1+\frac{1}{2}\right] \times \left[-\left(N_2+\frac{1}{2}\right),N_2+\frac{1}{2}\right] } Q(t,x) \, \di x =0$ $\forall t$;
\item[(A4)] $Q$ fulfills 
\begin{align}
\ddot{Q} &= \Delta_1(Q+\alpha Q^2+\beta Q^3), \label{FPUeqc} \\
\Delta_1 &:= 4 \sinh^2\left(\frac{\d_{x_1}}{2}\right) + 4 \sinh^2\left(\frac{\d_{x_2}}{2}\right). \label{Delta1c}
\end{align}
{ One can check that the operator \eqref{Delta1c} is equivalent to \eqref{Delta1} by using functional calculus.}
\end{itemize}

It is easy to verify that \eqref{FPUeqc} is Hamiltonian with Hamiltonian function
\begin{align} \label{HamFPUc}
\Ham H_{\mathrm{ETL}} (Q,P) &= \int_{ \left[-\frac{1}{\mu},\frac{1}{\mu}\right] \times \left[-\frac{1}{\mu^2},\frac{1}{\mu^2}\right]  } \left(\frac{-P \, \Delta_1 P + Q^2}{2} + \alpha \frac{Q^3}{3} + \beta \frac{Q^4}{4}\right) \, \di x,
\end{align}
where $P$ is a periodic function which has zero average and is canonically conjugated to $Q$.

{The existence of such an interpolating function is obvious, indeed such function can be taken to be a Fourier polynomial with $N$ terms. However, by construction, items (A1)-(A4) do not provide a \emph{unique} interpolating function. Among all possible choices of interpolating functions it is convenient to choose the Fourier polynomial supported on $k \in [-N_1,N_1]\times[-N_2,N_2]$ (i.e. the analytic function supported in the lowest possible number of Fourier modes). Also let us emphasize that (A4) ensures that if (A1) is valid for the intial time $t=0$, then is valid for all $t \in \mathbb{R}$.}

We consider \eqref{FPUeqc}, with $\alpha \neq 0$, 
and we look for small amplitude solutions of the form 
\begin{align} \label{KPr1}
Q(t,x) &= \mu^2 q(\mu t, \mu x_1, \mu^2 x_2),
\end{align}
with $\mu$ as in \eqref{mu}. { We introduce the rescaled time $\tau = \mu t$ and the rescaled space variables $y_1=\mu x_1$, $y_2=\mu^2 x_2$.}

Plugging \eqref{KPr1} into \eqref{FPUeqc}, leads to
\begin{align}
q_{\tau \tau} &= \frac{\Delta_{\mu,y_1}}{\mu^2} \left( q + \mu^2 \alpha q^2 {+ \mu^4 \beta q^3} \right), \label{FPUeqKPr1} \\
\Delta_{\mu,y_1} &:= 4 \sinh^2\left(\frac{\mu \d_{y_1}}{2}\right) + 4 \sinh^2\left(\mu^{2} \frac{\d_{y_2}}{2}\right), \label{Delta1mu}
\end{align}
which is a Hamiltonian PDE  corresponding to the Hamiltonian functional,
\begin{align}
\Ham K_1(q,p) &= \int_{I} \left(\frac{-p \, \Delta_{\mu,y_1} p}{2\mu^2} + \frac{q^2}{2} + \alpha \mu^2 \frac{q^3}{3} + \beta \mu^4 \frac{q^4}{4} \right) \di y,  \label{HamFPUcKP} 
\end{align}
where 
\begin{align} \label{I}
I &:= [-1,1]^2,
\end{align}
and $p$ is the variable canonically conjugated to $q$. 

Now, observe that the the operator $\Delta_{\mu,y_1}$ admits the following asymptotic expansion up to terms of order $\cO(\mu^4)$,
\begin{align}
\frac{\Delta_{\mu,y_1}}{\mu^2} &\sim \d_{y_1}^2 + \mu^2 \d_{y_2}^2 + \frac{\mu^2}{12} \d_{y_1}^4 + \cO(\mu^4), \label{ex2Delta1}
\end{align}
therefore the Hamiltonian \eqref{HamFPUcKP} admits the following asymptotic expansion
\begin{align}
\Ham K_1(q,p) &\;\sim\; \hat{\Ham h}_0(q,p) + \mu^2 \hat{\Ham F}_1(q,p) + \mu^4 \hat{\cR}(q,p), \label{asexp3} \\
\hat{\Ham h}_0(q,p) &= \int_{I} \left[\frac{p \, (-\d_{y_1}^2p) + q^2}{2} \right] \,\di y, \label{h0KP} \\
\hat{\Ham F}_1(q,p) &= \int_{I} \left(- \frac{p \, \d_{y_1}^4p}{24} - \frac{p \, \d_{y_2}^2p}{2} + \alpha \frac{q^3}{3}\right) \, \di y. \label{F1KP}
\end{align}

Following the approach of \cite{bambusi2006metastability}, we can introduce the following non-canonical change of coordinates
\begin{align}
\xi  := \frac{1}{\sqrt{2}} (q+\d_{y_1} p), &\; \; \eta := \frac{1}{\sqrt{2}} (q-\d_{y_1} p). \label{xieta} 
\end{align}
which transforms the Poisson tensor into
\begin{align} \label{Poisson}
J &= \d_{y_1} 
\begin{pmatrix}
-1 & 0 \\
0 & 1
\end{pmatrix}
,
\end{align}
and Hamilton equations associated to a Hamiltonian $\Ham K$ are
\begin{align*}
\d_\tau \xi &= -\d_{y_1} \frac{\delta \Ham K}{\delta \xi} \, , \\
\d_\tau \eta &= \d_{y_1} \frac{\delta \Ham K}{\delta \eta}.
\end{align*}

\begin{remark} \label{CasimirRem}
By the explicit expression of the Poisson tensor \eqref{Poisson} we can compute straightforwardly Casimir invariants associated to $J$, which are 
\begin{align} \label{CasimirKdV}
C(\xi,\eta) &= A(y_2) + B(y_2) \int_{-1}^1 \xi(\tau,y_1,y_2) \, \di y_1 + C(y_2)  \int_{-1}^1 \eta(\tau,y_1,y_2) \, \di y_1,
\end{align}
where $A$, $B$ and $C$ are arbitrary functions of $y_2$.

Since Casimir invariants are constants of motion, we can restrict our analysis on the subspace defined by
\begin{align} \label{ZeroAvy1KdV}
\int_{-1}^1 \left[\xi(\tau, y_1, y_2)-\eta(\tau,y_1,y_2)\right] \, \di y_1 &=0 \, , \qquad \forall \tau \in \R, \; |y_2| \leq 1\, .
\end{align}
However, by recalling \eqref{xieta} one sees that \eqref{ZeroAvy1KdV} {is equivalent to}
\begin{align} \label{ZeroAvy1KdV2}
\int_{-1}^1 \d_{y_1}p(\tau, y_1, y_2)\, \di y_1 &=0 \, ,\qquad \forall \tau \in \R, \; |y_2| \leq 1 \, ,
\end{align}
which is {trivially} true due to periodic boundary conditions.
\end{remark}

In the new coordinates the Hamiltonian takes the form
\begin{align}
\Ham{K}_1(\xi,\eta) &\;\sim\; \Ham{h}_0(\xi,\eta) + \mu^2 \Ham F_1(\xi,\eta) + \mu^4 \cR(\xi,\eta), \label{asexp3xieta} \\
\Ham{h}_0(\xi,\eta) &= \int_{I} \frac{\xi^2+\eta^2}{2} \di y, \label{h0KPxieta} \\
\Ham F_1(\xi,\eta) &= \int_{I} \left( -\frac{[\d_{y_1}(\xi-\eta)]^2}{48} + \frac{[\d_{y_2}\d_{y_1}^{-1}(\xi-\eta)]^2}{4} + \alpha \frac{(\xi+\eta)^3}{3 \cdot 2^{3/2}}\right) \, \di y, \label{F1KPxieta}
\end{align}
where \eqref{F1KPxieta} is well defined because of \eqref{ZeroAvy1}. 

Now we apply the averaging Theorem \ref{gavthm} to the Hamiltonian \eqref{asexp3xieta}, with $\delta=\mu^2$, {$\nu=1$}. Observe that the equations of motion of $\Ham{h}_0$ have the following simple form:
\begin{equation} \label{h0flow}
\begin{cases}
\xi_{\tau} & = -\d_{y_1}\xi \\
\eta_\tau & = \d_{y_1}\eta
\end{cases}
; \qquad
\begin{cases}
\xi(\tau,y) & = \xi_0(y_1-\tau,y_2) \\
\eta(\tau,y) & = \eta_0(y_1+\tau,y_2)
\end{cases}
.
\end{equation}

\begin{proposition} \label{F1avKPprop}
The average of $\Ham F_1$ in \eqref{asexp3xieta} with respect to the flow of $\Ham{h}_0$ in \eqref{asexp3xieta} is given by
\begin{align} \label{F1avKP}
\la \Ham F_1 \ra(\xi,\eta) &=  \int_{I} \left(- \frac{(\d_{y_1}\xi)^2+(\d_{y_1}\eta)^2}{48} + \frac{ (\d_{y_2}\d_{y_1}^{-1}\xi)^2 + (\d_{y_2}\d_{y_1}^{-1}\eta)^2 }{4}\right) \, \di y + \frac{\alpha}{3 \cdot 2^{3/2}} ([\xi^3]+[\eta^3]) 
\end{align}
where we denote by $[f^j]$ the average $\int_{I} f^j(y) \frac{\di y}{4}$.
\end{proposition}

\begin{proof}
{ For the computation of $\langle \Ham F_1 \rangle(\xi,\eta)$ one can exchange the order of the integrations in $y$ and $s$. To compute averages we use the following elementary facts:
\begin{enumerate}
\item[i.] let $u,v \in L^2([-1,1])$, then
\begin{align*}
\int_{-1}^1 \di y \int_{-1}^1 \di s  \, u(y \pm s)\, v(y \mp s) &=  \int_{-1}^1 u(y)\, \di y \int_{-1}^1 v(y) \, \di y;
\end{align*}
\item[ii.] let $u \in L^1([-1,1])$, then
\begin{align*}
\frac{1}{2} \int_{-1}^1 \di s \int_{-1}^1 \di y \; u(y \pm s) &= \int_{-1}^1 u(x) \, \di x \, .
\end{align*}
\end{enumerate}
Since we assume our functions to be periodic, then also $\int_I \partial_{y_r} u(y) \, \di y = 0$, $r=1,2$, and hence only terms with only $\xi$ or $\eta$ are not cancelled by the average procedure. }
\end{proof}

\begin{corollary} \label{KPcor}
The equations of motion associated to $\Ham{h}_0(\xi,\eta)+\mu^2\la \Ham F_1 \ra(\xi,\eta)$ are given by 
\begin{equation} \label{KPsys}
\begin{cases}
\xi_{\tau} & = -\d_{y_1}\xi - \frac{\mu^2}{24} \d_{y_1}^3\xi - \frac{\mu^2}{2} \d_{y_1}^{-1}\d_{y_2}^2\xi - \frac{\alpha \mu^2}{2\sqrt{2}} \, \d_{y_1}(\xi^2) \\
\eta_{\tau} & = \d_{y_1}\eta + \frac{\mu^2}{2} \d_{y_1}^{-1}\d_{y_2}^2\eta + \frac{\mu^2}{24} \d_{y_1}^3\eta + \frac{\alpha \mu^2}{2\sqrt{2}} \, \d_{y_1}(\eta^2)
\end{cases}
.
\end{equation}
\end{corollary}

More explicitly, we observe that \eqref{KPsys} is a system of two uncoupled 
KP equations on a two-dimensional torus in translating frames.

\subsection{The KdV regime for the ETL lattice} \label{KdVsubsec}

For this regime we consider \eqref{FPUeqc}, with $\alpha \neq 0$, 
and we look for small amplitude solutions of the form 
\begin{align} \label{KdVr1}
Q(t,x) &= \mu^2 q(\mu t, \mu x_1, \mu^\sigma x_2),
\end{align}
where $q:\mathbb{R} \times \mathbb{T}^2 \to \mathbb{R}$ is a periodic function  and $\mu$, $\sigma > 2$ are defined in \eqref{mu}-\eqref{sigma}. We introduce the rescaled variables $\tau = \mu t$, $y_1=\mu x_1$, $y_2=\mu^\sigma x_2$, and we denote $I$ is as in \eqref{I}.
Plugging \eqref{KdVr1} into \eqref{FPUeqc}, we get 
\begin{align}
q_{\tau \tau} &= \frac{\Delta_{\mu,y_1,\sigma}}{\mu^2} \left( q + \mu^2 \alpha q^2 + {\mu^4 \beta q^3} \right), \; \; \Delta_{\mu,y_1,\sigma} := 4 \sinh^2\left(\frac{\mu \d_{y_1}}{2}\right) + 4 \sinh^2\left(\mu^{\sigma} \frac{\d_{y_2}}{2}\right), \label{FPUeqKdVr1} 
\end{align}
which is a Hamiltonian PDE corresponding to the Hamiltonian functional
\begin{align}
\Ham K_2(q,p) &= \int_{I} \left( \frac{-p \, \Delta_{\mu,y_1,\sigma} p}{2\mu^2} + \frac{q^2}{2} + \alpha \mu^2 \frac{q^3}{3} { + \beta \mu^4 \frac{q^4}{4} } \right) \, \di y,  \label{HamFPUcKdV} 
\end{align}
and $p$ is the variable canonically conjugated to $q$. 

Now, observe that the the operator $\Delta_{\mu,y_1,\sigma}$ admits the following asymptotic expansion, 
\begin{align}
\frac{\Delta_{\mu,y_1,\sigma}}{\mu^2} &\sim \d_{y_1}^2 {+ \frac{\mu^2}{12} \d_{y_1}^4 + \mu^{2(\sigma-1)}\d_{y_2}^2 + \frac{ \mu^{2(2\sigma-1)} }{12} \, \d_{y_2}^4 } + \sum_{ {m \geq 2} } c_m \left(\mu^{2m} \d_{y_1}^{2(m+1)} + \mu^{2[(m+1)\sigma-1]} \d_{y_2}^{2(m+1)}\right), \label{exDelta1sigma} \\
c_m &:= { \frac{2}{(2m+2)!}  }, \nonumber
\end{align}
and by recalling that $\sigma > 2$ we have 
\begin{align}
\frac{\Delta_{\mu,y_1,\sigma}}{\mu^2} &\sim \d_{y_1}^2 + \frac{\mu^2}{12} \d_{y_1}^4 + \cO({ \mu^{\min(2\sigma-2,4)} }) \label{ex2Delta1sigma} \, .
\end{align}
Therefore the Hamiltonian \eqref{HamFPUcKdV} admits the following asymptotic expansion
\begin{align}
\Ham K_2(q,p) &\;\sim\; \hat{\Ham h}_0(q,p) + \mu^2 \hat{\Ham F}_2(q,p) + { \mu^{\min(2\sigma-2,4)} } \hat{\cR}(q,p), \label{asexp1} \\
\hat{\Ham h}_0(q,p) = \int_{I} \frac{-p \, (\d_{y_1}^2p) + q^2}{2} \di y, \; \; &\; \; \; \hat{\Ham F}_2(q,p) = \int_{I} \left(- \frac{p \, \d_{y_1}^4p}{24} + \alpha \frac{q^3}{3}\right) \, \di y. \label{F1KdV}
\end{align}
Note that the nonlinearity of degree $4$ does not affect the Hamiltonian 
up to order $\cO(\mu^4)$. 

By exploiting again the non-canonical change of coordinates $(q,p) \mapsto (\xi,\eta)$ introduced in \eqref{xieta} and the Poisson tensor \eqref{Poisson}, and 
\begin{align} \label{ZeroAvy1}
\int_{-1}^1 \left[ \xi(\tau, y_1, y_2)-\eta(\tau,y_1,y_2)\right] \, \di y_1 &=0 \, , \qquad \forall \tau \in \R, \; |y_2| \leq 1,
\end{align}
we obtain
\begin{align}
\Ham K_2(\xi,\eta) &\;\sim\; \Ham{h}_0(\xi,\eta) + \mu^2 \Ham F_2(\xi,\eta) + { \mu^{\min(2\sigma-2,4)} } \cR(\xi,\eta), \label{asexp1xieta} \\
\Ham{h}_0(\xi,\eta) = \int_{I} \frac{\xi^2+\eta^2}{2} \di y, \; \; &\; \; \; \Ham F_2(\xi,\eta) = \int_{I} \left(-\frac{[\d_{y_1}(\xi-\eta)]^2}{48} + \alpha \frac{(\xi+\eta)^3}{3 \cdot 2^{3/2}}\right) \, \di y. \label{F1xieta}
\end{align}
Now we apply the averaging Theorem \ref{gavthm} to the Hamiltonian \eqref{asexp1xieta}, with $\delta=\mu^2$, {$\nu = \min( \sigma-2,1 )$}.

\begin{proposition} \label{F1avKdVprop}
The average of $\Ham F_2$ in \eqref{asexp1xieta} with respect to the flow of $\Ham{h}_0$ in \eqref{F1xieta} is given by
\begin{align} \label{F1avKdV}
\la \Ham F_2 \ra(\xi,\eta) &= - \int_{I} \frac{(\d_{y_1}\xi)^2+(\d_{y_1}\eta)^2}{48} \di y + \frac{\alpha}{3 \cdot 2^{3/2}} ([\xi^3]+[\eta^3]),
\end{align}
where we denote by $[f^j]$ the average $\int_{I} f^j(y) \frac{\di y}{4}$.
\end{proposition}

\begin{corollary} \label{KdVcor}
The equations of motion associated to $\Ham{h}_0(\xi,\eta)+\mu^2\la \Ham F_1 \ra(\xi,\eta)$ are given by 
\begin{equation} \label{KdVsys}
\begin{cases}
\xi_{\tau} & = -\d_{y_1}\xi - \frac{\mu^2}{24} \d_{y_1}^3\xi - \frac{\mu^2 \alpha}{2\sqrt{2}} \, \d_{y_1}(\xi^2) \\
\eta_{\tau} & = \d_{y_1}\eta + \frac{\mu^2}{24} \d_{y_1}^3\eta + \frac{\mu^2 \alpha}{2\sqrt{2}} \, \d_{y_1}(\eta^2)
\end{cases}
.
\end{equation}
\end{corollary}

The latter is a system of two uncoupled KdV equations in translating frames with respect to the $y_1$-direction, for each fixed value of the coordinate $y_2$.

\begin{remark} \label{mKdVrem}
One can also study the $\beta$ model (namely, \eqref{FPUeqc} with $\alpha=0$, $\beta \neq 0$) in the following regime,
\begin{itemize}
\item[(mKdV)] the $\beta$ model in the very weakly transverse regime, $Q(t,x) = \mu \, q(\mu t, \mu x_1,\mu^\sigma x_2)$, where $\mu \ll 1$, {$ \sigma > 2$}.
\end{itemize}
Let us introduce again the rescaled variables $\tau = \mu t$, $y_1=\mu x_1$, $y_2=\mu^\sigma x_2$, and the domain $I$ as in \eqref{I}; plugging the ansatz for $Q$ into \eqref{FPUeqc}, we get 
\begin{align}
q_{\tau \tau} &= \frac{\Delta_{\mu,y_1,\sigma}}{\mu^2} \left( q + \mu^2 \beta q^3 \right), \label{FPUeqmKdVr1}
\end{align}
where $\Delta_{\mu,y_1,\sigma}$ is the operator introduced in \eqref{FPUeqKdVr1}. Eq. \eqref{FPUeqmKdVr1} is a Hamiltonian PDE with the following corresponding Hamiltonian,
\begin{align}
\Ham K_3(q,p) &= \int_{I} \left(\frac{-p \, \Delta_{\mu,y_1,\sigma} p}{2\mu^2} + \frac{q^2}{2} + \beta \mu^2 \frac{q^4}{4}\right) \, \di y,  \label{HamFPUcmKdV}
\end{align}
where $p$ is the variable canonically conjugated to $q$. Recalling that \eqref{ZeroAvy1KdV} holds true, we exploit again the non-canonical change of coordinates \eqref{xieta} and the Poisson tensor \eqref{Poisson}, obtaining that
\begin{align}
\Ham K_3(\xi,\eta) &\;\sim\; \Ham{h}_0(\xi,\eta) + \mu^2 \Ham F_3(\xi,\eta) + { \mu^{\min(2\sigma-2,4)} } \cR(\xi,\eta), \label{asexp2xieta}
\end{align}
where $\Ham{h}_0$ is the same as in \eqref{F1xieta}, while
\begin{align}
\Ham F_3(\xi,\eta) &= \int_{I}\left( -\frac{[\d_{y_1}(\xi-\eta)]^2}{48} + \beta \frac{(\xi+\eta)^4}{2^4}\right) \, \di y. \label{F1mxieta}
\end{align}
Applying Theorem \ref{gavthm} to the Hamiltonian \eqref{asexp2xieta} with $\delta=\mu^2$ and {$\nu = \min( \sigma-2,1 )$}, we get that the equations of motion associated to $\Ham{h}_0(\xi,\eta)+\mu^2\la \Ham F_3 \ra(\xi,\eta)$ are given by 
\begin{equation} \label{mKdVsys}
\begin{cases}
\xi_{\tau} & = - \left( 1+ \frac{3}{4} [\eta^2] \right) \d_{y_1}\xi - \frac{\mu^2}{24} \d_{y_1}^3\xi - \frac{\mu^2 \beta}{4} \, \d_{y_1}(\xi^3) \\
\eta_{\tau} & = \left( 1+ \frac{3}{4} [\xi^2] \right) \d_{y_1}\eta + \frac{\mu^2}{24} \d_{y_1}^3\eta + \frac{\mu^2 \beta}{4} \, \d_{y_1}(\eta^3)
\end{cases}
.
\end{equation}
which is a system of two uncoupled mKdV equations in translating frames 
with respect to the $y_1$-direction. The integrability properties of the mKdV equation and the existence of Birkhoff coordinates for this model have been proved in \cite{kappeler2008mkdv}.
\end{remark}

\subsection{The one-dimensional NLS regime for the KG Lattice} \label{1DNLSsubsec}
We want to study small amplitude solutions of \eqref{2DKGseq}, with initial data in which only one low-frequency Fourier mode is excited.

Analogously to the procedure of the previous Sections, the first step is to introduce an interpolating function $Q=Q(t,x)$ such that 
\begin{itemize}
\item[(B1)] $Q(t,j)=Q_j(t)$, for all $j \in \Z^2_{N_1,N_2}$;
\item[(B2)] $Q$ is periodic with period $2N_1+1$ in the $x_1$-variable, and periodic with period $2N_2+1$ in the $x_2$-variable;
\item[(B3)] $Q$ fulfills 
\begin{align}
\ddot{Q} &= \Delta_1 Q - Q - \beta Q^{2p+1}, \label{KGseqc}
\end{align}
where $\Delta_1$ is the operator defined in \eqref{Delta1c} 
(recall that we also assumed $m=1$ in \eqref{potKG2}).
\end{itemize}

It is easy to verify that \eqref{KGseqc} is Hamiltonian with Hamiltonian function
\begin{align} \label{HamKGsc}
\Ham H_{\color{black}\mathrm{KG}}(Q,P) &= \int_{[-\frac{1}{\mu},\frac{1}{\mu}] \times [-\frac{1}{\mu^\sigma},\frac{1}{\mu^\sigma}] } \left( \frac{P^2}{2} +  \frac{Q^2}{2} - \frac{Q \; \Delta_1 Q}{2} + \beta \frac{Q^{2p+2}}{2p+2}\right) \, \di x,
\end{align}
where $P$ is a periodic function and is canonically conjugated to $Q$.

Starting from the Hamiltonian \eqref{Ham2KGs}, where $p=1$, 
 we look for small amplitude solutions of the form
\begin{align} \label{1DNLSr}
Q(t,x) &= \mu \, q(\mu^2 t, \mu x_1,\mu^\sigma x_2) \, .
\end{align} 
where $q : \mathbb{R} \times \mathbb{T}^2 \to \mathbb{R}$ is a periodic function and $\sigma,\mu$ are defined respectively in \eqref{sigma}-\eqref{mu}.

We introduce the rescaled variables {\color{black}$\tau=\mu^2 t$}, $y_1=\mu x_1$ and $y_2=\mu^\sigma x_2$, and we define $I$ as in \eqref{I}. 
The Hamiltonian \eqref{Ham2KGs} in the rescaled variables is given by 
\begin{align}
\Ham K_4(q,p) &= \int_{I} \left(\frac{p^2}{2} + \frac{q^2}{2} - \frac{q \; \Delta_{\mu,y_1,\sigma} q}{2} + \beta \mu^2 \frac{q^4}{4}\right) \, \di y,  \label{Ham2KGc1DNLS} 
\end{align}
with the operator $\Delta_{\mu,y_1,\sigma}$ as in \eqref{FPUeqKdVr1}, 
and $p$ is the variable canonically conjugated to $q$. 
The corresponding equation of motion is given by
\begin{align}
q_{\tau \tau} &= - q + \Delta_{\mu,y_1,\sigma} q - \beta \mu^2 q^3. \label{1DKGsc2} 
\end{align}
Recalling \eqref{ex2Delta1sigma}, we have that the Hamiltonian \eqref{Ham2KGc1DNLS} admits the following asymptotic expansion
\begin{align}
\Ham K_4(q,p) &\;\sim\;\hat{\Ham h}_4(q,p) + \mu^2 \hat{\Ham F}_4(q,p) + { \mu^{ \min(2\sigma,4) } } \hat{\cR}(q,p), \label{asexp4} \\
\hat{\Ham h}_4(q,p) = \int_{I} \frac{p^2 + q^2}{2}\, \di y, \; \; &\; \; \; \hat{\Ham F}_4(q,p) = \int_{I} \left(- \frac{q \, \d_{y_1}^2 q}{2} + \beta \frac{q^4}{4}\right) \, \di y, \label{F11DNLS}
\end{align}
and the equation of motion associated to ${ \hat{\Ham{h}}_4 + \mu^2 \hat{\Ham F}_4 }$ is given by the following cubic one-dimensional nonlinear Klein-Gordon (NLKG) equation,
\begin{align} \label{1DNLKGeq}
q_{\tau \tau} &= -(q - \mu^2 \d_{y_1}^2 q) - \mu^2 \beta q^3.
\end{align}

We now exploit the change of coordinates $(q,p) \mapsto (\psi,\bar\psi)$ given by
\begin{align} \label{psi}
\psi &= \frac{1}{\sqrt{2}} (q- \ii p),
\end{align}
therefore the inverse change of coordinates is given by
\begin{align}
q = \frac{1}{\sqrt{2}} (\psi+\bar\psi), \; \; &\; \; \; p = \frac{\ii}{\sqrt{2}} (\psi-\bar\psi), \label{qp} 
\end{align}
while the symplectic form is given by $-\ii \, \di \psi \wedge \di\bar\psi$. 
With this change of variables the Hamiltonian takes the form
\begin{align}
\Ham K_4(\psi,\bar\psi) &\;\sim\; \Ham{h}_4(\psi,\bar\psi) + \mu^2 \Ham F_4(\psi,\bar\psi) + { \mu^{ \min(2\sigma,4) } } \cR(\psi,\bar\psi), \label{asexp4psi} \\
\Ham{h}_4(\psi,\bar\psi) = \int_{I} \psi \, \bar\psi \, \di y, \; \; &\; \; \; \Ham F_4(\psi,\bar\psi) = \int_{I} \left(-\frac{(\psi+\bar\psi) \, [-\d_{y_1}^2(\psi+\bar\psi)] }{4} + \beta \frac{(\psi+\bar\psi)^4}{16}\right) \, \di y. \label{F11DNLSpsi}
\end{align}
Now we apply the averaging Theorem \ref{gavthm} to the Hamiltonian \eqref{asexp4psi}, with $\delta=\mu^2$, {$\nu=\min(\sigma-1,1)$}. Observe that $\Ham{h}_4$ generates a periodic flow,
\begin{align}
-\ii \d_\tau \psi = \psi; \; &\; \; \psi({\tau} ,y) = e^{\ii \tau}\psi_0(y). \label{h0flowNLS}
\end{align}

\begin{proposition} \label{F1av1DNLSprop}
The average of $\Ham F_4$ in \eqref{asexp4psi} with respect to the flow of $\Ham{h}_4$ \eqref{F11DNLS} is given by
\begin{align} \label{F1av1DNLS}
\la \Ham F_4 \ra(\psi,\bar\psi) &=  \int_{I}  \frac{\bar\psi \; (-\d_{y_1}^2 \psi)}{2} \di y + \frac{3 \beta}{8} \int_I |\psi|^4 \di y.
\end{align}
\end{proposition}

\begin{corollary} \label{1DNLScor}
The equations of motion associated to $\Ham{h}_4(\psi,\bar\psi)+\mu^2\la \Ham F_4 \ra(\psi,\bar\psi)$ are given by a cubic one-dimensional nonlinear Schr\"odinger equation for each fixed value of $y_2$,
\begin{align} \label{1DimNLSeq}
-\ii \psi_{\tau} &= \psi - \mu^2 \; \d_{y_1}^2\psi + \mu^2 \frac{3 \beta}{4} \, |\psi|^2\psi.
\end{align}
\end{corollary}

\begin{remark} \label{2DNLSrem}
{If we apply the above argument to the Hamiltonian \eqref{Ham2KGs} in the regime}
\begin{itemize}
\item[(2-D NLS)] the scalar model \eqref{Ham2KGs} with $m=1$, $p=1$ and $Q(t,x) = \mu \, q(\mu^2 t, \mu x)$, where $\mu \ll 1$ and $\sigma =1$,
\end{itemize}
{we can obtain as a normal form equation the cubic nonlinear Schr\"odinger (NLS) equation}
\begin{align} \label{NLSeq}
-\ii \psi_{\tau} &= \psi - \mu^2 \; \Delta\psi + \mu^2 \frac{3 \beta}{4} \, |\psi|^2\psi.
\end{align}
\end{remark}

\section{Dynamics of the normal form equation} \label{BNFdyn}

\subsection{The KP equation} \label{KPdynsubsec}

In this Section we recall some known facts on the dynamics of the KP equation
on the two-dimensional torus 
\begin{align} \label{KPeq}
\xi_{\tau} & = - \frac{1}{24} \d_{y_1}^3\xi - \frac{1}{2} \d_{y_1}^{-1}\d_{y_2}^2\xi - \frac{\alpha}{2\sqrt{2}} \, \d_{y_1}(\xi^2), \; \alpha = \pm 1, \; y \in \T^2:=\R^2/(2\pi\Z)^2.
\end{align}

The KP equation has been introduced in order to describe 
weakly-transverse solutions of the water waves equations; 
it has been considered as a two-dimensional analogue of the KdV equation, 
since also the KP equation admits an infinite number of constants of motions 
\cite{lin1982constraints} \cite{chen1983new} \cite{chen1987infinite}. 
It is customary to refer to \eqref{KPeq}  as KP-I equation when $\alpha=-1$, and as KP-II equation when $\alpha=1$.

The global-well posedness for the KP-II equation on the two-dimensional torus has been discussed by Bourgain in \cite{bourgain1993cauchy}. The main point of the result by Bourgain consists in extending the local well-posedness result to a global one, even though the $L^2$-norm is the only constant of motion for the KP-II equation that allows an a-priori bound for the solution (see Theorem 8.10 and Theorem 8.12 in \cite{bourgain1993cauchy}).

\begin{theorem} \label{BouThm}
Consider \eqref{KPeq} with $\alpha=1$. 

Let $\rho \geq 0$ and $s \geq 0$, and assume that the initial datum $\xi(0,\cdot,\cdot)=\xi_0 \in \ell^2_{\rho,s}$. Then \eqref{KPeq} is globally well-posed in $\ell^2_{\rho,s}$. Moreover, the $\ell^2$-norm of the solution is conserved,
\begin{align} \label{L2normKP}
\|\xi(t)\|_{\ell^2} &= \|\xi_0\|_{\ell^2},
\end{align}
while
\begin{align} \label{HsnormKP}
\|\xi(t)\|_{ \ell^2_{0,s} } &\leq e^{C |t|} \|\xi_0\|_{ \ell^2_{0,s} },
\end{align}
where $C$ depends on $s$.
\end{theorem}

\begin{remark}
As pointed out by Bourgain in Sec. 10.2 of \cite{bourgain1993cauchy}, a global well-posedness result for sufficiently smooth solution of the KP-I equation (namely, \eqref{KPeq} with $\alpha=-1$) on the two-dimensional torus can be obtained by generalizing the argument in \cite{schwarz1987periodic} for small data and by using the a-priori bounds given by the constants of motion for the KP-I equation.
\end{remark}

{Last, we mention that f}or the KP equation the construction of action-angle/Birkhoff coordinates is still an open problem.

\subsection{The KdV equation} \label{KdVdynsubsec}

In this Section we recall some known facts on the dynamics of the KdV equation
with periodic boundary conditions. The interested reader can find more detailed explanations and proofs in \cite{kappeler2003kdv}.

Consider the KdV equation 
\begin{align} \label{KdVeq}
\xi_\tau &= - \frac{1}{24} \d_{y_1}^3 \xi - \frac{\alpha}{2\sqrt{2}} \d_{y_1}(\xi^2), \qquad y_1 \in [0,2].
\end{align}

Through the Lax pair formulation of the evolution problem \eqref{KdVeq} one get that the periodic eigenvalues $(\lambda_n)_{n \in \mathbb{N}}$ of the Sturm-Liouville operator
\begin{align}
L_\xi &:= -\d_{y_1}^2 + 6 \sqrt{2} \xi(\tau,y_1)
\end{align}
are conserved quantities under the evolution of the KdV equation \eqref{KdVeq}. 
Moreover, if we define the gaps of the spectrum $\gamma_m := \lambda_{2m}-\lambda_{2m-1}$ ($m \geq 1$), it is well known that the squared spectral gaps $(\gamma_m^2)_{m \geq 1}$ form a complete set of constants of motion for \eqref{KdVeq}. 

The following relation between the sequence of the spectral gaps and the regularity of the corresponding solution to the KdV equation holds (see Theorem 9, Theorem 10 and Theorem 11 in \cite{kappeler2008well}; see also \cite{poschel2011hill})

\begin{theorem} \label{KapPosThm1}
Assume that $\xi \in L^2$, then $\xi \in \ell^2_{0,s}$ if and only if its spectral gaps satisfy
\begin{align*}
\sum_{m \geq 1} m^{2s} |\gamma_m|^2 &< +\infty.
\end{align*}
Moreover if $\xi \in \ell^2_{\rho,s}$, then
\begin{align} \label{SpecGapEst}
\sum_{m \geq 1} m^{2s} e^{2 \rho m} |\gamma_m|^2 &< +\infty;
\end{align}
conversely, if \eqref{SpecGapEst} holds, then $\xi \in \ell^2_{\rho',0}$ for some $\rho'>0$.
\end{theorem}

Kappeler and P\"oschel constructed the following global Birkhoff coordinates (see Theorem 1.1 in \cite{kappeler2003kdv})

\begin{theorem} \label{KapPosThm2}
There exists a diffeomorphism $\Omega:L^2 \to \ell^2_{0,1/2} \times \ell^2_{0,1/2}$ such that:
\begin{itemize}
\item $\Omega$ is bijective, bianalytic and canonical;
\item for each $s \geq 0$, the restriction of $\Omega$ to $\ell^2_{0,s}$, namely the map
\begin{align*}
\Omega:\ell^2_{0,s} &\to \ell^2_{0,s+1/2} \times \ell^2_{0,s+1/2}
\end{align*}
is bijective, bianalytic and canonical;
\item the coordinates $(x, y) \in \ell^2_{0,3/2} \times \ell^2_{0,3/2}$ are Birkhoff coordinates for the KdV equation, namely they form a set of canonically conjugated coordinates in which the Hamiltonian of the KdV equation \eqref{KdVeq} depends only on the action $I_m := \frac{x_m^2 + y_m^2}{2}$ ($m \geq 1$).
\end{itemize}
\end{theorem}

The dynamics of the KdV equation \eqref{KdVeq} in terms of the variables $(x,y)$ is trivial: it can be immediately seen that any solution is periodic, quasiperiodic or almost periodic, depending on the number of spectral gaps (equivalently, depending on the number of actions) initially different from zero.

\subsection{The one-dimensional cubic NLS equation} \label{1DNLSdynsubsec}

In this Section we recall some known facts on the dynamics of the one-dimensional cubic defocusing NLS equation with periodic boundary conditions. The interested reader can find more detailed explanations and proofs in \cite{grebert2014defocusing} \cite{molnar2014new}.

Consider the cubic defocusing NLS equation 
\begin{align} \label{1DNLSeq}
\ii\psi_\tau &= - \d_{y_1}^2 \psi + 2 |\psi|^2\psi, \qquad y_1 \in \T:=\R/(2\pi\Z).
\end{align}

Eq. \eqref{1DNLSeq} is a PDE admitting a Hamiltonian structure: indeed, we can set $\cH^{\rho,s} = \ell^2_{\rho,s} \times \ell^2_{\rho,s}$ as the phase space with elements denoted by $\phi=(\phi_1,\phi_2)$, while the associated Poisson bracket and the Hamiltonian are given by
\begin{align}
\{\Ham F, \Ham G\} &:= -\ii \int_\T \left( \d_{\phi_1} \Ham F \, \d_{\phi_2} \Ham G - \d_{\phi_1} \Ham G \, \d_{\phi_2}\Ham F \right) \di y_1, \label{1DNLSPoisson} \\
\Ham H_{\mathrm{NLS}}(\phi_1,\phi_2) &:= \int_\T \left( \d_{y_1}\phi_1 \, \d_{y_1}\phi_2 + \phi_1^2 \phi_2^2 \right) \; \di y_1. \label{1DNLSHam}
\end{align}
The defocusing NLS equation \eqref{1DNLSeq} is obtained by restricting \eqref{1DNLSHam} to the invariant subspace of states of real type,
\begin{align} \label{RealStates}
\cH^{\rho,s} _r &:= \{ \phi \in \cH^{\rho,s} : \phi_2=\bar{\phi}_1 \}.
\end{align}
The above Hamiltonian \eqref{1DNLSHam} is well-defined on $\cH^{\rho,s}$ with $s \geq 1$ and $\rho \geq 0$, while the initial value problem for the NLS equation \eqref{1DNLSeq} is well-posed on $\cH^{0,0}=\ell^2 \times \ell^2$.

It is well known from the work by Zakharov and Shabat that the NLS equation \eqref{1DNLSeq} has a Lax pair, and that it admits infinitely many constants of motion in involution. More precisely, for any $\phi \in \cH^{0,0}$ consider the Zakharov-Shabat operator
\begin{align} \label{ZSOp}
L(\phi) &= 
\begin{pmatrix}
\ii & 0 \\
0 & -\ii
\end{pmatrix}
\d_{y_1} + 
\begin{pmatrix}
0 & \phi_1 \\
\phi_2 & 0
\end{pmatrix}
,
\end{align}
where we call $\phi$ the potential of the operator $L(\phi)$. 
The spectrum of $L(\phi)$ on the interval $[0,2]$ with periodic boundary conditions is pure point, and it consists of the following sequence of periodic eigenvalues
\begin{align} \label{perNLSspec}
\cdots < \lambda_{-1}^- \leq \lambda_{-1}^+ < \lambda_0^- &\leq \lambda_0^+ < \lambda_1^- \leq \lambda_1^+ < \cdots,
\end{align}
where the quantities $\gamma_m := \lambda_{m}^+ - \lambda_{m}^-$ ($m \in \Z$) are called gap lengths. It has been proved that the squared spectral lengths $(\gamma_m^2)_{m \in \Z}$ form a complete set of analytic constants of motion for \eqref{1DNLSeq}. 

Gr\'ebert, Kappeler and Mityagin proved the following relation between the sequence of the squared spectral gaps and the regularity of the corresponding potential (see Theorem in \cite{grebert1998gap}).

\begin{theorem} \label{GrebKapThm1}
Let $\rho \geq 0$ and $s >0$, then for any bounded subset $\cB \subset \ell^2_{\rho,s} \times \ell^2_{\rho,s}$ there exists $n_0 \geq 1$ and $M \geq 1$ such that for any $|k| \geq n_0$ and any $(\phi_1,\phi_2) \in \cB$, the following estimate holds
\begin{align} \label{SpecGapEstNLS}
\sum_{|k| \geq n_0} (1+|k|)^{2s} e^{2 \rho |k|} |\gamma_m|^2 &\leq M.
\end{align}
\end{theorem}

Moreover, Gr\'ebert and Kappeler constructed the following global Birkhoff coordinates (see Theorem 20.1 - Theorem 20.3 in \cite{grebert2014defocusing})

\begin{theorem} \label{GrebKapThm2}
There exists a diffeomorphism $\Omega:L^2_r \to \cH^{0,0}_r$ such that:
\begin{itemize}
\item $\Omega$ is bianalytic and canonical;
\item for each $s \geq 0$, the restriction of $\Omega$ to $\cH^{0,s}_r$, namely the map
\begin{align*}
\Omega:\cH^{0,s}_r &\to \cH^{0,s}_r
\end{align*}
is again bianalytic and canonical;
\item the coordinates $(x, y) \in \cH^{0,1}_r$ are Birkhoff coordinates for the NLS equation, namely they form a set of canonically conjugated coordinates in which the Hamiltonian of the NLS equation \eqref{1DNLSeq} depends only on the action $I_m := \frac{x_m^2 + y_m^2}{2}$ ($m \in \Z$).
\end{itemize}
\end{theorem}

The dynamics of the NLS equation \eqref{1DNLSeq} in terms of the variables $(x,y)$ is trivial: it can be immediately seen that any solution is periodic, quasiperiodic or almost periodic, depending on the number of spectral gaps (equivalently, depending on the number of actions) initially different from zero.

\section{Approximation results} \label{ApprSec}

In this Section we show how to use the normal form equations in order to construct approximate solutions of \eqref{2DETLeq} and \eqref{2DKGseq}, and we estimate the difference with respect to the true solutions with corresponding initial data.

{The approach is the same for all regimes \eqref{KPr1}, \eqref{KdVr1} and \eqref{1DNLSr}. First, we point out a relation between the specific energy of normal mode $\cE_\kappa=2 E_k/N$ (where the energy of normal mode $E_k$ is defined in \eqref{EnNormMode} for \eqref{2DETLeq}, and in \eqref{EnNormModeKG} for \eqref{2DKGseq}), $k \in \Z^2_{N_1,N_2}$, and the Fourier coefficients of the solutions of the normal form equations. This procedure has to be done carefully, since all wavevectors $k+(N_1 L_1, N_2L_2)$ contributes to the specific energy $\cE_\kappa$ of the discrete system. Then we have to prove that the approximate solutions approximate the specific energy of the true normal mode $\cE_\kappa$ up to the time-scale for which the continuous approximation is valid, and finally we can deduce the result about the dynamics of the lattice. For simplicity we present in this Section only the main part of the argument, and we defer the proof of technical results to Appendices \ref{ApprEstSec11}-\ref{ApprEstSec22}.}

\subsection{The KP regime} \label{apprKPsubsec}

Let $I=[-1,1]^2$ be as in \eqref{I} {and let $k=(k_1,k_2) \in \mathbb{Z}^2$}. We define the Fourier coefficients of the function $q:I \to \R$ by

\begin{align} \label{Fourierqcont}
\hat{q}_{\color{black}{k}} &:= \frac{1}{2}  \int_{I} q(y_1,y_2) \, e^{-\ii \pi ({k}_1 y_1 + {k}_2 y_2) } \di y_1 \, \di y_2,
\end{align}
and similarly for the Fourier coefficients of the function $p$. {Note that in this definition we omitted an eventual time dependence because we are interested in relations among quantities for points in the phase space, i.e. pairs of functions $(q,p):I\to \mathbb{R}$. These relations extend automatically to trajectories and, in fact, this will be the subject of the last result of each subsection.}

\begin{lemma} \label{EnSpecKPLemma}
Consider the lattice \eqref{HamQ} in the regime (KP) and with interpolating function \eqref{KPr1}. Then for a state corresponding to $(q,p)$ one has
\begin{align} \label{SpecEnNormModeKP}
\cE_\kappa &= \frac{\mu^4}{2}  \sum_{ L=(L_1,L_2) \in \Z^2:\mu L_1, \mu^2 L_2 \in 2\Z } \left( |\hat{q}_{K+L}|^2 + \omega_k^2 \left| \frac{\hat{p}_{K+L}}{\mu} \right|^2\right) , \; \; \forall k : \kappa(k) = (\mu K_1,\mu^2 K_2)
\end{align}
(where the $\omega_k$ are defined as in \eqref{FreqNormMode}), 
and $\cE_\kappa=0$ otherwise.

\end{lemma}

\begin{proof}
Take a smooth $(2N_1+1,2N_2+1)$-periodic interpolating function $Q$ for $Q_j$, and similarly for $P_j$. We denote by
\begin{align} \label{FourierQcont}
\hat{Q}({{k}}) &:= \frac{1}{ \sqrt{N}  } \int_{\left[-\left( N_1+\frac{1}{2} \right),\left( N_1+\frac{1}{2} \right) \right] \times \left[-\left( N_2+\frac{1}{2} \right),\left( N_2+\frac{1}{2} \right) \right] } Q(x) e^{-2\pi \ii \frac{ {k} \cdot x }{ N  } } \, \di x,
\end{align}
so that by the interpolation property we obtain
\begin{align*}
Q_j  &=  Q(j) \, = \, \frac{1}{ \sqrt{N} } \sum_{k \in \Z^2} \hat{Q}({k}) e^{ 2 \pi \ii  \frac{j \cdot k}{ N  } }  \nonumber \\
&=  \frac{1}{ \sqrt{N}  } \sum_{k=(k_1,k_2) \in \Z^2_{2N_1+1,2N_2+1}} \left[ \sum_{h=(h_1,h_2) \in \Z^2} \hat{Q}( k_1+(2N_1+1)h_1, k_2+(2N_2+1)h_2 ) \right] \, e^{2 \pi \ii\frac{ j \cdot k}{ N } },
\end{align*}
hence
\begin{align} \label{FourierRel}
\hat{Q}_{{K}} &= \sum_{h \in \Z^2} \hat{Q}( K_1+(2N_1+1)h_1, K_2+(2N_2+1)h_2 ).
\end{align}
The relation between $\hat{Q}(k)$ and $\hat{q}_k$ can be deduced from \eqref{KPr1} {and from the rescalings $\tau=\mu t$, $y_1=\mu x_1$, $y_2=\mu^2 x_2$},
\begin{align}
Q(j) &= \mu^2 q(\mu j_1,\mu^2 j_2); \nonumber \\
\hat{Q}{\color{black}(k)} &=  \frac{1}{2} \mu^{3/2} \, \int_{ \left[ -\frac{1}{\mu},\frac{1}{\mu} \right] \times \left[ -\frac{1}{\mu^2},\frac{1}{\mu^2}  \right] } Q(x_1,x_2) e^{-\ii\pi (k_1 x_1 \mu + k_2 x_2 \mu^2 )} \di x_1 \, \di x_2 \nonumber \\
&= \frac{1}{2} \mu^{3/2} \, \int_{ \left[ -\frac{1}{\mu},\frac{1}{\mu} \right] \times \left[ -\frac{1}{\mu^2},\frac{1}{\mu^2}  \right] }  \mu^2 \, q \left( \mu x_1, \mu^2 x_2 \right) e^{-\ii \pi (k_1 x_1 \mu + k_2 x_2 \mu^2 )} \di x_1 \, \di x_2 \nonumber \\
&\stackrel{\eqref{KdVr1}}{=} \frac{1}{2} \mu^{1/2} \int_{I} q(y) e^{-\ii \pi (k_1 y_1 + k_2 y_2) } \di y \; = \; \mu^{1/2} \hat{q}_k, \label{FourierRelQq}
\end{align}
and similarly
\begin{align} \label{FourierRelPp}
\hat{P}{\color{black}(k)} &= \mu^{-1/2} \hat{p}_k.
\end{align}
By using \eqref{EnNormMode}, \eqref{enkappa} and \eqref{FourierRel}-\eqref{FourierRelPp} we have
\begin{align*}
\cE_\kappa &\stackrel{\eqref{enkappa}}{=} \mu^{3} \, \frac{1}{2}  \sum_{ L=(L_1,L_2) \in \Z^2:\mu L_1, \mu^{2} L_2 \in 2\Z }  \left(|\hat{Q}{\color{black}(K+L)}|^2 + \omega_k^2 |\hat{P}{\color{black}(K+L)} |^2\right) \\
&\stackrel{\eqref{FourierRelQq}  }{=} \mu^4 \frac{1}{2}  \sum_{ L=(L_1,L_2) \in \Z^2:\mu L_1, \mu^2 L_2 \in 2\Z }  \left( |\hat{q}_{K+L}|^2 + \omega_k^2 \left| \frac{\hat{p}_{K+L}}{\mu} \right|^2 \right)
\end{align*}
for all $k$ such that $\kappa(k) = (\mu K_1,\mu^2 K_2)$, and this leads to \eqref{SpecEnNormModeKP}.
\end{proof}

\begin{proposition} \label{KPxietaProp}
Fix $\rho>0$ and $0 < \delta \ll 1$. Consider the normal form system \eqref{KPsys}, and define the Fourier coefficients of $(\xi,\eta)$ through the following formula
\begin{align}
\xi(y) &= \frac{1}{ 2 } \sum_{h \in \Z^2} \hat{\xi}_h \, e^{\ii \pi h \cdot y}, \label{FourierXiKP} \\
\eta(y) &= \frac{1}{ 2 } \sum_{h \in \Z^2} \hat{\eta}_h \, e^{\ii \pi h \cdot y}, \label{FourierEtaKP}.
\end{align}
{Suppose that} $(\xi,\eta) \in \cH^{\rho,0}$, and denote by $\cE_\kappa$ the specific energy of the normal mode with index $\kappa$ as defined in \eqref{kappa}-\eqref{enkappa}. Then for any positive $\mu$ sufficiently small
\begin{align} \label{KPcoeffxieta}
\left| \cE_\kappa - \mu^4 \frac{|\hat{\xi}_K|^2+|\hat{\eta}_K|^2}{2} \right| &\leq C \mu^{ 4+3/2 } \|(\xi,\eta)\|_{\cH^{\rho,0}}^2
\end{align}
for all $k$ such that $\kappa(k) = (\mu K_1,\mu^2 K_2)$ and $|K_1|+|K_2| \leq \frac{(2+\delta) |\log \mu|}{\rho}$. Moreover,
\begin{align} \label{KPSpecEnEst}
|\cE_\kappa| &\leq C \, \mu^{ 8 } \|(\xi,\eta)\|_{\cH^{\rho,0}}^2
\end{align}
for all $k$ such that $\kappa(k) = (\mu K_1,\mu^2 K_2)$ and $|K_1^2+K_2^2|^{1/2} > \frac{(2+\delta) |\log \mu|}{\rho}$, and $\cE_\kappa=0$ otherwise.
\end{proposition}

The proof of the above Proposition is deferred to Appendix \ref{ApprEstSec11}. \\

Now, consider the following system of uncoupled KP equations
\begin{align}
\xi_{\tau} & = - \frac{1}{24} \d_{y_1}^3\xi - \frac{1}{2} \d_{y_1}^{-1}\d_{y_2}^2\xi - \frac{\alpha }{2\sqrt{2}} \, \d_{y_1}(\xi^2), \label{KPsys1} \\
\eta_{\tau} & = \frac{1}{2} \d_{y_1}^{-1}\d_{y_2}^2\eta + \frac{1}{24} \d_{y_1}^3\eta + \frac{\alpha}{2\sqrt{2}} \, \d_{y_1}(\eta^2). \label{KPsys2}
\end{align}
and consider a solution $(\tau,y) \mapsto (\widetilde{\xi_a}(\tau,y),\widetilde{\eta_a}(\tau,y))$ such that it belongs to $\cH^{\rho,n}$, for some $n \geq 1$.

We consider the approximate solutions $(Q_a,P_a)$ of the {ETL} model \eqref{FPUeqc}
\begin{align}
  Q_a(t,y) &:= \frac{\mu^2}{\sqrt{2}} \left[ \widetilde{\xi_a}(\mu^2 \tau,y_1- \tau,y_2) + \widetilde{\eta_a}(\mu^2 \tau,y_1 + \tau,y_2)  \right] \label{Qappr2} \\
\d_{y_1}P_a(t,y) &:= \frac{\mu}{\sqrt{2}} \left[ \widetilde{\xi_a}(\mu^2 \tau,y_1- \tau,y_2) - \widetilde{\eta_a}(\mu^2 \tau,y_1 + \tau,y_2)  \right], \label{Pappr2} 
\end{align}
{where we made a little abuse of notation since the left-hand side depends on $t$ and right-hand side on $\tau$ that are related as $\tau=\mu t$.}

We need to compare the difference between the approximate solution \eqref{Qappr2}-\eqref{Pappr2} and the true solution of \eqref{2DETLeq}. Let us consider \eqref{2DETLeq}, and take an initial datum $(Q_0,P_0)$ with corresponding Fourier coefficients $(\hat{Q}_{0,k},\hat{P}_{0,k})$ given by \eqref{fourierQ}. {Observe that}
\begin{align} \label{InDatumHyp21}
\hat Q_{0,k} \neq 0 \; \; &\text{only if} \; \;  \kappa(k) = (\mu K_1, \mu^2 K_2). 
\end{align}
{ and that, as a consequence of the analiticity of $Q_0$ and $P_0$,} there exist $C$, $\rho >0 $ such that
\begin{align} \label{InDatumHyp22}
\frac{ |\hat{Q}_{0,k}|^2 + \omega_k^2 |\hat{P}_{0,k}|^2 }{N} &\leq C e^{-2\rho | (\kappa_1(k)/\mu , \kappa_2(k)/\mu^2 ) |}.
\end{align}
Moreover, we define an interpolating function for the initial datum $(Q_0,P_0)$ by
\begin{align*}
Q_0(y) &= \frac{1}{ \sqrt{N} } \sum_{K: \left( \mu^2 |K_1|^2 + \mu^{4} |K_2|^2 \right)^{1/2} = |\kappa(k)| \leq 1 } \hat{Q}_{0,k} e^{\ii \pi (\mu K_1 y_1 + \mu^2 K_2 y_2)},
\end{align*}
and similarly for $y \mapsto P_0(y)$. \\

{Next we show that we can exploit the analiticity of the solution of the approximating integrable PDEs to prove the vicinity between the approximate solution and the true solution.}

\begin{proposition} \label{ApprPropKP}
Consider \eqref{2DETLeq} with $\sigma=2$, and fix {$0 < \gamma < \frac{1}{2}$ }. Let us assume that the initial datum for \eqref{2DETLeq} {satisfies} \eqref{InDatumHyp21}-\eqref{InDatumHyp22}, and denote by $(Q_{\color{black}j}(t),P_{\color{black}j}(t))_{\color{black} j \in \mathbb{Z}^2_{N_1,N_2}}$ the corresponding solution. Consider the approximate solution $(\widetilde{\xi_a},\widetilde{\eta_a})$ with the corresponding initial datum. Assume that $(\widetilde{\xi_a},\widetilde{\eta_a}) \in \cH^{\rho,n}$ for some $\rho>0$ and for some $n \geq 1$ for all times, and fix $T_0>0$ and $0<\delta \ll 1$.

Then there exists $\mu_0=\mu_0( T_0, \|( \widetilde{\xi_a}(0),\widetilde{\eta_a}(0) )\|_{\cH^{\rho,n}} )$ such that, if $\mu < \mu_0$, we have that there exists $C>0$ such that
\begin{align} \label{KPapprDiscrCont}
\sup_j |Q_j(t) - Q_a(t,j)| + |P_j(t) - P_a(t,j)| &\leq C \mu^\gamma, \qquad {{\forall \, }} |t| \leq \frac{T_0}{ \mu^{3} },
\end{align}
where $(Q_a,P_a)$ are given by \eqref{Qappr2}-\eqref{Pappr2}. Moreover,
\begin{align} \label{LowModesApprKP}
\left| \cE_\kappa{\color{black}(t)} - \mu^{4} \frac{|\hat{\xi}_K{\color{black}(\tau)}|^2+|\hat{\eta}_K{\color{black}(\tau)}|^2}{2} \right| &\leq C \mu^{ 4+\gamma } \, \qquad {\forall |t| \leq \frac{T_0}{\mu^3}}
\end{align}
for all $k$ such that $\kappa(k) = (\mu K_1,\mu^2 K_2)$ and $|K_1|+|K_2| \leq \frac{(2+\delta) |\log \mu|}{\rho}$. Moreover,
\begin{align} \label{HighModesApprKP}
|\cE_\kappa{\color{black}(\tau)}| \leq \mu^{ 4+\gamma } \, \qquad {\forall |t| \leq \frac{T_0}{\mu^3}}
\end{align}
for all $k$ such that $\kappa(k) = (\mu K_1,\mu^2 K_2)$ and $|K_1|+|K_2| > \frac{(2+\delta) |\log \mu|}{\rho}$, and $\cE_\kappa=0$ otherwise.
\end{proposition}

The proof of the above Proposition is deferred to Appendix \ref{ApprEstSec12}.

\begin{proof}[Proof of Theorem \ref{KPrThm}]
First we prove \eqref{EnModesKP}. 

We consider an initial datum as in \eqref{KPData}; when passing to the continuous approximation \eqref{FPUeqc}, this initial datum corresponds to an initial data $(\xi_0,\eta_0) \in \cH^{\rho_0,n}$ for some $\rho_0>0$ and $n \geq 1$. By Theorem \ref{BouThm} the corresponding solution $(\xi(\tau),\eta(\tau))$ is analytic in a complex strip of width $\rho(t)$. Taking the minimum of such quantities one gets the coefficient $\rho$ appearing in the statement of Theorem \eqref{KPrThm}. Applying Proposition \ref{ApprPropKP}, we can deduce the corresponding result for the discrete model \eqref{2DETLeq} and the specific quantities \eqref{enkappa}.
\end{proof}

\subsection{The KdV regime} \label{apprKdVsubsec}

Similarly to Lemma \ref{EnSpecKPLemma}, Proposition \ref{KPxietaProp} we can prove the following results

\begin{lemma} \label{EnSpecKdVLemma}
Consider the lattice \eqref{HamQ} in the regime (KdV) and with interpolating function \eqref{KdVr1}. Then for a state corresponding to $(q,p)$ one has
\begin{align} \label{SpecEnNormModeKdV}
\cE_\kappa &= \frac{\mu^{4}}{2}  \sum_{ L=(L_1,L_2) \in \Z^2:\mu L_1, \mu^{\sigma} L_2 \in 2\Z } \left(|\hat{q}_{K+L}|^2 + \omega_k^2 \left| \frac{\hat{p}_{K+L}}{\mu} \right|^2\right), \qquad \forall k : \kappa(k) = (\mu K_1,\mu^{\sigma} K_2)
\end{align}
(where the $\omega_k$ are defined as in \eqref{FreqNormMode} and the $\cE_\kappa$ in \eqref{enkappa}), 
and $\cE_\kappa=0$ otherwise.
\end{lemma}

\begin{proof}
As in Lemma \ref{EnSpecKdVLemma} we introduce a $(2N_1+1,2N_2+1)$-periodic interpolating function for $Q_j$ and $P_j$. We denote $\hat{Q}(t,k)$ and $\hat{Q}_K(t)$ as in \eqref{FourierQcont} and \eqref{FourierRel}.

The relation between $\hat{Q}(k)$ and $\hat{q}_k$ can be deduced from \eqref{KdVr1} {and from the rescalings  $y_1=\mu x_1$, $y_2=\mu^\sigma x_2$},
\begin{align}
Q(j) &= \mu^2 q(\mu j_1,\mu^\sigma j_2); \nonumber \\
\hat{Q}({k}) &=  \frac{1}{2} \mu^{(\sigma+1)/2} \, \int_{ \left[ -\frac{1}{\mu},\frac{1}{\mu} \right] \times \left[ -\frac{1}{\mu^\sigma},\frac{1}{\mu^\sigma}  \right] } Q( x_1,x_2) e^{-\ii\pi (k_1 x_1 \mu + k_2 x_2 \mu^\sigma )} \di x_1 \, \di x_2 \nonumber \\
&= \frac{1}{2} \mu^{(\sigma+1)/2} \, \int_{ \left[ -\frac{1}{\mu},\frac{1}{\mu} \right] \times \left[ -\frac{1}{\mu^\sigma},\frac{1}{\mu^\sigma}  \right] }  \mu^2 \, q \left( \mu x_1, \mu^\sigma x_2 \right) e^{-\ii\pi (k_1 x_1 \mu + k_2 x_2 \mu^\sigma )} \di x_1 \, \di x_2 \nonumber \\
&\stackrel{\eqref{KdVr1}}{=} \frac{1}{2} \mu^{(3-\sigma)/2} \int_{I} q(	y) e^{-\ii \pi (k_1 y_1 + k_2 y_2) } \di y \; =  \; \mu^{(3-\sigma)/2} \hat{q}_k, \label{FourierRelQqKdV}
\end{align}
and similarly
\begin{align} \label{FourierRelPpKdV}
\hat{P}{\color{black}(k)} &= \mu^{(1-\sigma)/2} \hat{p}_k.
\end{align}
By using \eqref{EnNormMode}, \eqref{enkappa} and \eqref{FourierRelQqKdV}-\eqref{FourierRelPpKdV} we have
\begin{align*}
\cE_\kappa &\stackrel{\eqref{enkappa}}{=} \mu^{\sigma+1} \, \frac{1}{2}  \sum_{ L=(L_1,L_2) \in \Z^2:\mu L_1, \mu^{\sigma} L_2 \in 2\Z } \left( |\hat{Q}{(K+L)}|^2 + \omega_k^2 |\hat{P}{(K+L)} |^2 \right) \\
&\stackrel{ \eqref{FourierRelQqKdV},\eqref{FourierRelQqKdV}  }{=} \mu^{\sigma+1} \, \mu^{3-\sigma} \frac{1}{2}  \sum_{ L=(L_1,L_2) \in \Z^2:\mu L_1, \mu^{\sigma} L_2 \in 2\Z } \left( |\hat{q}_{K+L}|^2 + \omega_k^2 \left| \frac{\hat{p}_{K+L}}{\mu} \right|^2 \right)
\end{align*}
for all $k$ such that $\kappa(k) = (\mu K_1,\mu^{\sigma} K_2)$, and this leads to \eqref{SpecEnNormModeKdV}.
\end{proof}

\begin{proposition} \label{KdVxietaProp}
Fix $\rho>0$ and $0 <\delta \ll 1$. Consider the normal form system \eqref{KdVsys}, and define the Fourier coefficients of $(\xi,\eta)$ through the following formula
\begin{align}
\xi( y) &= \frac{1}{ 2 } \sum_{h \in \Z^2} \hat{\xi}_h e^{\ii \pi h \cdot y }, \label{FourierXiKdV} \\
\eta( y) &= \frac{1}{ 2 } \sum_{h \in \Z^2} \hat{\eta}_h e^{\ii\pi h \cdot y }, \label{FourierEtaKdV}.
\end{align}
{Suppose that} $(\xi,\eta) \in \cH^{\rho,0}$, and denote by $\cE_\kappa$ the specific energy of the normal mode with index $\kappa$ as defined in \eqref{kappa}-\eqref{enkappa}. Then for any positive $\mu$ sufficiently small 
\begin{align} \label{KdVcoeffxieta}
\left| \cE_\kappa - \mu^{4} \frac{|\hat{\xi}_K|^2+|\hat{\eta}_K|^2}{2} \right| &\leq C \mu^{ 4+3/2 } \|(\xi,\eta)\|_{\cH^{\rho,0}}^2
\end{align}
for all $k$ such that $\kappa(k) = (\mu K_1,\mu^\sigma K_2)$ and $|K_1|+|K_2| \leq \frac{(2+\delta) |\log \mu|}{\rho}$. Moreover,
\begin{align} \label{KdVSpecEnEst}
|\cE_\kappa| &\leq C \, \mu^{ 8 } \|(\xi,\eta)\|_{\cH^{\rho,0}}^2
\end{align}
for all $k$ such that $\kappa(k) = (\mu K_1,\mu^\sigma K_2)$ and $|K_1|+|K_2| > \frac{(2+\delta) |\log \mu|}{\rho}$, and $\cE_\kappa=0$ otherwise.
\end{proposition}

We defer the proof of the above Proposition to Appendix \ref{ApprEstSec11}. \\

Now, consider the following system of uncoupled KdV equations
\begin{align}
\xi_\tau &= - \frac{1}{24} \d_{y_1}^3 \xi - \frac{\alpha}{2\sqrt{2}} \d_{y_1}(\xi^2), \; \; \label{KdVsys1} \\
\eta_\tau &=  \frac{1}{24} \d_{y_1}^3 \eta + \frac{\alpha}{2\sqrt{2}} \d_{y_1}(\eta^2), \; \; \label{KdVsys2}
\end{align}
and consider a solution $(\tau,y) \mapsto (\widetilde{\xi_a}(\tau,y),\widetilde{\eta_a}(\tau,y))$ such that it belongs to $\cH^{\rho,n}$, for some $n \geq 1$.

We consider the approximate solutions $(Q_a,P_a)$ of the FPU model \eqref{FPUeqc} defined by formulas \eqref{Qappr2}-\eqref{Pappr2}.

We need to compare the difference between the approximate solution \eqref{Qappr2}-\eqref{Pappr2} and the true solution of \eqref{2DETLeq}. { Let us consider \eqref{2DETLeq}, and take an initial datum $(Q_0,P_0)$ with corresponding Fourier coefficients $(\hat{Q}_{0,k},\hat{P}_{0,k})$ given by \eqref{fourierQ}; observe that }
\begin{align} \label{InDatumHyp1}
Q_{0,k} \neq 0 \; \; &\text{only if} \; \;  \kappa(k) = (\mu K_1, \mu^\sigma K_2). 
\end{align}
{Since $Q_0$ and $P_0$ are analytic functions, }
 there exist $C$, $\rho >0 $ such that
\begin{align} \label{InDatumHyp2}
\frac{ |\hat{Q}_{0,k}|^2 + \omega_k^2 |\hat{P}_{0,k}|^2 }{(2N_1+1)(2N_2+1)} &\leq C e^{-2\rho | (\kappa_1(k)/\mu , \kappa_2(k)/\mu^\sigma ) |}.
\end{align}
Moreover, we define an interpolating function for the initial datum $(Q_0,P_0)$ by
\begin{align*}
Q_0(y,t) &= \frac{1}{{ \sqrt{N}} } \sum_{K: \left( \mu^2 |K_1|^2 + \mu^{2\sigma} |K_2|^2 \right)^{1/2} = |\kappa(k)| \leq 1 } \hat{Q}_{0,k}(t) e^{\ii\pi (\mu K_1 y_1 + \mu^\sigma K_2 y_2)},
\end{align*}
and similarly for $y \mapsto P_0(y)$.

\begin{proposition} \label{ApprPropKdV}
Consider \eqref{2DETLeq} with $\sigma > 2$ and $\gamma \geq 1$ such that $\sigma+2\gamma < { \min(4\sigma-5,7) }$. Let us assume that the initial datum satisfies \eqref{InDatumHyp1}-\eqref{InDatumHyp2}, and denote by $(Q(t),P(t))$ the corresponding solution. Consider the approximate solution $(\widetilde{\xi_a}(t,x), \widetilde{\eta_a}(t,x))$ with the corresponding initial datum. Assume that $(\widetilde{\xi_a},\widetilde{\eta_a}) \in \cH^{\rho,n}$ for some $\rho>0$ and for some $n \geq 1$ for all times, and fix $T_0>0$ and $0<\delta \ll 1$.

Then there exists $\mu_0=\mu_0( T_0, {\sigma}, \|( \widetilde{\xi_a}(0),\widetilde{\eta_a}(0) )\|_{\cH^{\rho,n}} )$ such that, if $\mu < \mu_0$, we have that there exists $C>0$ such that
\begin{align} \label{apprDiscrCont}
\sup_j |Q_j(t) - Q_a(t,j)| + |P_j(t) - P_a(t,j)| &\leq C \mu^\gamma, \qquad {\color{black}\forall} \, |t| \leq \frac{T_0}{ \mu^{3} },
\end{align}
where $(Q_a,P_a)$ are given by \eqref{Qappr2}-\eqref{Pappr2}. Moreover,
\begin{align} \label{LowModesAppr}
\left| \cE_\kappa{(t)} - \mu^{4} \frac{|\hat{\xi}_K{(\tau)}|^2+|\hat{\eta}_K{(\tau)}|^2}{2} \right| &\leq C \, \mu^{ 4+\gamma }  {, \, \qquad \forall |t| \leq \frac{T_0}{\mu^3} }
\end{align}
for all $k$ such that $\kappa(k) = (\mu K_1,\mu^\sigma K_2)$ and $|K_1|+|K_2| \leq \frac{(2+\delta) |\log \mu|}{\rho}$. Moreover,
\begin{align} \label{HighModesAppr}
|\cE_\kappa{(t)}| &\leq \mu^{ 4+\gamma }{, \, \qquad \forall |t| \leq \frac{T_0}{\mu^3}}
\end{align}
for all $k$ such that $\kappa(k) = (\mu K_1,\mu^\sigma K_2)$ and $|K_1|+|K_2| > \frac{(2+\delta)|\log \mu|}{\rho}$, and $\cE_\kappa=0$ otherwise.
\end{proposition}

We defer the proof to Appendix \ref{ApprEstSec12}.

\begin{remark} \label{assKdVrem}
The conditions {$\sigma+2\gamma < \min(4\sigma-5,7)$}, which, together with $\gamma>0$, implies the upper bound {$\sigma < 7$} found in the statement of Theorem \eqref{KdVrThm}, is the consequence of a technical condition which allows to estimate the error in the proof of Proposition \ref{ApprPropKdV} (see Claim 2, together with \eqref{EstTimeDer1}-\eqref{EstTimeDer2}).
\end{remark}

\begin{proof}[Proof of Theorem \ref{KdVrThm}]
First we prove \eqref{EnModesKdV}. 

We consider an initial datum as in \eqref{KdVData}; when passing to the continuous approximation \eqref{FPUeqc}, this initial datum corresponds to an initial data $(\xi_0,\eta_0) \in \cH^{\rho_0,n}$ for some $\rho_0>0$ and $n \geq 1$. By Theorem \ref{KapPosThm1} the corresponding sequence of gaps belongs to $\cH^{\rho_0,n}$, and that the solution $(\xi(\tau),\eta(\tau))$ is analytic in a complex strip of width $\rho(t)$. Taking the minimum of such quantities one gets the coefficient $\rho$ appearing in the statement of Theorem \eqref{KdVrThm}. Applying Proposition \ref{ApprPropKdV}, we can deduce the corresponding result for the discrete model \eqref{2DETLeq} and the specific quantities \eqref{enkappa}.
\\ 

Next, we prove \eqref{ApprEnModesKdV}. In order to do so, we exploit the Birkhoff coordinates $(x,y)$ introduced in Theorem \ref{KapPosThm2}; indeed, by rewriting the normal form system \eqref{KdVsys} in Birkhoff coordinates we get that every solution is almost-periodic in time. Now, let us introduce the quantities
\begin{align*}
E_K^{(1)} &:= \left| \hat{\xi}_K \right|^2, \\
E_K^{(2)} &:= \left| \hat{\eta}_K \right|^2,
\end{align*}
then $\tau \mapsto E_K^{(1)}(x(\tau),y(\tau))$ and $\tau \mapsto E_K^{(2)}(x(\tau),y(\tau))$ are almost-periodic. If we set $E_K := \frac{1}{2} \left( E_K^{(1)} + E_K^{(2)} \right)$, we can exploit \eqref{LowModesAppr} of Proposition \ref{ApprPropKdV} to translate the results in terms of the specific quantities $\cE_\kappa$, and we get the thesis.
\end{proof}

\subsection{The one-dimensional NLS regime} \label{appr1DNLSsubsec}

Let $\beta >0$ and let $I$ be as in \eqref{I}, we define the Fourier coefficients of the function $q:I \to \R$ by

\begin{align} \label{FourierqcontKG}
\hat{q}_{\color{black} k} \;&:=\; \frac{1}{2} \int_I q(y_1,y_2) \, e^{-\ii \pi ({k}_1 y_1 + {k}_2 y_2 ) } \di y_1 \, \di y_2,
\end{align}
and similarly for the Fourier coefficients of the function $p$. 

\begin{lemma} \label{EnSpec1DNLSLemma}
Consider the lattice \eqref{Ham2KGs} in the regime (1D NLS) and with interpolating function \eqref{1DNLSr}. Then for a state corresponding to $(q,p)$ one has
\begin{align} \label{SpecEnNormMode1DNLS}
\cE_\kappa \;&=\; \frac{\mu^{2}}{2}  \sum_{ L=(L_1,L_2) \in \Z^2:\mu L_1, \mu^\sigma L_2 \in 2\Z } \left( |\hat{p}_{K+L}|^2 + \omega_k^2 |\hat{q}_{K+L}|^2\right) \, , \qquad \forall k : \kappa(k) = (\mu K_1,\mu^\sigma K_2)
\end{align}
(where the $\omega_k$ are defined as in \eqref{FreqNormModeKG}), 
and $\cE_\kappa=0$ otherwise.
\end{lemma}

\begin{proof}
We introduce a $(2N_1+1,2N_2+1)$-periodic interpolating function for $Q_j$ and $P_j$. We denote $\hat{Q}(t,k)$ and $\hat{Q}_k(t)$ as in \eqref{FourierQcont} and \eqref{FourierRel}. By the interpolation property we obtain \eqref{FourierRel}.

The relation between $\hat{Q}(t,k)$ and $\hat{q}_k(t)$ can be deduced from \eqref{1DNLSr} {and from the rescalings $\tau=\mu t$, $y_1=\mu x_1$, $y_2=\mu^\sigma x_2$},
\begin{align}
Q(j) &= \mu q(\mu j_1,\mu^\sigma j_2); \nonumber \\
\hat{Q}{\color{black}(k)} &=  \frac{1}{2} \mu^{(\sigma+1)/2} \, \int_{ \left[ -\frac{1}{\mu},\frac{1}{\mu} \right] \times \left[ -\frac{1}{\mu^\sigma},\frac{1}{\mu^\sigma}  \right] } Q(x_1,x_2)\, e^{-\ii\pi (k_1 x_1 \mu + k_2 x_2 \mu^\sigma )} \di x_1 \, \di x_2 \nonumber \\
&\stackrel{\eqref{1DNLSr}}{=} \frac{1}{2} \mu^{(\sigma+1)/2} \, \int_{ \left[ -\frac{1}{\mu},\frac{1}{\mu} \right] \times \left[ -\frac{1}{\mu^\sigma},\frac{1}{\mu^\sigma}  \right] }  \mu \, q \left( \mu x_1, \mu^\sigma x_2 \right) e^{-\ii\pi (k_1 x_1 \mu + k_2 x_2 \mu^\sigma )} \di x_1 \, \di x_2 \nonumber \\
&= \frac{1}{2} \mu^{(1-\sigma)/2} \int_{I} q(y)\, e^{-\ii \pi (k_1 y_1 + k_2 y_2) } \di y \; = \; \mu^{(1-\sigma)/2} \hat{q}_k, \label{FourierRelQq2}
\end{align}
and similarly
\begin{align} \label{FourierRelPp2}
\hat{P}{\color{black}(k)} &= \mu^{(1-\sigma)/2} \hat{p}_k.
\end{align}
By using \eqref{EnNormModeKG}, \eqref{enkappa}, \eqref{FourierRel} and \eqref{FourierRelPp2} we have
\begin{align*}
\cE_\kappa &\stackrel{\eqref{enkappa}}{=} \mu^{\sigma+1} \, \frac{1}{2}  \sum_{ L=(L_1,L_2) \in \Z^2:\mu L_1, \mu^{\sigma} L_2 \in 2\Z } \left(|\hat{P}{\color{black}({K+L})}|^2 + \omega_k^2 |\hat{Q}{\color{black}({K+L})} |^2 \right)\\
&\stackrel{ \eqref{FourierRelQq2}  }{=} \mu^{\sigma+1} \, \mu^{1-\sigma} \frac{1}{2}  \sum_{ L=(L_1,L_2) \in \Z^2:\mu L_1, \mu^{\sigma} L_2 \in 2\Z } \left( |\hat{p}_{K+L}|^2 + \omega_k^2 |\hat{q}_{K+L}|^2 \right)
\end{align*}
for all $k$ such that $\kappa(k) = (\mu K_1,\mu^{\sigma} K_2)$, and this leads to \eqref{SpecEnNormMode1DNLS}.
\end{proof}

\begin{proposition} \label{1DNLSpsiProp}
Fix $\rho>0$ and $0 < \delta \ll 1$. Consider the normal form equation \eqref{1DimNLSeq}, and define the Fourier coefficients of $(\psi,\bar\psi)$ through the following formula
\begin{align}
\psi(y) &= \frac{1}{ 2 } \sum_{h \in \Z^2} \hat{\psi}_h  e^{\ii\pi h \cdot y }, \label{FourierPsi1D} 
\end{align}
{Suppose that} $(\psi,\bar\psi) \in \cH^{\rho,0}$, and denote by $\cE_\kappa$ the specific energy of the normal mode with index $\kappa$ as defined in \eqref{kappa}-\eqref{enkappa}. Then for any positive $\mu$ sufficiently small 
\begin{align} \label{1DNLScoeffpsi}
\left| \cE_\kappa - \mu^{2} \frac{|\hat{\psi}_K|^2}{2} \right| &\leq C \mu^{ 2+ {3/2} } \|(\psi,\bar\psi)\|_{\cH^{\rho,0}}^2
\end{align}
for all $k$ such that $\kappa(k) = (\mu K_1,\mu^\sigma K_2)$ and $|K_1|+|K_2| \leq \frac{(2+\delta) |\log \mu|}{\rho}$. Moreover,
\begin{align} \label{1DNLSSpecEnEst}
|\cE_\kappa| &\leq C \, \mu^{ {6} } \|(\psi,\bar\psi)\|_{\cH^{\rho,1}}^2
\end{align}
for all $k$ such that $\kappa(k) = (\mu K_1,\mu^\sigma K_2)$ and $|K_1|+|K_2| > \frac{(2+\delta) |\log \mu|}{\rho}$, and $\cE_\kappa=0$ otherwise.
\end{proposition}

We defer the proof of the above Proposition to Appendix \ref{ApprEstSec21}. \\

Now, consider the normal form equation, namely the following cubic defocusing one-dimensional NLS
\begin{align} \label{1DimNLSeq1}
-\ii \psi_{\tau} &= - \; \d_{y_1}^2\psi + \frac{3 \beta}{4} \, |\psi|^2\psi.
\end{align}
and consider a solution $(\widetilde{\psi_a},\bar{\widetilde{\psi_a}})$ such that it belongs to $\cH^{\rho,n}$, for some $n > 0$.

We consider the approximate solutions $(Q_a,P_a)$ of the KG lattice \eqref{Ham2KGs} (in the following $\tau=\mu^2 t$)
\begin{align}
Q_a({\color{black}t},y) &:= \frac{\mu}{\sqrt{2}} \left[ e^{\ii\tau} \widetilde{\psi_a}(\tau,y_1,y_2) + e^{-\ii\tau} \bar{\widetilde{\psi_a}}(\tau,y_1,y_2)  \right] \, , \label{Qappr3} \\
P_a({\color{black}t},y) &:= \frac{\mu}{\sqrt{2} \ii} \left[ e^{\ii\tau} \widetilde{\psi_a}(\tau,y_1,y_2) + e^{-\ii\tau} \bar{\widetilde{\psi_a}}(\tau,y_1,y_2)  \right] \, . \label{Pappr3}
\end{align}

We need to compare the difference between the approximate solution \eqref{Qappr3}-\eqref{Pappr3} and the true solution of \eqref{Ham2KGs}. Let us consider \eqref{Ham2KGs}, and take an initial datum $(Q_0,P_0)$ with corresponding Fourier coefficients $(\hat{Q}_{0,k},\hat{P}_{0,k})$ given by \eqref{fourierQ}; {observe that}
\begin{align} \label{InDatumHyp31}
\hat Q_{0,k} \neq 0 \; \; &\text{only if} \; \;  \kappa(k) = (\mu K_1, \mu^\sigma K_2). 
\end{align}
{Since $P_0$ and $Q_0$ are analytic functions,}  there exist $C$, $\rho >0 $ such that
\begin{align} \label{InDatumHyp32}
\frac{ |\hat{P}_{0,k}|^2 + \omega_k^2 |\hat{Q}_{0,k}|^2 }{N} &\leq C e^{-2\rho | (\kappa_1(k)/\mu , \kappa_2(k)/\mu^\sigma ) |}.
\end{align}
Moreover, we define an interpolating function for the initial datum $(Q_0,P_0)$ by 
\begin{align*}
Q_0({\color{black}t},y) &= \frac{1}{ \color{black}\sqrt{N} } \sum_{K: \left( \mu^2 |K_1|^2 + \mu^{2\sigma} |K_2|^2 \right)^{1/2} = |\kappa(k)| \leq 1 } \hat{Q}_{0,{K}}({\color{black}t})\, e^{\ii\pi (\mu K_1 y_1 + \mu^\sigma K_2 y_2)},
\end{align*}
and similarly for $y \mapsto P_0(y)$.

\begin{proposition} \label{ApprProp1DNLS}
Consider \eqref{Ham2KGs} with $\sigma > 1$ and $\gamma > 0$ such that $\sigma+2\gamma < {\min(4\sigma-1,7)}$. Let us assume that the initial datum satisfies \eqref{InDatumHyp31}-\eqref{InDatumHyp32}, and denote by $(Q_{\color{black}j}(t),P_{\color{black}j}(t))_{\color{black}j \in \mathbb{Z}_{N_1,N_2}}$ the corresponding solution. Consider the approximate solution $(\widetilde{\psi_a}(t,x), \bar{\widetilde{\psi_a}}(t,x))$ with the corresponding initial datum. Assume that $(\widetilde{\psi_a},\bar{\widetilde{\psi_a}}) \in \cH^{\rho,n}$ for some $\rho>0$ and for some $n \geq 0$ for all times, and fix $T_0>0$ and $0<\delta \ll 1$.

Then there exists $\mu_0=\mu_0( T_0, {\sigma,} \|( \widetilde{\psi_a}(0),\bar{\widetilde{\psi_a}}(0) )\|_{\cH^{\rho,n}} )$ such that, if $\mu < \mu_0$, we have that there exists $C>0$ such that
\begin{align} \label{apprDiscrCont3}
\sup_j |Q_j(t) - Q_a(t,j)| + |P_j(t) - P_a(t,j)| &\leq C \mu^\gamma, \qquad {\color{black} \forall } \, |t| \leq \frac{T_0}{ \mu^{2} },
\end{align}
where $(Q_a,P_a)$ are given by \eqref{Qappr3}-\eqref{Pappr3}. Moreover,
\begin{align} \label{LowModesApprNLS}
\left| \cE_\kappa{(t)} - \mu^{2} \frac{|\hat{\psi}_K{(\tau)}|^2}{2} \right| &\leq C \, \mu^{ 2+\gamma } \, , \qquad {\color{black}|t| \leq \frac{T_0}{ \mu^{2} }}
\end{align}
for all $k$ such that $\kappa(k) = (\mu K_1,\mu^\sigma K_2)$ and $|K_1|+|K_2| \leq \frac{(2+\delta) |\log \mu|}{\rho}$. Moreover,
\begin{align} \label{HighModesApprNLS}
|\cE_\kappa{(t)}| &\leq \mu^{ 2+\gamma } \, , \qquad {\color{black}|t| \leq \frac{T_0}{ \mu^{2} }}
\end{align}
for all $k$ such that $\kappa(k) = (\mu K_1,\mu^\sigma K_2)$ and $|K_1|+|K_2| > \frac{(2+\delta)|\log \mu|}{\rho}$, and $\cE_\kappa=0$ otherwise.
\end{proposition}

We defer the proof to Appendix \ref{ApprEstSec22}.

\begin{remark} \label{assNLSrem}
The condition $\sigma+2\gamma < {\min(4\sigma-1,7)}$, which, together with $\gamma>0$, implies the upper bound $\sigma < 7$ found in the statement of Theorem \eqref{1DNLSrThm}, is the consequence of a technical condition which allows to estimate the error in the proof of Proposition \ref{ApprProp1DNLS} (see Claim 2, together with \eqref{EstTimeDer31}-\eqref{EstTimeDer33}).
\end{remark}

\begin{proof}[Proof of Theorem \ref{1DNLSrThm}]
First we prove \eqref{EnModes1DNLS}. 

We consider an initial datum as in \eqref{1DNLSData}; when passing to the continuous approximation \eqref{HamKGsc}, this initial datum corresponds to an initial data $(\xi_0,\eta_0) \in \cH^{\rho_0,n}$. By Theorem \ref{GrebKapThm1} the corresponding sequence of gaps belongs to $\cH^{\rho_0,n}$, and that the solution $(\xi(\tau),\eta(\tau))$ is analytic in a complex strip of width $\rho(t)$. Taking the minimum of such quantities one gets the coefficient $\rho$ appearing in the statement of Theorem \eqref{1DNLSrThm}. Applying Proposition \ref{ApprProp1DNLS}, we can deduce the corresponding result for the discrete model \eqref{2DKGseq} and the specific quantities \eqref{enkappa}.
\\ 

Next, we prove \eqref{ApprEnModes1DNLS}. In order to do so, we exploit the Birkhoff coordinates $(x,y)$ introduced in Theorem \ref{GrebKapThm2}; indeed, by rewriting the normal form system \eqref{1DimNLSeq} in Birkhoff coordinates we get that every solution is almost-periodic in time. Now, let us introduce the quantity
\begin{align*}
E_K &:= \frac{1}{2} \, \left| \hat{\psi}_K \right|^2, 
\end{align*}
then $\tau \mapsto E_K(x(\tau),y(\tau))$ is almost-periodic. Hence we can exploit \eqref{LowModesApprNLS} of Proposition \ref{ApprProp1DNLS} to translate the results in terms of the specific quantities $\cE_\kappa$, and we get the thesis.
\end{proof}

\begin{appendix}

\section{Proof of Lemma \ref{NFest}} \label{BNFest}

This appendix is devoted to the proof of the Lemma \ref{NFest}, which is a key step to normalize the system \eqref{truncsys}. Its proof is an adaptation of Theorem 4.4 in \cite{bambusi1999nekhoroshev} and it is based on the method of Lie transform, briefly recalled in the following. Throughout this Section, we consider  $s \geq s_1$ and $\rho \geq 0$ to be fixed quantities.

Given an auxiliary function $\chi$ analytic on $\cH^{\rho,s}$, 
we consider the auxiliary differential equation
\begin{align} \label{auxDE}
\dot \zeta &= X_\chi(\zeta)
\end{align}
and denote by $\Phi^t_\chi$ its flow at time $t$. 

\begin{lemma} \label{cauchylemma}
Let $\chi$ and its vector field be analytic in $\cB_{\rho,s}(R)$. 
Fix ${d}<R$, and assume that 
\begin{align*}
\sup_{\zeta \in \cB_{\rho,s}(R)} \|X_\chi(\zeta)\|_{\cH^{\rho,s}} &\leq {d}.
\end{align*}
Then, if we consider the time-$t$ flow $\Phi^t_\chi$ of $X_\chi$ we have that 
for $|t| \leq 1$ 
\begin{align*}
\sup_{\zeta \in \cB_{\rho,s}( R-{d} )} \|\Phi^t_\chi(\zeta)-\zeta\|_{\cH^{\rho,s}} &\leq \sup_{\zeta \in  \cB_{\rho,s}(R)} \|X_\chi(\zeta)\|_{\cH^{\rho,s}}.
\end{align*}
\end{lemma}

\begin{definition}
The map $\Phi_\chi := \Phi^1_\chi$ is called the \emph{Lie transform} 
generated by $\chi$.
\end{definition}

Given $\Ham G$ analytic on $\cH^{\rho,s}$, let us consider the differential equation
\begin{align} \label{orDE}
\dot \zeta &= X_{\Ham G}(\zeta),
\end{align}
where by $X_{\Ham G}$ we denote the vector field of $\Ham G$. Now define
\begin{align*}
\Phi_\chi^\ast \Ham G(\tilde\zeta) &:= \Ham G \circ \Phi_\chi(\tilde\zeta).
\end{align*}
By exploiting the fact that $\Phi_\chi$ is a canonical transformation, we have that in the new variable $\tilde\zeta$ defined by $\zeta=\Phi_\chi(\tilde\zeta)$ equation \eqref{orDE} is equivalent to
\begin{align} \label{pullbDE}
\dot{\tilde\zeta} &= X_{ \Phi_\chi^\ast \Ham G }(\tilde\zeta).
\end{align}

Using the relation
\begin{align}\label{eq:CompositionRelation}
\frac{\di}{\di t} \Phi_\chi^\ast \Ham G &= \Phi_\chi^\ast \{\chi, \Ham G\}, 
\end{align}
and the Poisson bracket formalism $\{\Ham G_1,\Ham G_2\}(\zeta):= \di \Ham G_1(\zeta) [ X_{\Ham G_2}(\zeta)]$ we formally get
\begin{equation} \label{lieseries}
\begin{split}
\Phi^\ast_\chi \Ham G &= \sum_{\ell=0}^\infty \Ham G_\ell, \\
\Ham G_0 &:= \Ham G, \\
\Ham G_{\ell} &:= \frac{1}{\ell} \{\chi,\Ham G_{\ell-1}\}, \; \; \ell \geq 1.
\end{split}
\end{equation}

In order to estimate the vector field of the terms appearing in \eqref{lieseries}, we exploit the following results

\begin{lemma}\label{lem:liebrest}
Let $R>0$, and assume that $\chi$, $\Ham G$ are analytic on $\cB_{\rho,s}(R)$ as well as their vector fields. Then, for any $d \in (0,R)$ we have that 
$\{\chi,\Ham G\}$ is analytic on $\cB_{\rho,s}(R-d)$, and
\begin{align} \label{liebrest}
\sup_{\zeta \in \cB_{\rho,s}(R-d)} \|X_{ \{\chi,\Ham G\} }(\zeta)\|_{\cH^{\rho,s}} &\leq \frac{2}{d} \; \left( \sup_{\zeta \in \cB_{\rho,s}(R)} \|X_{\chi} (\zeta)\|_{\cH^{\rho,s}} \right) \; \left( \sup_{\zeta \in \cB_{\rho,s}(R)} \|X_{ \Ham G }(\zeta)\|_{\cH^{\rho,s}} \right).
\end{align}
\end{lemma}

\begin{proof}
Observe that 
\begin{align*}
\|X_{ \{\chi, \Ham G\} }(\zeta)\|_{\cH^{\rho,s}} &= \| \di X_\chi(\zeta) \; X_{\Ham G}(\zeta) - \di X_{\Ham G}(\zeta) \; X_\chi(\zeta) \|_{\cH^{\rho,s}} \\
&\leq \| \di X_\chi(\zeta) \; X_{\Ham G}(\zeta) \|_{\cH^{\rho,s}} + \| \di X_{\Ham G}(\zeta) \; X_\chi(\zeta) \|_{\cH^{\rho,s}},
\end{align*}
and since for any $d \in (0,R)$ Cauchy inequality gives
\begin{align*}
\sup_{\zeta \in \cB_{\rho,s}(R-d)} \| \di X_\chi(\zeta) \|_{\cH^{\rho,s}\to\cH^{\rho,s}} &\leq \frac{1}{d} \; \sup_{\zeta \in \cB_{\rho,s}(R)} \| X_\chi(\zeta) \|_{\cH^{\rho,s}},
\end{align*}
we finally get
\begin{align*}
\sup_{\zeta \in \cB_{\rho,s}(R-d)} \| \di X_\chi(\zeta) \; X_{\Ham G}(\zeta) \|_{\cH^{\rho,s}} &\leq \frac{1}{d} \; \left( \sup_{\zeta \in \cB_{\rho,s}(R)} \|X_{\chi} (\zeta)\|_{\cH^{\rho,s}} \right) \; \left( \sup_{\zeta \in \cB_{\rho,s}(R)} \|X_{ \Ham G }(\zeta)\|_{\cH^{\rho,s}} \right)\, .
\end{align*}
With a similar estimate for the other term we obtain the thesis.
\end{proof}

\begin{lemma}\label{lem:Liesrest}
Let $R>0$, and assume that $\chi$, $\Ham G$ are analytic on $\cB_{\rho,s}(R)$ 
as well as their vector fields. 
Let $\ell \geq 1$, and consider $\Ham G_\ell$ as defined in \eqref{lieseries};
for any $d \in (0,R)$,  $\Ham G_\ell$ is analytic on $\cB_{\rho,s}(R-d)$ 
as well as it vector field, and
\begin{align} \label{lieserest}
\sup_{\zeta \in \cB_{\rho,s}(R-d)} \|X_{ \Ham G_\ell }(\zeta)\|_{\cH^{\rho,s}} &\leq \left( \frac{2e}{d} \sup_{\zeta \in \cB_{\rho,s}(R)} \|X_\chi(\zeta)\|_{\cH^{\rho,s}} \right)^\ell  \; \sup_{\zeta \in \cB_{\rho,s}(R)} \|X_{\Ham G}(\zeta)\|_{\cH^{\rho,s}}.
\end{align}
\end{lemma}

\begin{proof}
Fix {$\ell \geq 1$}. We look for a sequence $C^{(\ell)}_m$ such that
\begin{align*}
\sup_{\zeta \in \cB_{\rho,s}(R-m \, d/\ell)} \|X_{\Ham G_m}(\zeta)\|_{\cH^{\rho,s}} &\leq C^{(\ell)}_m, \; \; 
\forall m \leq \ell.
\end{align*}
Lemma \ref{lem:liebrest} ensures that the following sequence satisfies this property.
\begin{align*}
C^{(\ell)}_0 &:= \sup_{\zeta \in \cB_{\rho,s}(R)} \|X_{\Ham G}(\zeta)\|_{\cH^{\rho,s}}, \\
C^{(\ell)}_m &= \frac{2 \ell}{d m} C^{(\ell)}_{m-1} \, \sup_{\zeta \in \cB_{\rho,s}(R)} \|X_\chi(\zeta)\|_{\cH^{\rho,s}}.
\end{align*}
One has
\begin{align*}
C^{(\ell)}_\ell &= \frac{1}{\ell!} \left( \frac{2\ell}{d} \sup_{\zeta \in \cB_{\rho,s}(R)} \|X_\chi(\zeta)\|_{\cH^{\rho,s}} \right)^\ell  \; \sup_{\zeta \in  \cB_{\rho,s}(R)} \|X_{\Ham G}(\zeta)\|_{\cH^{\rho,s}},
\end{align*}
and by using the inequality $\ell^\ell < \ell! e^\ell$ one obtains the estimate \eqref{lieserest}.
\end{proof}

Before stating the next Lemma, we point out that the Poisson tensor $\Omega_2^{-1}$, obtained by inversion from the associated symplectic form $\Omega_2$ in \eqref{eq:SymplecticForms}, is not a bounded operator on $\mathcal{H}^{\rho,s}$. We thus have to weaken the hypothesis of Theorem 4.4 in \cite{bambusi1999nekhoroshev}; indeed, we just assume that
\begin{align*}
\Vert \Omega^{-1} f \Vert_{\mathcal{H}^{\rho,s}} &\leq \Vert f \Vert_{ \mathcal{H}^{\rho,s+1}} \, . 
\end{align*}
This property is satisfied by both $\Omega_1^{-1}$ and $\Omega_2^{-1}$.

\begin{lemma}\label{lem:VectorFieldCanonicalTransformation}
Let $\chi$ and $\Ham F$ be analytic on $\cB_{\rho,s}(R)$ as well as their vector fields. 
Fix $d \in (0,R)$, and assume also that
\begin{align*}
 \sup_{\zeta \in \cB_{\rho,s}(R)} \|X_\chi(\zeta)\|_{\cH^{\rho,s}} \leq d/3 \, .
\end{align*}
Then for $|t| \leq 1$
\begin{align}
\sup_{\zeta \in \cB_{\rho,s}(R-d)} \|X_{ (\Phi^t_\chi)^\ast \Ham F -  \Ham F }(\zeta)\|_{\cH^{\rho,s}} 
{\leq} 
\frac{9}{d} \, \sup_{\zeta \in \cB_{\rho,s}(R)} \|X_\chi(\zeta)\|_{\cH^{\rho,s}} \,
\sup_{\zeta \in \cB_{\rho,s}(R)} \|X_{\Ham F}(\zeta)\|_{\cH^{\rho,s}}. \label{vfest}
\end{align}
\end{lemma}

\begin{proof}
Since the bound on the norm of $X_\chi$ implies that $\Phi^t_\chi(\zeta) \in \cB_{\rho,s}(R)$ when $\zeta \in \cB_{\rho,s}(R-d/3)$, using Cauchy inequality and Lemma \ref{cauchylemma}
\[
	\begin{split}
		\sup_{\zeta \in \cB_{\rho,s}(R-d)} \Vert \di \Phi_\chi^{-t}(\Phi^t_\chi(\zeta))- \mathrm{id} \Vert_{\cH^{\rho,s} \to \cH^{\rho,s}} & \leq \sup_{\zeta \in \cB_{\rho,s}(R-2d/3)} \Vert \di \Phi_\chi^{-t}(\zeta) - \mathrm{id} \Vert_{\cH^{\rho,s} \to \cH^{\rho,s}} \\
		& \leq \frac{3}{d}\sup_{\zeta \in \cB_{\rho,s}(R-d/3)} \Vert \Phi_\chi^{-t}(\zeta)-\zeta \Vert_{\cH^{\rho,s}} \\
		& \leq \frac{3}{d} \sup_{\zeta \in \cB_{\rho,s}(R)} \Vert X_\chi(\zeta) \Vert_{\cH^{\rho,s}}
	\end{split}
\]

Since $\Phi_\chi^t$ is a canonical transformation, a direct computation shows
\begin{align*}
\Omega^{-1} \di ( \Ham F \circ \Phi_\chi^t)(\zeta) &= ( \di \Phi_{\chi}^{-t} ( \Phi_\chi^t(\zeta)) - \mathrm{id}) \Omega^{-1} \di \Ham F (\Phi^t_\chi)+ \Omega^{-1} \di \Ham F (\Phi_\chi^t(\zeta))
\end{align*}
whence 
\[
	\begin{split}
		\sup_{\zeta \in \cB_{\rho,s}(R-d)} \Vert X_{(\Phi_\chi^t)^* \Ham F- \Ham F}(\zeta) \Vert_{\cH^{\rho,s}} &= \sup_{\zeta \in \cB_{\rho,s}(R-d)} \Vert\Omega^{-1} \di (\Ham F(\Phi_\chi^t(\zeta))-\Ham F(\zeta)) \Vert_{\cH^{\rho,s}} \\
		& \leq \sup_{\zeta \in \cB_{\rho,s}(R-d)} \Vert ( \di \Phi_{\chi}^{-t} ( \Phi_\chi^t(\zeta)) - \mathrm{id}) \Omega^{-1} \di \Ham F (\Phi^t_\chi)+ \Omega^{-1} \di ( \Ham F (\Phi_\chi^t(\zeta))-\Ham F(\zeta) ) \Vert_{\cH^{\rho,s}} \\
		&\leq \sup_{\zeta \in \cB_{\rho,s}(R-d)} \Vert \di \Phi_{\chi}^{-t}(\Phi_\chi^t(\zeta))- \mathrm{id} \Vert_{ \cH^{\rho,s} \to \cH^{\rho,s} } \sup_{\zeta \in \cB_{\rho,s}(R-d)} \Vert X_{\Ham F}(\Phi^t_\chi(\zeta))\Vert_{\cH^{\rho,s}}\\
		& + \sup_{\zeta \in \cB_{\rho,s}(R-d)} \Vert X_{\Ham F}(\Phi_\chi^t(\zeta))-X_{\Ham F}(\zeta) \Vert_{\cH^{\rho,s}} \\
		&\leq \frac{3}{d} \sup_{\zeta \in \cB_{\rho,s}(R)} \Vert X_\chi(\zeta) \Vert_{\cH^{\rho,s}} \sup_{\zeta \in \cB_{\rho,s}(R)} \Vert X_{\Ham F}(\zeta) \Vert_{\cH^{\rho,s}}\\ &+ \sup_{\zeta \in \cB_{\rho,s}(R-d)} \left\Vert \int_0^t [X_\chi,X_{\Ham F}](\Phi_\chi^s(\zeta)) \di s \right\Vert_{\cH^{\rho,s}} \, .
	\end{split}
\]
To estimate the last term we use Cauchy inequality
\[
	\begin{split}
		\sup_{\zeta \in \cB_{\rho,s}(R-d)} \left\Vert \int_0^t [X_\chi,X_{\Ham F}](\Phi_\chi^s(\zeta)) \, \di s \right\Vert_{\cH^{\rho,s}} & \leq 2 \sup_{\zeta \in \cB_{\rho,s}(R-2d/3)} \Vert [X_\chi,X_{\Ham F}](\zeta) \Vert_{\cH^{\rho,s}} \\
		&\leq \frac{6}{d} \sup_{\zeta \in \cB_{\rho,s}(R)} \Vert X_\chi(\zeta) \Vert_{\cH^{\rho,s}} \sup_{\zeta \in \cB_{\rho,s}(R)} \Vert X_{\Ham F}(\zeta) \Vert_{\cH^{\rho,s}} \, .
	\end{split}
\]
Then the thesis follows.
\end{proof}

\begin{lemma} \label{homeqlemma}
Assume that $\Ham G$ is analytic on $\cB_{\rho,s}(R)$ as well as its vector field, 
and that $\Ham{h}_0$ satisfies (PER). 
Then there exists $\chi$ analytic on $\cB_{\rho,s}(R)$ and $\Ham Z$ analytic 
on $\cB_{\rho,s}(R)$ with $\Ham Z$ in normal form, namely $\{\Ham{h}_0,\Ham Z\}=0$, such that
\begin{align} \label{homeq}
\{ \chi,\Ham{h}_0 \} \; + \; \Ham G \; &= \; \Ham Z.
\end{align}
Such $\Ham Z$ and $\chi$ are given explicitly by
\begin{align}\label{eq:ExplicitZ}
\Ham Z(\zeta)\; = \; \frac{1}{T} \int_0^T \Ham G(\Phi^t_{\Ham{h}_0}(\zeta)) \, \di t \, ,
\end{align}
\begin{align}\label{eq:ExplicitChi}
\chi(\zeta) &= 
\frac{1}{T} \int_0^T t \, \left[ \Ham Z(\Phi^t_{\Ham{h}_0}(\zeta))-\Ham G(\Phi^t_{\Ham{h}_0}(\zeta)) \right] \, \di t \, .
\end{align}
Furthermore, we have that the vector fields of $\chi$ and $\Ham Z$ are analytic on $\cB_{\rho,s}(R)$, and satisfy 
\begin{align} 
\sup_{\zeta \in \cB_{\rho,s}(R)} \|X_{\Ham Z}(\zeta)\|_{\cH^{\rho,s}} &\leq  \sup_{\zeta \in \cB_{\rho,s}(R)} \|X_{\Ham G}(\zeta)\|_{\cH^{\rho,s}}, \nonumber \\
\sup_{\zeta \in \cB_{\rho,s}(R)} \|X_\chi(\zeta)\|_{\cH^{\rho,s}} &\leq  2T  \sup_{\zeta \in \cB_{\rho,s}(R)} \|X_{\Ham G}(\zeta)\|_{\cH^{\rho,s}} \label{vfhomeq}.
\end{align}
\end{lemma}
\begin{proof}
We check directly that the solution of \eqref{homeq} is \eqref{eq:ExplicitChi}. Indeed,
\begin{align*}
\{ \chi,\Ham{h}_0 \}(\zeta) &= \frac{\di}{\di s} \Big|_{s=0} \chi(\Phi^s_{\Ham{h}_0}(\zeta)) \\
&= \frac{1}{T} \int_0^{T} t \frac{\di}{\di s} \Big|_{s=0} \left[ \Ham Z(\Phi^{t+s}_{\Ham{h}_0}(\zeta))-\Ham G(\Phi^{t+s}_{\Ham{h}_0}(\zeta)) \right] \di t \\
&= \frac{1}{T} \int_0^{T} t \frac{\di}{\di t} \left[ \Ham Z(\Phi^{t}_{\Ham{h}_0}(\zeta))-\Ham G(\Phi^{t}_{\Ham{h}_0}(\zeta)) \right] \di t \\
&= \frac{1}{T} \left[ t \Ham Z(\Phi^{t}_{\Ham{h}_0}(\zeta))- t \Ham G(\Phi^{t}_{\Ham{h}_0}(\zeta)) \right]_{t=0}^{T} - \frac{1}{T} \int_0^{T} \left[ \Ham Z(\Phi^{t}_{\Ham{h}_0}(\zeta))-\Ham G(\Phi^{t}_{\Ham{h}_0}(\zeta)) \right] \di t \\
&= \Ham Z(\zeta)-\Ham G(\zeta).
\end{align*}
In the last step we used the explicit expression of $\Ham Z$ provided in \eqref{eq:ExplicitZ}. Finally, the first estimate in \eqref{vfhomeq} follows from the explicit expression of $\Ham Z$ in \eqref{eq:ExplicitZ} while for the second estimate we write explicitly the vector field $X_\chi$: 
\begin{align*}
X_\chi(\zeta) &= 
\frac{1}{T} \int_0^T t \, D\Phi^{-t}_{\Ham{h}_0}(\Phi_{\Ham{h}_0}^t(\zeta)) \circ X_{\Ham Z-\Ham G}(\Phi^t_{\Ham{h}_0}(\zeta)) \, \di t \, .
\end{align*}
Hypothesis (PER) guarantees that $\Phi_{\Ham{h}_0}^t$ as well as its derivatives and the inverses are uniformly bounded as operators from $\cH^{\rho,s}$ into itself. Moreover, for any $t \in \mathbb{R}$, the map $\zeta \mapsto \Phi_{\Ham{h}_0}^t(\zeta)$ is a diffeomorphism of $\cB_{\rho,s}(R)$ into itself. {Using the fact that $\Phi_{\Ham h_0}^t$ is an isometry, we have}
\[
	\begin{split}
		\sup_{\zeta \in \cB_{\rho,s}(R)} \Vert X_\chi(\zeta) \Vert_{\cH^{\rho,s}} &\leq T  \sup_{\zeta \in \cB_{\rho,s}(R)} \left(\Vert X_{\Ham Z}(\zeta) \Vert_{\cH^{\rho,s}} + \Vert X_{\Ham G}(\zeta) \Vert_{\cH^{\rho,s}} \right) \leq 2T \sup_{\zeta \in \cB_{\rho,s}(R)} \Vert X_{\Ham G}(\zeta) \Vert_{\cH^{\rho,s}}
	\end{split}
\]
where in the last step we used the first inequality in \eqref{vfhomeq}. 
\end{proof}

\begin{lemma}
Assume that $\Ham G$ and its vector fields are analytic on $\cB_{\rho,s}(R)$, and that $\Ham{h}_0$ satisfies (PER). 
Let $\chi$ and its vector field be analytic on $\cB_{\rho,s}(R)$, and assume that $\chi$ solves \eqref{homeq}. For any $\ell\geq 1$ denote by $h_{0,\ell}$ 
the functions defined recursively as in \eqref{lieseries} from $\Ham{h}_0$.
Then for any $d \in (0,R)$ one has that $h_{0,\ell}$ and its vector field 
are analytic on $\cB_{\rho,s}(R-d)$, and
\begin{align} \label{lieseriesh0}
\sup_{\zeta \in \cB_{\rho,s}(R-d)} \|X_{ \Ham h_{0,\ell} }(\zeta)\|_{\cH^{\rho,s}} &\leq 
2 \sup_{\zeta \in \cB_{\rho,s}(R)} \|X_{\Ham G}(\zeta)\|_{\cH^{\rho,s}} 
\left( \frac{9}{d} \, \sup_{\zeta \in  \cB_{\rho,s}(R)} \|X_\chi(\zeta)\|_{\cH^{\rho,s}} \right)^\ell.
\end{align}
\end{lemma}

\begin{proof}
By using \eqref{homeq} one gets that $\Ham h_{0,1} =\Ham Z- \Ham G$ is analytic on $\cB_{\rho,s}(R)$. 
Then by exploiting \eqref{vfest} one gets the result.
\end{proof}

\begin{lemma} \label{Poissonlemma}
Assume that $G$ and its vector field are analytic on $\cB_{\rho,s}(R)$, 
and that $\Ham{h}_0$ satisfies PER. 
Let $\chi$ be the solution of \eqref{homeq}, denote by $\Phi^t_\chi$ 
the flow of the Hamiltonian vector field associated to $\chi$ 
and by $\Phi_\chi$ the corresponding time-one map. Moreover, denote by 
\begin{align*}
\cF(\zeta) &:= \Ham{h}_0(\Phi_\chi(\zeta)) - \Ham{h}_0(\zeta) - \{\chi,\Ham{h}_0\}(\zeta).
\end{align*}
Let $d < R$, and assume that
\begin{align*}
\sup_{\zeta \in \cB_{\rho,s}(R)} \|X_\chi(\zeta)\|_{\cH^{\rho,s}} \leq d/3 \, .
\end{align*}
Then we have that $\cF$ and its vector field are analytic on $\cB_{\rho,s}(R-d)$, and 
\begin{align} 
\sup_{\zeta \in \cB_{\rho,s}(R-d)} \|X_\cF(\zeta)\|_{\cH^{\rho,s}} &\leq  \frac{18}{d} \, \sup_{\zeta \in \cB_{\rho,s}(R)} \|X_\chi(\zeta)\|_{\cH^{\rho,s}} \; \sup_{\zeta \in \cB_{\rho,s}(R)} \|X_{\Ham G}(\zeta)\|_{\cH^{\rho,s}} . \label{vfPois}
\end{align}
\end{lemma}
\begin{proof}
Since
\begin{align*}
\Ham{h}_0(\Phi_\chi(\zeta)) - \Ham{h}_0(\zeta) &= \int_0^1 \; \{\chi,\Ham{h}_0 \} \circ \Phi^{t}_{\chi}(\zeta) \; \di t \stackrel{\eqref{homeq}}{=} \int_0^1 \;  \Ham Z( \Phi^{t}_{\chi}(\zeta) ) - \Ham G( \Phi^{t}_{\chi}(\zeta) ) \; \di t,
\end{align*}
if we define $\Ham F(\zeta):=\Ham Z(\zeta)-\Ham G(\zeta)$, we get
\begin{align*}
\cF(\zeta) &= \int_0^1 \left(\Ham F( \Phi^t_\chi(\zeta) ) -\Ham F(\zeta)\right) \, \di t.
\end{align*}
Now, we have
\begin{align*}
&\sup_{\zeta \in \cB_{\rho,s}(R-d)} \Vert X_{\cF}(\zeta) \Vert_{\cH^{\rho,s}} \\
&= \sup_{\zeta \in \cB_{\rho,s}(R-d)} \left\Vert \Omega^{-1} \di \left[ \; \int_0^1 \;\left( \Ham F(\Phi_\chi^t(\zeta))- \Ham F(\zeta)\right) \, \di t \; \right] \right\Vert_{\cH^{\rho,s}} \\
& \leq \sup_{\zeta \in \cB_{\rho,s}(R-d)} \left\Vert \int_0^1 \; \left[( \di \Phi_{\chi}^{-t} ( \Phi_\chi^t(\zeta)) - \mathrm{id}) \Omega^{-1} \di \Ham F (\Phi^t_\chi)+ \Omega^{-1} \di ( \Ham F (\Phi_\chi^t(\zeta))-\Ham F(\zeta) )\right] \; \di t \right\Vert_{\cH^{\rho,s}} \\
&\leq  \sup_{\zeta \in \cB_{\rho,s}(R-d)} \left\Vert \int_0^1 \; ( \di \Phi_{\chi}^{-t} ( \Phi_\chi^t(\zeta)) - \mathrm{id}) \Omega^{-1} \di \Ham F (\Phi^t_\chi) \; \di t \right\Vert_{\cH^{\rho,s}} \\
&+ \sup_{\zeta \in \cB_{\rho,s}(R-d)} \left\Vert \int_0^1 \;  \left(X_{\Ham F}(\Phi_\chi^t(\zeta))-X_{\Ham F}(\zeta)\right) \; \di t \right\Vert_{\cH^{\rho,s}} 
\end{align*}
and by dominated convergence we can bound the last quantity by 
\begin{align*}
&\sup_{\zeta \in \cB_{\rho,s}(R-d)} \sup_{t \in[0,1]} \Vert \di \Phi_{\chi}^{-t}(\Phi_\chi^t(\zeta))- \mathrm{id} \Vert_{ \cH^{\rho,s} \to \cH^{\rho,s} } \sup_{\zeta \in \cB_{\rho,s}(R-d)} \Vert X_{\Ham F}(\Phi^t_\chi(\zeta))\Vert_{\cH^{\rho,s}}\\
& + \sup_{\zeta \in \cB_{\rho,s}(R-d)} \sup_{t \in [0,1]} \Vert  X_{\Ham F}(\Phi_\chi^t(\zeta))-X_{\Ham F}(\zeta) \Vert_{\cH^{\rho,s}} \\
&\leq \sup_{t \in [0,1]} \sup_{\zeta \in \cB_{\rho,s}(R-d)} \Vert \di \Phi_{\chi}^{-t}(\Phi_\chi^t(\zeta))- \mathrm{id} \Vert_{ \cH^{\rho,s} \to \cH^{\rho,s} } \sup_{\zeta \in\cB_{\rho,s}(R-d)} \Vert X_{\Ham F}(\Phi^t_\chi(\zeta))\Vert_{\cH^{\rho,s}}\\
& + \sup_{t \in [0,1]} \sup_{\zeta \in \cB_{\rho,s}(R-d)}  \Vert  X_{\Ham F}(\Phi_\chi^t(\zeta))-X_{\Ham F}(\zeta) \Vert_{\cH^{\rho,s}} \\
&\leq \frac{3}{d} \sup_{\zeta \in \cB_{\rho,s}(R)} \Vert X_\chi(\zeta) \Vert_{\cH^{\rho,s}} \sup_{\zeta \in \cB_{\rho,s}(R)} \Vert X_{\Ham F}(\zeta) \Vert_{\cH^{\rho,s}} + \sup_{t \in [0,1]} \sup_{\zeta \in \cB_{\rho,s}(R-d)} \left\Vert \int_0^t [X_\chi,X_{\Ham F}](\Phi_\chi^s(\zeta)) \di s \right\Vert_{\cH^{\rho,s}},
\end{align*}
where we can estimate the last term by Cauchy inequality
\begin{align*}
\sup_{\zeta \in \cB_{\rho,s}(R-d)} \left\Vert \int_0^t [X_\chi,X_{\Ham F}](\Phi_\chi^s(\zeta)) \di s \right\Vert_{\cH^{\rho,s}} & \leq 2 \sup_{\zeta \in \cB_{\rho,s}(R-2d/3)} \Vert [X_\chi,X_{\Ham F}](\zeta) \Vert_{\cH^{\rho,s}} \\
&\leq \frac{6}{d} \sup_{\zeta \in \cB_{\rho,s}(R)} \Vert X_\chi(\zeta) \Vert_{\cH^{\rho,s}} \sup_{\zeta \in \cB_{\rho,s}(R)} \Vert X_{\Ham F}(\zeta) \Vert_{\cH^{\rho,s}} \, .
\end{align*}
By the above computations  and \eqref{vfhomeq} we obtain
\begin{align*}
\sup_{\zeta \in\cB_{\rho,s}(R-d)} \Vert X_\cF(\zeta) \Vert_{\cH^{\rho,s}} &\leq \frac{9}{d} \sup_{\cB_{\rho,s}(R)} \Vert X_{ \chi }(\zeta) \Vert_{\cH^{\rho,s}} \; \sup_{\zeta \in\cB_{\rho,s}(R)} \Vert X_{ \Ham F }(\zeta) \Vert_{\cH^{\rho,s}} \\
&\stackrel{\eqref{vfhomeq}}{\leq} \frac{18}{d} \sup_{\zeta \in \cB_{\rho,s}(R)} \Vert X_{ \chi }(\zeta) \Vert_{\cH^{\rho,s}} \; \sup_{\zeta \in \cB_{\rho,s}(R)} \Vert X_{ G }(\zeta) \Vert_{\cH^{\rho,s}}.
\end{align*}
\end{proof}

\begin{lemma} \label{itlemma}
Let $s \geq s_1 \gg 1$, $R>0$, $m\geq 0$, {$0 < \nu \leq 1$}, and consider the Hamiltonian 
\begin{align} \label{Hm}
H^{(m)}(\zeta) &= \Ham{h}_0(\zeta) + \delta \Ham Z^{(m)}(\zeta) + \delta^{ {\nu}m+1 } \Ham F^{(m)}(\zeta).
\end{align}
Assume that $\Ham{h}_0$ satisfies (PER) and (INV), and that
\begin{align*}
\sup_{\zeta \in \cB_{\rho,s}(R)} \|X_{\Ham F^{(0)}}(\zeta)\|_{\cH^{\rho,s}} &\leq F.
\end{align*}
Fix $d < \frac{R}{m+1}$, and set $R_m := R-md$ ($m \geq 1$). \\
Assume also that $\Ham Z^{(m)}$ is analytic on $\cB_{\rho,s}(R_m)$, and that
\begin{align}
\sup_{\zeta \in \cB_{\rho,s}(R_m)} \|X_{\Ham Z^{(0)}}(\zeta)\|_{\cH^{\rho,s}} &=0, \nonumber \\
\sup_{\zeta \in \cB_{\rho,s}(R_m)} \|X_{\Ham Z^{(m)}}(\zeta)\|_{\cH^{\rho,s}} &\leq F\sum_{i=0}^{m-1}\delta^{ {\nu} i} K_0^i, \; \; m \geq 1, \label{eq:estimateZm} \\
\sup_{\zeta \in \cB_{\rho,s}(R_m)} \|X_{\Ham F^{(m)}}(\zeta)\|_{\cH^{\rho,s}} &\leq F \, K_0^m, \; \; m \geq 1, \label{stepm}
\end{align}

with $K_0 \geq 15$ and $d > 3T \delta F$. \\
Then, if $\delta^{{\nu}} K_0 < 1/2$ there exists a canonical transformation $\cT^{(m)}_\delta$ analytic on $\cB_{\rho,s}(R_{m+1})$ such that 
\begin{align} \label{CTm}
\sup_{\zeta \in \cB_{\rho,s}(R_{m+1})} \|\cT^{(m)}_\delta(\zeta)-\zeta\|_{\cH^{\rho,s}} &\leq 2T \delta^{ {\nu} m+1 }  K_0^m F,
\end{align}
$\Ham H^{(m+1)} := \Ham H^{(m)} \circ \cT^{(m)}$ has the form \eqref{Hm} and 
satisfies \eqref{stepm} with $m$ replaced by $m+1$.
\end{lemma}

\begin{proof}
The key point of the proof is to look for $\cT^{(m)}_\delta$ as the time-one map of the Hamiltonian vector field of an analytic function $\delta^{ {\nu} m+1 }\chi_m$. Hence, consider the differential equation
\begin{align} \label{chim}
\dot\zeta &= X_{\delta^{ {\nu}m+1 } \chi_m}(\zeta).
\end{align}
By standard theory we have that, if $\|X_{\delta^{ {\nu} m+1 }\chi_m}\|_{\cB_{\rho,s}(R_m)}$ is small enough (e.g. $\|X_{\delta^{ {\nu} m+1 } \chi_m}\|_{\cB_{\rho,s}(R_m)} \leq \frac{ md}{m+1}$) and $\zeta_0 \in \cB_{\rho,s}(R_{m+1})$, then the solution of \eqref{chim} exists for $|t| \leq 1$. 

Therefore we can define $\cT^t_{m,\delta}:\cB_{\rho,s}(R_{m+1}) \to \cB_{\rho,s}(R_m)$, and in particular the corresponding time-one map $\cT^{(m)}_\delta:=\cT^1_{m,\delta}$, which is an analytic canonical transformation, $\delta^{ {\nu} m+1 }$-close to the identity. We have 
\begin{align}
(\cT^{(m)}_\delta)^\ast \; (\Ham{h}_0 + \delta\Ham Z^{(m)} + \delta^{m+1} &\Ham F^{(m)}) = \Ham{h}_0 + \delta \Ham Z^{(m)} + \delta^{ {\nu} m+1 } \left[ \{ \chi_m,\Ham{h}_0 \} + \Ham F^{(m)} \right] + \nonumber \\
&\; \; \; + \left( \Ham{h}_0 \circ \cT^{(m)}_\delta - \Ham{h}_0  - \delta^{ {\nu} m+1} \{ \chi_m,\Ham{h}_0 \} \right) + \delta \left( \Ham Z^{(m)} \circ \cT^{(m)}_\delta - \Ham Z^{(m)} \right) \label{nonnorm1} \\
&\; \; \; + \delta^{ {\nu} m+1 } \left( \Ham F^{(m)} \circ \cT^{(m)}_\delta - \Ham F^{(m)} \right). \label{nonnorm2}
\end{align}
It is easy to see that the first two terms on the right-hand side are already normalized, that the third term is the non-normalized part of order $m+1$ that can be normalized through the choice of a suitable $\chi_m$, and that \eqref{nonnorm1}-\eqref{nonnorm2} contain all the terms of order higher than $m+1$. 
 
In order to normalize the third term on the right-hand side we solve the homological equation
\begin{align*}
\{ \chi_m,\Ham{h}_0 \} + \Ham F^{(m)} \; &= \; \Ham Z_{m+1},
\end{align*}
with $\Ham Z_{m+1}$ in normal form. Lemma \ref{homeqlemma} ensures the existence of $\chi_m$ and $\Ham Z_{m+1}$ as well as their explicit expressions:
\[
	\begin{split}
		\Ham Z_{m+1}(\zeta) & = \frac{1}{T} \int_0^T \Ham F^{(m)}(\Phi_{\Ham{h}_0}^t(\zeta)) \, \di t \, ,\\
		\chi_m(\zeta) &= \frac{1}{T} \int_0^T t \, [\Ham F^{(m)}(\Phi_{\Ham{h}_0}^t(\zeta))- \Ham Z_{m+1}(\Phi_{\Ham{h}_0}^t(\zeta))] \, \di t \, .\\
	\end{split}
\]
The explicit expression of $X_{\chi_m}$ can be computed following the argument of Lemma \ref{homeqlemma}. Using this explicit expression, the analyticity of the flow $\Phi_{\Ham{h}_0}^t$ ensured by (PER) and \eqref{vfhomeq} one has
\[\tag{*} \label{eq:EstimateVectorField}
	\begin{split}
		\sup_{\zeta \in\cB_{\rho,s}(R_m)} \Vert X_{\chi_m} (\zeta)\Vert_{\cH^{\rho,\sigma}} \leq 2T  \sup_{\zeta \in \cB_{\rho,s}(R_m)} \Vert X_{\Ham F^{(m)}}\Vert_{\cH^{\rho,\sigma}} \leq 2T  K_0^m F \, .
	\end{split}
\]
Straightforwardly, from the explicit expression of $Z_{m+1}(\zeta)$ and \eqref{stepm} one has
\[
	\sup_{\zeta \in \cB_{\rho,s}(R_m)} \Vert X_{Z_{m+1}} \Vert_{\cH^{\rho,s}} \leq   K_0^m F \, .
\]
Now define $\Ham Z^{(m+1)}:=\Ham Z^{(m)}+ \delta^{ {\nu} m } \Ham Z_{m+1}$ and notice that as a consequence of the latter estimate and \eqref{eq:estimateZm} we have
\[
	\begin{split}
		\sup_{\zeta \in \cB_{\rho,s}(R_{m+1})} \Vert X_{\Ham Z^{(m+1)}}(\zeta) \Vert &\leq  \sup_{\zeta \in\cB_{\rho,s}(R_{m+1})}\Vert X_{\Ham Z^{(m)}}(\zeta) \Vert_{\cH^{\rho,s}} + \sup_{\zeta \in\cB_{\rho,s}(R_{m+1})} \Vert X_{\delta^{ {\nu} m } \Ham  Z_{m+1}} (\zeta) \Vert_{\cH^{\rho,s}} \\
		&\leq  F\left(\sum_{j=0}^{m-1} \delta^{{\nu} j} K_0^j + \delta^{ {\nu} m} K_0^m\right) \, .
	\end{split}
\]
Defining now $\cT^{(m)}_\delta(\zeta) := \Phi^1_{\delta^{ {\nu} m+1 } \chi_m}(\zeta)$ we can apply Lemma \ref{cauchylemma} and \eqref{eq:EstimateVectorField} to obtain
\[
	\begin{split}
		\sup_{\zeta \in \cB_{\rho,s}(R_{m+1})} \Vert \cT^{(m)}_\delta(\zeta)-\zeta 	\Vert_{\cH^{\rho,s}} &= \sup_{\zeta \in \cB_{\rho,s}(R_{m+1})}  \Vert \Phi^1_{\delta^{ {\nu} m+1} \chi_m}(\zeta)-\zeta \Vert_{\cH^{\rho,s}} \\
		&\leq \sup_{\zeta \in \cB_{\rho,s}(R_m)} \Vert X_{\delta^{ {\nu} m+1} \chi_m} \Vert_{\cH^{\rho,s}} \leq 2T \delta^{ {\nu} m+1}  K_0^m F \, .
	\end{split}
\]

Let us set now $\delta^{ { \nu(m+1) +1 } } \Ham F^{(m+1)} :=  \eqref{nonnorm1} + \eqref{nonnorm2}$. Using Lemma \ref{lem:VectorFieldCanonicalTransformation} one can estimate separately the three pieces. We notice that $\sup_{\cB_{\rho,s}(R_m)}\Vert X_{\delta^{ {\nu} m+1} \chi_m} \Vert_{\cH^{\rho,s}} \leq 2T \delta^{ {\nu} m+1} K_0^m F$ and since $\delta^{{\nu}} K_0 < \frac{1}{2}$ we have $\sup_{\cB_{\rho,s}(R_m)}\Vert X_{\delta^{ {\nu} m+1} \chi_m} \Vert_{\cH^{\rho,s}} < T \delta F < \frac{d}{3} \leq \frac{(m+1) d}{3}$. We can thus apply Lemma \ref{lem:VectorFieldCanonicalTransformation} and Lemma \ref{Poissonlemma} to get
\begin{align*}
\sup_{\zeta \in \cB(R_{m+1})} \Vert X_{\Ham Z^{(m)} \circ \cT^{(m)}_\delta-\Ham Z^{(m)}} (\zeta) \Vert_{\cH^{\rho,s}} &\leq \frac{27 \; \delta^{ {\nu} m+1 }}{(m+1)d} \sup_{\zeta \in \cB_{\rho,s}(R_m)} \Vert X_{\chi_m}(\zeta) \Vert_{\cH^{\rho,s}} \; \sup_{\zeta \in \cB_{\rho,s}(R_m)} \Vert X_{\Ham Z^{(m)}} \Vert_{\cH^{\rho,s}} , \\
\sup_{\zeta \in \cB(R_{m+1})} \Vert X_{\Ham F^{(m)} \circ \cT^{(m)}_\delta-\Ham F^{(m)}} (\zeta) \Vert_{\cH^{\rho,s}} &\leq \frac{27 \; \delta^{ {\nu} m+1 }}{(m+1)d} \sup_{\zeta \in \cB_{\rho,s}(R_m)} \Vert X_{\chi_m}(\zeta) \Vert_{\cH^{\rho,s}} \; \sup_{\zeta \in \cB_{\rho,s}(R_m)} \Vert X_{\Ham F^{(m)}} \Vert_{\cH^{\rho,s}} , \\
\sup_{\zeta \in \cB(R_{m+1})} \Vert X_{\Ham{h}_0 \circ \cT^{(m)}_\delta-\Ham{h}_0-\delta^{ {\nu} m+1 } \{ \chi_m, \Ham{h}_0\}} \Vert_{\cH^{\rho,s}} &\leq \frac{18 \; \delta^{2m {\nu}+2} }{(m+1)d} \; \sup_{\zeta \in \cB_{\rho,s}(R_m)} \|X_{\chi_m}(\zeta)\|_{\cH^{\rho,s}} \; \sup_{\zeta \in \cB_{\rho,s}(R_m)} \|X_{\Ham F^{(m)}}(\zeta)\|_{\cH^{\rho,s}}.
\end{align*}
By means of these inequalities and by exploiting $\Vert X_{\delta^{{\nu} m+1}\chi_m} \Vert_{\cH^{\rho,s}} \leq \frac{(m+1)d}{3}$ together with assumptions \eqref{eq:estimateZm} and \eqref{stepm}, we can estimate
\[
	\begin{split}
		\sup_{\zeta \in \cB_{\rho,s}(R_{m+1})} \|X_{ \delta^{ {\nu (m+1)+1 } } \Ham F^{(m+1)} }(\zeta)\|_{\cH^{\rho,s}} & \leq 9 \delta^{ {\nu(m+1)+1} } \sup_{\zeta \in \cB_{\rho,s}(R_m)} \Vert X_{\Ham Z^{(m)}}(\zeta) \Vert_{\cH^{\rho,s}} \\
                & + 9 \; \delta^{ 2m {\nu} + 2 } \; \sup_{\zeta \in  \cB_{\rho,s}(R_m)} \Vert X_{\Ham F^{(m)}}(\zeta) \Vert_{\cH^{\rho,s}}\\
		& + 6 \; \delta^{ 2m {\nu} +2 } \sup_{\zeta \in \cB_{\rho,s}(R_m)} \Vert X_{\Ham F^{(m)}}(\zeta) \Vert_{\cH^{\rho,s}} \\
		& \leq 9 \; \delta^{ {\nu (m+1)+1 } }  F\sum_{i=0}^{m-1}\delta^i K_0^i +9 \; \delta^{2m {\nu} +2} F \, K_0^m +6 \; \delta^{2m {\nu}+2} F \, K_0^m \\
                &=\delta^{ {\nu(m+1)+1} } \left(9 F \sum_{i=0}^{m-1}\delta^i K_0^i+ 9 \delta^{ {\nu(m-1)+1} } F \, K_0^m + 6 \; \delta^{ {\nu(m-1)+1} } F \, K_0^m  \right).
	\end{split}
\]
If $m=0$ then we have
\begin{align*}
\sup_{\zeta \in \cB_{\rho,s}(R_1)} \Vert X_{\delta^{ {1+\nu} } \Ham F^{(1)}} \Vert_{\cH^{\rho,s}} &\leq  \delta^{{1+\nu} } (9 F  + 6 \, F ).
\end{align*}
If $m \geq 1$ we exploit the smallness condition $\delta^{{\nu}} K_0 < \frac{1}{2}$ to get $\sum_{i=0}^{m-1} \delta^{ {\nu} i } K_0^i < 2$ and
\begin{align}
\sup_{\zeta \in \cB_{\rho,s}(R_{m+1})} \Vert X_{\delta^{ {\nu(m+1)+1} } \Ham F^{(m+1)}} \Vert_{\cH^{\rho,s}} &\leq \delta^{ {\nu(m+1)+1} } \left( 6F + 9 \, { \delta K_0 \frac{F}{2^{m-1}} }  + 6 \, { \delta K_0 \frac{F}{2^{m-1}} } \right), \label{itbound}
\end{align}
{
and since $0< \nu \leq 1$ the right-hand side of \eqref{itbound} can be bounded by
\begin{align*}
\delta^{ \nu(m+1)+1 } \left( 6F + 9 \, \delta^{1-\nu} \frac{F}{2^{m}} + 6 \,  \delta^{1-\nu} \frac{F}{2^{m}} \right) &\leq  15 \, \delta^{\nu(m+1)+1} F.
\end{align*}
}
\end{proof}

\begin{proof}[Proof of Lemma \ref{NFest}]
The Hamiltonian \eqref{truncsys} satisfies the assumptions 
of Lemma \ref{itlemma} with $m=0$, $\Ham F_{1,M}$ in place of $\Ham F^{(0)}$, 
$F = K^{(F)}_{1,s} \, M^{2+{{\Gamma}}}$.
So we apply Lemma \ref{itlemma} with $d=R/4$, provided that 
\[\delta < \frac{R}{12 \, T \, F} = \frac{R}{12 \, T \, K_{1,s}^{(F)} M^{2+{{\Gamma}}}} \]
which is true due to \eqref{smallcond}. 
Hence there exists an analytic canonical transformation 
$\cT^{(0)}_{\delta,M}: \cB_{\rho,s}(3R/4) \to \cB_{\rho,s}(R)$ with
\[ \sup_{\zeta \in \cB_{\rho,s}(3R/4)} \|\cT^{(0)}_{\delta,M}(\zeta)-\zeta\|_{\cH^{\rho,s}} \leq 2T \,   F \, \delta, \]
such that 
\begin{align}
& \Ham H_{1,M} \circ \cT^{(0)}_{\delta,M} = \Ham{h}_0 + \delta \Ham Z^{(1)}_M + \delta^{{1+\nu}} \cR^{(1)}_M, \label{step1} \\
& \Ham Z^{(1)}_M := \la \Ham F_{1,M} \ra, \\
& \delta^{{1+\nu}} \cR^{(1)}_M := \delta^{{1+\nu}} \Ham F^{(1)} \nonumber \\
&=\left( \Ham{h}_0 \circ \cT^{(0)}_{\delta,M} - \Ham{h}_0  - \delta \{ \chi_1,\Ham{h}_0 \} \right) 
+ \delta \left( \Ham Z^{(1)}_M \circ \cT^{(0)}_{\delta,M} - \Ham Z^{(1)}_M \right) + \delta^{{1+\nu}} \left( \Ham F_{1,M} \circ \cT^{(0)}_{\delta,M} - \Ham F_{1,M} \right), \\
& \sup_{\zeta \in \cB_{\rho,s}(3R/4)} \|X_{ \Ham Z^{(1)}_M }(\zeta)\|_{\cH^{\rho,s}} \leq F, \\
& \sup_{\zeta \in \cB_{\rho,s}(3R/4)} \|X_{ \cR^{(1)}_N }(\zeta)\|_{\cH^{\rho,s}} \leq 
15 \; F.
\end{align}
and $K_0=15$, whence $\delta < \frac{1}{30^{ {1/\nu} } }$.
\end{proof}

\section{Proof of Proposition \ref{KdVxietaProp} and Proposition \ref{KPxietaProp}} \label{ApprEstSec11}

\begin{proof}[Proof of Proposition \ref{KdVxietaProp}] In order to prove Proposition \ref{KdVxietaProp} we first discuss the specific energies associated to the high modes, and then the ones associated to the low modes.

First we remark that for all $k$ such that $\kappa(k) = (\mu K_1,\mu^\sigma K_2)$ we have
\begin{align}
\left| \frac{\omega^2_k}{\mu^2} \right| \;&\stackrel{\eqref{FreqNormMode}}{=}\; \frac{4}{\mu^2} \left[ \sin^2\left(\frac{k_1 \pi}{2N_1+1}\right) + \sin^2\left(\frac{k_2 \pi}{2N_2+1}\right) \right] \nonumber \\
&=\; \frac{4}{\mu^2} \left[ \sin^2\left(\frac{\mu K_1 \pi}{2}\right) + \sin^2\left(\frac{\mu^\sigma K_2 \pi}{2}\right) \right] \; \leq \; \pi^2 (K_1^2+\mu^{2(\sigma-1)}K_2^2); \label{EstFreqKdVr}
\end{align}
moreover, for $K_1 \neq 0$
\begin{align}
\frac{|\hat{q}_K|^2+\pi^2(K_1^2+\mu^{2(\sigma-1)}K_2^2)|\hat{p}_K|^2}{2} &\leq \pi^2 \, e^{-2\rho|K|} \; \frac{|\hat{q}_K|^2+(K_1^2+\mu^{2(\sigma-1)}K_2^2)|\hat{p}_K|^2}{2} e^{2\rho|K|} \nonumber \\
&\leq \pi^2 \, e^{-2\rho|K|} \; \left(1+\mu^{2(\sigma-1)} \frac{K_2^2}{K_1^2}\right) \; \|(\xi,\eta)\|_{\cH^{\rho,0}}^2, \label{NormModeKdVEst1} 
\end{align}
while for $|K_2| \leq |K_1|$
\begin{align}
\frac{|\hat{q}_K|^2+\pi^2(K_1^2+\mu^{2(\sigma-1)}K_2^2)|\hat{p}_K|^2}{2} &\stackrel{|K_2| \leq |K_1| }{\leq} \; \; 2 \pi^2 \, e^{-2\rho|K|}  \; \|(\xi,\eta)\|_{\cH^{\rho,0}}^2. \label{NormModeKdVEst2} 
\end{align}
{\color{black}In order to estimate $\mathcal{E}_\kappa$ for large $K_1$ and $K_2$, it is convenient to divide the frequency-space in different regions and bound the terms supported in each region separately. Unlike the one-dimensional case, only few some terms will turn out to be exponentially small with respect to $\mu$, so the introduction of different regions in the frequency space will help us estimating most of the terms in an efficient way. Let us define
\begin{equation}\label{eq:ScrL}
	\mathscr{L}_{\mu,\delta,\rho}\;:=\; \left\{L=(L_1,L_2) \in \mathbb{Z}^2 \, : \, \mu L_1, \mu^\sigma L_2 \in 2 \mathbb{Z} \, , \, |K_1|+ |K_2| > \frac{(2+\delta) |\log \mu|}{\rho} \right\}
\end{equation}}
Using the explicit expression \eqref{SpecEnNormModeKdV} we obtain
\begin{align}
&\frac{ \cE_\kappa }{\mu^{4}} = \sum_{ \substack{ L \in \mathscr{L}_{\mu,\delta,\rho} \\ |K_2+L_2| \leq |K_1+L_1| } } \left( |\hat{q}_{K+L}|^2 + \omega_k^2 \left| \frac{\hat{p}_{K+L}}{\mu} \right|^2 \right) + \sum_{ \substack{L \in \mathscr{L}_{\mu,\delta,\rho} \\ |K_2+L_2| > |K_1+L_1| } } \left( |\hat{q}_{K+L}|^2 + \omega_k^2 \left| \frac{\hat{p}_{K+L}}{\mu} \right|^2 \right) \nonumber \\
&\stackrel{ \eqref{EstFreqKdVr},\eqref{NormModeKdVEst2},\eqref{ZeroAvy1KdV} }{\leq} \pi^2 \; \|(\xi,\eta)\|_{\cH^{\rho,0}}^2 \; 2 \sum_{ \substack{  L \in \mathscr{L}_{\mu,\delta,\rho} \\ |K_2+L_2| \leq |K_1+L_1| } } e^{-2\rho|K+L|} \nonumber \\
&\; \;   \qquad + \pi^2 \; \|(\xi,\eta)\|_{\cH^{\rho,0}}^2 \; \sum_{ \substack{ L \in \mathscr{L}_{\mu,\delta,\rho} \\ |K_2+L_2| > |K_1+L_1| \\ K_1+L_1 \neq 0 } } e^{-2\rho|K+L|} \; \left(1+\mu^{2(\sigma-1)} \frac{(K_2+L_2)^2}{(K_1+L_1)^2}\right) \nonumber \\
&\; \;   \qquad + \pi^2 \; \|(\xi,\eta)\|_{\cH^{\rho,0}}^2 \; \sum_{ \substack{ L \in \mathscr{L}_{\mu,\delta,\rho} \\ |K_2+L_2| > |K_1+L_1| \\ K_1+L_1 = 0 } } e^{-2\rho|K_2+L_2|}.  \nonumber
\end{align}

Now,
\begin{align}
&\sum_{ \substack{ L \in \mathscr{L}_{\mu,\delta,\rho}  } } e^{-2\rho|K+L|}\;\leq\; e^{-2\rho|K|} + \sum_{ \substack{ L \in \mathscr{L}_{\mu,\delta,\rho} \\ L_1=0,L_2 \neq 0  } } e^{-2\rho|K+L|} + \sum_{ \substack{  L \in \mathscr{L}_{\mu,\delta,\rho} \\ L_1 \neq 0, L_2=0 } } e^{-2\rho|K+L|} + \sum_{ \substack{  L \in \mathscr{L}_{\mu,\delta,\rho} \\ L_1,L_2 \neq 0 } } e^{-2\rho|K+L|}. \label{DecompHighFreqTerm1}
\end{align}
We now estimate the last sum in \eqref{DecompHighFreqTerm1}; we point out that for $L_1,L_2 \neq 0$ we have $|L| \geq \frac{2}{\mu} + \frac{2}{\mu^\sigma}$, hence $2|K| \leq |L|$.

Therefore, for any $k$ such that $\kappa(k) = (\mu K_1,\mu^{\sigma} K_2)$ and $|K_1|+|K_2| \geq \frac{(2+\delta)|\log\mu|}{\rho}$
\begin{align}
\sum_{ \substack{  L \in \mathscr{L}_{\mu,\delta,\rho} \\ L_1,L_2 \neq 0 } } e^{-2\rho|K+L|} &\leq \sum_{ \substack{ L \in \mathscr{L}_{\mu,\delta,\rho} \\ L_1,L_2 \neq 0 } } e^{-2\rho\, | \, |K|-|L| \, |} \leq \sum_{ \substack{  L \in \mathscr{L}_{\mu,\delta,\rho} \\ L_1,L_2 \neq 0 } } e^{2\rho|K|} e^{-2\rho|L|} \nonumber \\
&\leq e^{2\rho|K|} \, 2\pi \, \int_{2|K|}^{+\infty} R e^{-2\rho R} \di R = 2 \pi \, e^{2\rho|K|} \, \left( -\frac{1}{2} \right) \frac{\di}{\di\rho} \left[ \int_{2|K|}^{+\infty} e^{-2\rho R} \di R \right] \nonumber \\
&= -\pi \, e^{2\rho|K|} \, \frac{\di}{\di\rho} \left( \frac{e^{-4\rho |K|}}{2\rho} \right) 
\; = \; \frac{\pi}{2 \rho} \left(\frac{1}{\rho}+4 \right) e^{-2 \rho |K|} \, . \label{Est1HighFreqTerm1}
\end{align}
Next we estimate the second sum in \eqref{DecompHighFreqTerm1}; we have
\begin{align}
\sum_{ \substack{  L \in \mathscr{L}_{\mu,\delta,\rho} \\ L_1 \neq 0, L_2 = 0 } } e^{-2\rho|K+L|} &\leq e^{-2\rho\, (|K_1|+|K_2|)} \sum_{ \substack{  \ell \in \Z \setminus \{0\}   } } e^{-4 \rho |\ell|/\mu}, \label{Est2HighFreqTerm1}
\end{align}
which is exponentially small with respect to $\mu$. Similarly,
\begin{align}
\sum_{ \substack{  L \in \mathscr{L}_{\mu,\delta,\rho} \\ L_1 = 0, L_2 \neq 0 } } e^{-2\rho|K+L|} &\;\leq\; e^{-2\rho\, (|K_1|+|K_2|)} \sum_{ \substack{  \ell \in \Z \setminus \{0\}   } } e^{-4 \rho |\ell|/\mu^\sigma}. \label{Est3HighFreqTerm1}
\end{align}
Then,
\begin{align}
\sum_{ \substack{L \in \mathscr{L}_{\mu,\delta,\rho} \\ |K_2+L_2| > |K_1+L_1| \\ K_1+L_1 \neq 0 } } &e^{-2\rho|K+L|} \; \frac{(K_2+L_2)^2}{(K_1+L_1)^2} \;\leq \; e^{-2\rho|K|} \left( \frac{K_2}{K_1} \right)^2  + \sum_{ \substack{ L \in \mathscr{L}_{\mu,\delta,\rho} \\ |K_2+L_2| > |K_1+L_1| \\ K_1+L_1 \neq 0 \\ L_1 \neq 0, L_2=0} } e^{-2\rho|K+L|} \; \frac{(K_2+L_2)^2}{(K_1+L_1)^2} \nonumber \\
& +\sum_{ \substack{ L \in \mathscr{L}_{\mu,\delta,\rho} \\ |K_2+L_2| > |K_1+L_1| \\ K_1+L_1 \neq 0 \\ L_1=0,L_2 \neq 0} } e^{-2\rho|K+L|} \; \frac{(K_2+L_2)^2}{(K_1+L_1)^2} + \sum_{ \substack{ L \in \mathscr{L}_{\mu,\delta,\rho} \\ |K_2+L_2| > |K_1+L_1| \\ K_1+L_1 \neq 0 \\ L_1,L_2 \neq 0} } e^{-2\rho|K+L|}\; \frac{(K_2+L_2)^2}{(K_1+L_1)^2}. \label{DecompHighFreqTerm2}
\end{align}
First we estimate the last term in \eqref{DecompHighFreqTerm2}: 
we have that $|L+K| \geq |K|$, hence
\begin{align}
\sum_{ \substack{ L \in \mathscr{L}_{\mu,\delta,\rho} \\ |K_2+L_2| > |K_1+L_1| \\ K_1+L_1 \neq 0 \\ L_1,L_2 \neq 0} } e^{-2\rho|K+L|} \; \frac{(K_2+L_2)^2}{(K_1+L_1)^2} &= \int_{|K|}^{+\infty} \int_0^{\pi/4} e^{-2\rho \, \xi} \, \xi \tan^2\phi \, \di\phi \, \di \xi \nonumber \\
&= \left( 1-\frac{\pi}{4} \right) \, e^{-2\rho |K|} \, \frac{1+2\rho |K|}{4\rho^2} \nonumber \\
&\leq \left( 1-\frac{\pi}{4} \right) \, \mu^4 \, e^{ -2\rho \left[ |K| - \frac{2|\log\mu|}{\rho} - \frac{1}{2\rho} \log(2\rho|K|) \right] } \nonumber \\
&\stackrel{\delta < 1 - 1/e}{\leq} \left( 1-\frac{\pi}{4} \right) \, \mu^4 \, e^{ -2\rho \, \left[ \delta |K| - \frac{2 |\log\mu|}{\rho} \right] } \nonumber \\
&=\left(1-\frac{\pi}{4} \right) \mu^8 e^{-2 \rho \delta |K|}  . \label{Est1HighFreqTerm21}
\end{align}

Now we bound the other two nontrivial terms in \eqref{DecompHighFreqTerm2}; on the one hand, we notice that
{
\begin{align}
\sum_{ \substack{ L \in \mathscr{L}_{\mu,\delta,\rho} \\ |K_2+L_2| > |K_1+L_1| \\ K_1+L_1 \neq 0 \\ L_1 \neq 0,L_2=0} } e^{-2\rho|K+L|} \frac{(K_2+L_2)^2}{(K_1+L_1)^2} &= K_2^2 \; \sum_{ \substack{ L \in \mathscr{L}_{\mu,\delta,\rho} \\ |K_2+L_2| > |K_1+L_1| \\ K_1+L_1 \neq 0 \\ L_1 \neq 0,L_2=0} } e^{-2\rho|K+L|} \frac{1}{(K_1+L_1)^2} \nonumber \\
&\leq K_2^2 \, e^{-2\rho|K|} \sum_{ \substack{  \ell \in \Z \setminus \{0\}   } } e^{-4 \rho |\ell|/\mu} \nonumber \\
&\leq \; 2 K_2^2 \, e^{-2\rho|K|} \int_1^{+\infty} e^{-4 \rho |\ell|/\mu} \, \di \ell , \label{DecompHighFreqTerm22}
\end{align}
where the last integral is exponentially small with respect to $\mu$, while on the other hand 
\begin{align}
\sum_{ \substack{ L \in \mathscr{L}_{\mu,\delta,\rho} \\ |K_2+L_2| > |K_1+L_1| \\ K_1+L_1 \neq 0 \\ L_1 \neq 0,L_2=0} } e^{-2\rho|K+L|} \frac{(K_2+L_2)^2}{(K_1+L_1)^2} &\leq \frac{1}{K_1^2} \; \sum_{ \substack{ L \in \mathscr{L}_{\mu,\delta,\rho} \\ |K_2+L_2| > |K_1+L_1| \\ K_1+L_1 \neq 0 \\ L_1 = 0,L_2 \neq 0} } e^{-2\rho|K+L|}  (K_2+L_2)^2 \nonumber \\
&\leq \frac{1}{K_1^2} \; e^{-2\rho|K|} \sum_{ \substack{  \ell \in \Z \setminus \{0\}   } } e^{-4 \rho |\ell|/\mu^\sigma} \, \left( K_1^2 + 2K_1 \frac{\ell}{\mu^{\sigma}} + \frac{\ell^2}{\mu^{2\sigma}} \right) \nonumber \\
&\leq \frac{ 2 e^{-2\rho|K|} }{K_1^2} \;  \int_1^{+\infty} e^{-4 \rho |\ell|/\mu^\sigma} \, \left( K_1^2 + 2K_1 \frac{\ell}{\mu^{\sigma}} + \frac{\ell^2}{\mu^{2\sigma}} \right) \, \di \ell, \label{DecompHighFreqTerm23}
\end{align}
}
and the last integral is again exponentially small with respect to $\mu$. \\

On the other hand, for any $k$ such that $\kappa(k) = (\mu K_1,\mu^{\sigma} K_2)$ and $|K_1|+|K_2| \leq \frac{(2+\delta) |\log\mu|}{\rho}$
\begin{align}
&\left| \frac{\cE_\kappa}{\mu^{4}} - \frac{|\hat{\xi}_K|^2+|\hat{\eta}_K|^2}{2} \right| \nonumber \\
&\leq \left|\frac{\omega_k^2-\mu^2 \, \pi^2 K_1^2}{2\mu^2}\right| |\hat{p}_K|^2 + \frac{1}{2} \sum_{ \substack{ L=(L_1,L_2) \in \Z^2 \setminus \{0\} \\ \mu L_1, \mu^{\sigma} L_2 \in 2\Z} } \left( |\hat{q}_{K+L}|^2 + \omega_k^2 \left| \frac{\hat{p}_{K+L}}{\mu} \right|^2\right) \, , \nonumber \\
&\stackrel{\eqref{EstFreqKdVr}}{\leq} (\mu^2 \pi^4 K_1^4 + \pi^2 \mu^{2(\sigma-1)} K_2^2) |\hat{p}_K|^2 \nonumber \\
&\; \; \; \; + \frac{1}{2} \sum_{ \substack{ L=(L_1,L_2) \in \Z^2 \setminus \{0\} \\ \mu L_1, \mu^{\sigma} L_2 \in 2\Z} } \left(|\hat{q}_{K+L}|^2 + \pi^2[(K_1+L_1)^2+\mu^{2(\sigma-1)}(K_2+L_2)^2] |\hat{p}_{K+L}|^2\right), \nonumber \\
&\leq \left( \pi^4 \, \mu^2 K_1^4 + \pi^2 \mu^{2(\sigma-1)} \, \frac{9 |\log\mu|^2}{\rho^2} \right) |\hat{p}_K|^2 \nonumber \\
&\; \; \; \; + \frac{1}{2} \sum_{ \substack{ L=(L_1,L_2) \in \Z^2 \setminus \{0\} \\ \mu L_1, \mu^{\sigma} L_2 \in 2\Z} } \left( |\hat{q}_{K+L}|^2 + \pi^2[(K_1+L_1)^2+\mu^{2(\sigma-1)}(K_2+L_2)^2] |\hat{p}_{K+L}|^2\right), \nonumber \\
&\leq \left( \pi^4 \, \mu^2  + \pi^2 \mu^{2(\sigma-1)} \, \right) \frac{9 |\log\mu|^2}{\rho^2}  \, 2 \|(\xi,\eta)\|_{\cH^{\rho,0}}^2 \label{EstLowFreqTerm1} \\
&\; \; \; \; + \frac{\pi^2}{2} \sum_{ \substack{ L=(L_1,L_2) \in \Z^2 \setminus \{0\} \\ \mu L_1, \mu^{\sigma} L_2 \in 2\Z} } e^{2\rho |K+L|} \, (|\hat\xi_{K+L}|^2 + |\hat\eta_{K+L}|^2) \, \left( 1 + 2 \mu^{2(\sigma-1)} \frac{K_2^2+L_2^2}{(K_1+L_1)^2} \right)\, e^{-2\rho |K+L|}, \label{EstLowFreqTerm2}
\end{align}
and we can conclude by estimating \eqref{EstLowFreqTerm1} by exploiting the fact that {$|\log\mu| \leq \mu^{-1/4}$}, while we can estimate \eqref{EstLowFreqTerm2} by
\begin{align}
&\frac{\pi^2}{2} \, \|(\xi,\eta)\|_{\cH^{\rho,0}}^2 \, \sum_{ \substack{ L=(L_1,L_2) \in \Z^2 \setminus \{0\} \\ \mu L_1, \mu^{\sigma} L_2 \in 2\Z} } \, \left( 1 + 2 \mu^{2(\sigma-1)} \frac{K_2^2+L_2^2}{(K_1+L_1)^2} \right)\, e^{-2\rho |K+L|} \nonumber \\
&\leq \frac{\pi^2}{2} \, \|(\xi,\eta)\|_{\cH^{\rho,0}}^2 \, \sum_{ \substack{ L=(L_1,L_2) \in \Z^2 \setminus \{0\} \\ \mu L_1, \mu^{\sigma} L_2 \in 2\Z} } \, \left( 1 + 2 \, \mu^{2(\sigma-1)} \, K_2^2 + 2 \, \mu^{2(\sigma-1)} \, L_2^2 \right) \, e^{-2\rho |K+L|} \nonumber \\
&\leq \frac{\pi^2}{2} \, \|(\xi,\eta)\|_{\cH^{\rho,0}}^2 \, \left[ (1+2 \, \mu^{2(\sigma-1)} \, K_2^2) 2\pi \, \int_{2/\mu}^{+\infty} e^{-2\rho \ell} \, \ell \di \ell + 4\pi \, \int_{2/\mu}^{+\infty} e^{-2\rho \ell} \, \ell^3 \di \ell \right] \nonumber \\
&=  \frac{\pi^2}{2} \, \|(\xi,\eta)\|_{\cH^{\rho,0}}^2 \times \nonumber \\
&\; \; \; \left[ 2\pi \,\left(1+ 2 \, \mu^{2(\sigma-1)} \,  \frac{9|\log\mu|^2}{\rho^2}\right) \,  e^{-4\rho/\mu} \, \frac{\mu+4\rho}{4\mu\rho^2} + 4\pi \, e^{-4\rho/\mu} \, \frac{3\mu^3+12\rho\mu^2+24\rho^2\mu+32\rho^3}{8\mu^3\rho^4} 
\right]. \label{Est1LowFreqTerm2}
\end{align}
\end{proof}

\begin{proof}[Proof of Proposition \ref{KPxietaProp}]
Proposition \ref{KPxietaProp} is obtained as a Corollary of Proposition \ref{KdVxietaProp} by setting $\sigma = 2$.
\end{proof}

\section{Proof of Proposition \ref{ApprPropKdV} and Proposition \ref{ApprPropKP} } \label{ApprEstSec12}

\begin{proof}[Proof of Proposition \ref{ApprPropKdV}]
The argument follows along the lines of Appendix C in \cite{bambusi2006metastability}.

Exploiting the canonical transformation found in Theorem \ref{gavthm}, we also define
\begin{align}
\zeta_a \, := \, (\xi_a,\eta_a) &=  \cT_{\mu^2}( \widetilde{\xi_a},\widetilde{\eta_a} ) \, = \, \tilde\zeta_a + \psi_a(\tilde\zeta_a),
\end{align}
where $\psi_a(\tilde\zeta_a):=( \psi_\xi(\tilde\zeta_a),\psi_\eta(\tilde\zeta_a) )$; by \eqref{CTthm} we have
\begin{align} \label{estRemThm}
\sup_{\zeta \in \cB_{\rho,n}(R)} \|\psi_a(\zeta)\|_{\cH^{\rho, {n} }} &\leq C'_n \mu^2 \, R.
\end{align}

For convenience we define
\begin{align}
q_a(\tau,y) &:= \frac{1}{\sqrt{2}} \left[ \xi_a(\mu^2 \tau,y_1- \tau,y_2) + \eta_a(\mu^2 \tau,y_1 + \tau,y_2)  \right] \label{qappr} \\
\d_{y_1}p_a(\tau,y) &:= \frac{1}{\sqrt{2}} \left[ \xi_a(\mu^2 \tau,y_1- \tau,y_2) - \eta_a(\mu^2 \tau,y_1 + \tau,y_2)  \right], \label{pappr} 
\end{align}

We observe that the pair $(q_a,p_a)$ satisfies
\begin{align*}
\mu^2 (q_a)_t = -\Delta_1 \, \mu p_a + \mu^{ { \min(2\sigma,6) } } \cR_q, \; &\; \; \mu (p_a)_t = - \mu^2 q_a - \mu^4 \, \alpha \, \overline{\pi_{0}}q_a^2 + \mu^{ { \min(2\sigma-1,5) } } \cR_p,
\end{align*}
where the operator $\Delta_1$ acts on the variable $x$, $\overline{\pi_{0}}$ is the projector on the space of the functions with zero average (defined in Section \ref{galavsec}), and the remainders are functions of the rescaled variables $\tau$ and $y$ which satisfy
\begin{align*}
\sup_{\cB_{\rho,n}(R)} \|\cR_q\|_{\ell^2_{\rho,0}} \leq C, \; &\; \; \sup_{\cB_{\rho,n}(R)} \|\cR_p\|_{\ell^2_{\rho,1}} \leq C.
\end{align*}
We now restrict the space variables to integer values; keeping in mind that $q_a$ and $p_a$ are periodic, we assume that $j \in \Z^2_{N_1,N_1^\sigma}$. 

For a finite sequence $Q=(Q_j)_{j \in  \Z^2_{N_1,N_1^\sigma} }$ we define the norm
\begin{align} \label{NormSeq}
\|Q\|_{\ell^2_{N_1,N_1^\sigma}}^2 &:= \sum_{j \in \Z^2_{N_1,N_1^\sigma}} |Q_j|^2.
\end{align}

Now we consider the discrete model \eqref{2DETLeq}: we rewrite in the following form,
\begin{align}
\dot{Q}_j &= - (\Delta_1 P)_j \label{DiscrEq1} \\
\dot{P}_j &= - Q_j - \, \alpha \, \overline{\pi_{0}}Q_j^2 \label{DiscrEq2}
\end{align}
and we want to show that there exist two sequences $E=(E_j)_{j \in \Z^2_{N_1,N_1^\sigma} }$ and $F=(F_j)_{j \in \Z^2_{N_1,N_1^\sigma} }$ such that
\begin{align*}
Q \, = \, \mu^2 \, q_a + \mu^{2+\gamma} E, &\; \; P \, = \, \mu p_a + \mu^{2+\gamma} F
\end{align*}
fulfills \eqref{DiscrEq1}-\eqref{DiscrEq2}, where $\gamma>0$ is a parameter we will fix later in the proof. Therefore, we have that
\begin{align}
\dot{E} &= - \Delta_1 \, F - \mu^{ {\min(2\sigma,6)} -2-\gamma } \cR_q  \label{EqSeq1} \\
\dot{F} &= - E - \alpha \pi_0\, ( \mu^2 \, 2q_a E + \mu^{2+\gamma} E^2) - \mu^{ {\min(2\sigma-1,5)} -2-\gamma }\cR_p, \label{EqSeq2}
\end{align}
where we impose initial conditions on $(E,F)$ such that $(\tilde{q},\tilde{p})$ has initial conditions corresponding to the ones of the true initial datum,
\begin{align}
\mu^2 q_a(0,\mu j_1,\mu^\sigma j_2) + \mu^{2+\gamma} E_{0,j} &= Q_{0,j}, \nonumber \\
\mu p_a(0,\mu j_1,\mu^\sigma j_2) + \mu^{2+\gamma} F_{0,j} &= P_{0,j} . \nonumber
\end{align}

We now define the operator $\d_{i}$, $i=1,2$,  by $(\d_if)_j := f_j - f_{j-e_i}$ for each $f \in \ell^2_{N_1,N_1^\sigma}$. 

\begin{itemize}
\item Claim 1: Let $\sigma > 2$ and $\gamma >0$, we have
\begin{align*}
\|E_0\|_{\ell^2_{N_1,N_1^\sigma}} \leq C' \mu^{(3-2\gamma-\sigma)/2}, \; \|\d_1F_0\|_{\ell^2_{N_1,N_1^\sigma}} &\leq C' \mu^{(3-2\gamma-\sigma)/2}, \; \|\d_2F_0\|_{\ell^2_{N_1,N_1^\sigma}} \leq C' \mu^{(1-2\gamma+\sigma)/2}.
\end{align*}
\end{itemize}
To prove Claim 1 we observe that
\begin{align*}
E_0 &= \mu^2 \frac{ \xi_a+\eta_a-(\widetilde{\xi}_a+\widetilde{\eta}_a) }{\sqrt{2} \mu^{2+\gamma} }\, = \, \mu^{-\gamma} \, \frac{ \psi_\xi+\psi_\eta }{\sqrt{2}}, \\
F_0 &= \mu \frac{\d_{y_1}^{-1} [ \xi_a-\eta_a-(\widetilde{\xi}_a-\widetilde{\eta}_a) ] }{\sqrt{2} \mu^{2+\gamma} }\, = \, \mu^{-1-\gamma} \frac{ \d_{y_1}^{-1} (\psi_\xi-\psi_\eta) }{\sqrt{2}}, \\
\end{align*}
from which we can deduce
\begin{align*}
\|E_0\|_{\ell^2_{N_1,N_1^\sigma}}^2 &\leq \sum_{j \in \Z^2_{N_1,N_1^\sigma} } |E_{0,j}|^2 \,\leq \, C \, 4 N_1^{\sigma+1} \, (\mu^{2-\gamma})^2 \,  = \, C \, \mu^{3-2\gamma-\sigma}\, , \\
\|\d_1F_0\|_{\ell^2_{N_1,N_1^\sigma}}^2 &\leq \sum_{j \in \Z^2_{N_1,N_1^\sigma} } |\d_1F_{0,j}|^2 \, \leq \, \, C \, 4 N_1^{\sigma+1} \, (\mu^{2-\gamma})^2 \, \leq \, C \, \mu^{3-2\gamma-\sigma} \, , \\
\|\d_2F_0\|_{\ell^2_{N_1,N_1^\sigma}}^2 &\leq \sum_{j \in \Z^2_{N_1,N_1^\sigma} } |\d_2F_{0,j}|^2 \, \leq \,\, C \, 4 N_1^{\sigma+1} \, (\mu^{1+\sigma-\gamma})^2 \, = \, C \, \mu^{1-2\gamma+\sigma}
\end{align*}
and this leads to the thesis.

\begin{itemize}
\item Claim 2: Fix $n \geq 0$, $T_0 > 0$  and $K_\ast >0$, then for any $\mu < \mu_s$ and for any $\sigma > 2$ and {$\gamma >0 $} such that {$\sigma+2\gamma < \min(4\sigma-5,7)$} we have
\begin{align}
\|E\|_{\ell^2_{N_1,N_1^\sigma}}^2  +  \|\d_1F\|_{\ell^2_{N_1,N_1^\sigma}}^2 + \|\d_2F\|_{\ell^2_{N_1,N_1^\sigma}}^2  &\leq K_\ast, \; \; {\color{black} \forall} \, |t| < \frac{T_0}{ \mu^{3} }.
\end{align}
\end{itemize}

To prove the claim, we define
\begin{align} \label{auxFuncClaim2}
\cG(E,F) &:= \sum_{j \in \Z^2_{N_1,N_1^\sigma} } \left(\frac{E_j^2 + F_j (-\Delta_1F)_j}{2} + \frac{ 2\mu^2\alpha q_{a,j} E_j^2}{2}\right) \, ,
\end{align}
and we remark that, using the boundedness of $q_{a,j}$,
\begin{align*}
\frac{1}{2} \cG(E,F) \, \leq \, \|E\|_{\ell^2_{N_1,N_1^\sigma}}^2  &+  \|\d_1F_0\|_{\ell^2_{N_1,N_1^\sigma}}^2 + \|\d_2F_0\|_{\ell^2_{N_1,N_1^\sigma}}^2 \, \leq \, 4 \cG(E,F).
\end{align*}

Now we compute the time derivative of $\cG$. Exploiting \eqref{EqSeq1}-\eqref{EqSeq2}
\begin{align}
\dot{\cG} &= \sum_{j} E_j \left[ -(\Delta_1F)_j - \mu^{ {\min(2\sigma-2,4)} -\gamma } (\cR_q)_j \right] \label{TimeDerAuxFunc1} \\
&\; \; + \sum_{j} (-\Delta_1F)_j \left[ - E_j -\alpha (\mu^2 2q_{a,j} E_j + \mu^{2+\gamma} E_j^2) - \mu^{ {\min(2\sigma-3,3)} -\gamma } (\cR_p)_j \right] \label{TimeDerAuxFunc2} \\
&\; \; + \sum_{j} 2\mu^2 \, \alpha \, q_{a,j} E_j \left[ - (\Delta_1F)_j - \mu^{ {\min(2\sigma-2,4)} -\gamma } (\cR_q)_j \right] \label{TimeDerAuxFunc3} \\
&\; \; + \sum_{j} \mu^2\alpha E_j^2 \, \mu \, \frac{\d q_{a,j}}{\d\tau} \label{TimeDerAuxFunc4} \\
&= \sum_{j} - E_j \, \mu^{ {\min(2\sigma-2,4)} -\gamma } (\cR_q)_j + \sum_{j} (-\Delta_1F)_j \left[  -\alpha \,  \mu^{2+\gamma} E_j^2 - \mu^{ {\min(2\sigma-3,3)} -\gamma } (\cR_p)_j \right] \label{TimeDerAuxFunc21} \\
&\; \; - \sum_{j} 2\mu^2 \, \alpha \, q_{a,j} E_j  \, \mu^{ {\min(2\sigma-2,4)} -\gamma } (\cR_q)_j + \sum_{j} \mu^2\alpha E_j^2 \, \mu \, \frac{\d q_{a,j}}{\d\tau} \label{TimeDerAuxFunc22}
\end{align}
In order to estimate \eqref{TimeDerAuxFunc21}-\eqref{TimeDerAuxFunc22}, we notice that
\begin{align*}
\sup_j |(\Delta_1 F)_j| &\leq 2 \sup_j |(\d_1F)_j| + |(\d_2F)_j| \leq 4 \sqrt{\cG}, \\
\|\cR_q\|^2_{\ell^2_{N_1,N_1^\sigma}} \, \leq \, \sum_j |(\cR_q)_j|^2 & \leq 4N_1^{\sigma+1} \sup_y |\cR_q(y)|^2 \, \leq \, C \mu^{-1-\sigma},
\end{align*}
and that $ |(\d_i\cR_p)_j| \leq \mu \sup_y \left| \frac{\d \cR_p}{\d y} (y) \right|$, which implies 
\begin{align*}
\|\d_i\cR_p\|^2_{\ell^2_{N_1,N_1^\sigma}} & \leq  C \mu^{1-\sigma}.
\end{align*}

Now, the first sum in \eqref{TimeDerAuxFunc21} is estimated by $C \cG^{1/2} \mu^{( {2\min(2\sigma-2,4)-1} -2\gamma-\sigma)/2}$ ; the second sum in \eqref{TimeDerAuxFunc21} can be bounded by
\begin{equation*}
C (\mu^{2+\gamma} \cG^{3/2} + \mu^{( {2\min(2\sigma-3,3)+1} -2\gamma-\sigma)/2} \cG^{1/2}).
\end{equation*}
Recalling that $q_{a,j}$ is bounded, the first sum in \eqref{TimeDerAuxFunc22} can be bounded by $C \cG^{1/2} \mu^{( {2\min(2\sigma,6)-1 }  -2\gamma-\sigma)/2}$, while the second one is estimated by $C \mu^{3} \cG$. 
Hence, as long as $\cG < 2 K_\ast$ we have
\begin{align}
\left| \dot{\cG} \right| &\leq C \left| ( \mu^{( {\min(4\sigma-5,7)} -2\gamma-\sigma)/2} + \mu^{( {\min(4\sigma-5,7)} -2\gamma-\sigma)/2} + \mu^{( {\min(4\sigma-1,11) } -2\gamma-\sigma)/2} ) \cG^{1/2} + \mu^{2+\gamma} \cG^{3/2} + \mu^{3} \cG \right| \nonumber \\
&\leq C (\mu^{2+\gamma} \, \sqrt{2} \, K_{\ast}^{1/2}+\mu^3) \cG + C { ( 2 \mu^{( \min(4\sigma-5,7) -2\gamma-\sigma)/2}  + \mu^{( \min(4\sigma-1,11) -2\gamma-\sigma)/2} ) } \sqrt{2} \, K_{\ast}^{1/2}, \label{EstTimeDer1} \\
&\stackrel{\gamma \geq 1}{\leq} C \, \mu^{3} \, 2 \sqrt{2} \, K_{\ast}^{1/2} \cG +  C \, 3 \mu^{( {\min(4\sigma-5,7)} -2\gamma-\sigma)/2} \, \sqrt{2} \, K_{\ast}^{1/2}, \label{EstTimeDer2}
\end{align}
and by applying Gronwall's lemma we get
\begin{align} \label{Gronwall}
\cG(t) &\leq \cG(0) e^{C \, 2 \sqrt{2} \, K_{\ast}^{1/2} \, \mu^3 t} + e^{C \, 2 \sqrt{2} \, K_{\ast}^{1/2} \, \mu^3 t} \; C \, 2 \sqrt{2} \, K_{\ast}^{1/2} \, \mu^3 t \; C \, 3 \mu^{( {\min(4\sigma-5,7)} -2\gamma-\sigma)/2} \, \sqrt{2} \, K_{\ast}^{1/2},
\end{align}
from which we can deduce the thesis.
\end{proof}

\begin{proof}[Proof of Proposition \ref{ApprPropKP}]
Proposition \ref{ApprPropKP} can be obtained from Proposition \ref{ApprPropKdV} by setting $\sigma = 2$.
\end{proof}

\section{Proof of Proposition \ref{1DNLSpsiProp}} \label{ApprEstSec21}

We argue as in the proof of Proposition \eqref{KdVxietaProp}.

First we remark that for all $k$ such that $\kappa(k) = (\mu K_1,\mu^\sigma K_2)$ we have
\begin{align}
| \omega^2_k | &\stackrel{\eqref{FreqNormModeKG}}{=} 1 + 4 \left[ \sin^2\left(\frac{k_1 \pi}{2N_1+1}\right) + \sin^2\left(\frac{k_2 \pi}{2N_2+1}\right) \right] \nonumber \\
&= 1 + 4 \left[ \sin^2\left(\frac{\mu K_1 \pi}{2}\right) + \sin^2\left(\frac{\mu^\sigma K_2 \pi}{2}\right) \right] \nonumber \\
&\leq 1 + \pi^2 (\mu^2 \, K_1^2 + \mu^{2\sigma} \, K_2^2) \; \leq \; \pi^2 (1 + \mu^2 \, K_1^2 + \mu^{2\sigma} \, K_2^2), \label{EstFreq1DNLSr}
\end{align}
hence 
\begin{align}
\frac{|\hat{p}_K|^2+\pi^2(1 + \mu^2 K_1^2 + \mu^{2\sigma} K_2^2)|\hat{q}_K|^2}{2} &\leq \pi^2 \, e^{-2\rho|K|} \; \frac{|\hat{p}_K|^2+(1 + \mu^2 K_1^2 + \mu^{2\sigma} K_2^2)|\hat{q}_K|^2}{2} e^{2\rho|K|} \nonumber \\
&\leq \pi^2 \, e^{-2\rho|K|} \; \left(1 + \mu^2 \, K_1^2 + \mu^{2\sigma} \, K_2^2\right) \; \|(\psi,\bar\psi)\|_{\cH^{\rho,0}}^2. \label{NormMode1DNLSEst} 
\end{align}
{Hence, analogously to Appendix \ref{ApprEstSec11}, it is convenient to define $\mathscr{L}_{\mu,\delta,\rho}$ as in \eqref{eq:ScrL} and} by \eqref{SpecEnNormMode1DNLS} we obtain that 
for all $k$ such that $\kappa(k) = (\mu K_1,\mu^{\sigma} K_2)$  and $|K_1|+ |K_2| > \frac{(2+\delta) |\log \mu|}{\rho}$
\begin{align}
&\frac{ \cE_\kappa }{\mu^2} \nonumber \leq \sum_{ \substack{ L \in \mathscr{L}_{\mu,\delta,\rho} } } \left( |\hat{p}_{K+L}|^2 + \omega_k^2 \left| \hat{q}_{K+L} \right|^2 \right)  \nonumber \\
&\stackrel{ \eqref{EstFreq1DNLSr},\eqref{NormMode1DNLSEst} }{\leq} \pi^2 \; \|(\psi,\bar\psi)\|_{\cH^{\rho,0}}^2 \; 2 \sum_{ L\in \mathscr{L}_{\mu,\delta,\rho}  } e^{-2\rho|K+L|} \left[1 + \mu^2 \, (K_1+L_1)^2 + \mu^{2\sigma} \, (K_2+L_2)^2\right], \\
&\quad =\sum_{ \substack{  L\in \mathscr{L}_{\mu,\delta,\rho} } } e^{-2\rho|K+L|} + \mu^2 \sum_{ \substack{  L \in \mathscr{L}_{\mu,\delta,\rho}} } e^{-2\rho|K+L|} (K_1+L_1)^2 + \mu^{2\sigma} \sum_{ L \in \mathscr{L}_{\mu,\delta,\rho} } e^{-2\rho|K+L|} (K_2+L_2)^2.
\end{align}
Now,
\begin{align}
&\sum_{ L \in \mathscr{L}_{\mu,\delta,\rho} } e^{-2\rho|K+L|} \leq e^{-2\rho|K|} + \sum_{ \substack{  L \in \mathscr{L}_{\mu,\delta,\rho} \\ L_1=0,L_2 \neq 0  } } e^{-2\rho|K+L|} + \sum_{ \substack{  L \in \mathscr{L}_{\mu,\delta,\rho} \\ L_1 \neq 0, L_2=0 } } e^{-2\rho|K+L|}  + \sum_{ \substack{ L \in \mathscr{L}_{\mu,\delta,\rho} \\ L_1,L_2 \neq 0 } } e^{-2\rho|K+L|}, \label{DecompHighModesTerm1}
\end{align}
and we can estimate the above terms as for \eqref{DecompHighFreqTerm1} in Proposition \ref{KdVxietaProp}; indeed, by \eqref{Est1HighFreqTerm1}, \eqref{Est2HighFreqTerm1} and \eqref{Est3HighFreqTerm1} we have that the last sum in \eqref{DecompHighModesTerm1} is bounded by
\begin{align}
&e^{-2\rho|K|} + \pi \left( \frac{1}{2\rho^2} + 2|K| \right) e^{-2\rho |K|} + e^{-2\rho\, (|K_1|+|K_2|)} \sum_{ \substack{  \ell \in \Z \setminus \{0\}   } } e^{-4 \rho |\ell|/\mu} + e^{-2\rho\, (|K_1|+|K_2|)} \sum_{ \substack{  \ell \in \Z \setminus \{0\}   } } e^{-4 \rho |\ell|/\mu^\sigma}. \label{EstHighModesTerm1}
\end{align}

Next, we have
\begin{align}
\sum_{ \substack{ L \in \mathscr{L}_{\mu,\delta,\rho}  } } e^{-2\rho|K+L|} \; (K_1+L_1)^2 &\leq e^{-2\rho|K|} \, K_1^2 \nonumber  + \sum_{ \substack{ L \in \mathscr{L}_{\mu,\delta,\rho} \\ L_1 \neq 0, L_2=0} } e^{-2\rho|K+L|} \; (K_1+L_1)^2 + \sum_{ \substack{L \in \mathscr{L}_{\mu,\delta,\rho}  \\ L_1=0,L_2 \neq 0} } e^{-2\rho|K+L|} \; K_1^2 \nonumber \\
& + \sum_{ \substack{ L \in \mathscr{L}_{\mu,\delta,\rho} \\ L_1,L_2 \neq 0} } e^{-2\rho|K+L|}\; (K_1+L_1)^2. \label{DecompHighModesTerm2}
\end{align}
First we estimate the last term in \eqref{DecompHighModesTerm2}: 
we have that $|L+K| \geq |K|$, hence
\begin{align}
\sum_{ \substack{L \in \mathscr{L}_{\mu,\delta,\rho} \\ L_1,L_2 \neq 0} } e^{-2\rho|K+L|} \; (K_1+L_1)^2 &= \int_{|K|}^{+\infty} \int_0^{2\pi} e^{-2\rho \, \xi} \, \xi \cos^2\phi \, \di\phi \, \di \xi \; = \; \pi \, e^{-2\rho |K|} \, \frac{1+2\rho |K|}{4\rho^2} \nonumber \\
&\leq \pi \, \mu^4 \, e^{ -2\rho \left[ |K| - \frac{2|\log\mu|}{\rho} - \frac{1}{2\rho} \log(2\rho|K|) \right] } 
\stackrel{\delta < 1 - 1/e}{\leq} \pi \, \mu^4 \, e^{ -2\rho \, \left[ \delta |K| - \frac{2 |\log\mu|}{\rho} \right] } \, .
\label{Est1HighModesTerm2}
\end{align}

Now we bound the other two nontrivial terms in \eqref{DecompHighModesTerm2}; on the one hand, we notice that
\begin{align}
&\sum_{ \substack{ L \in \mathscr{L}_{\mu,\delta,\rho} \\  L_1=0,L_2 \neq 0} } e^{-2\rho|K+L|} (K_1+L_1)^2 \leq 2 \sum_{ \substack{ L \in \mathscr{L}_{\mu,\delta,\rho} \\  L_1=0,L_2 \neq 0} } e^{-2\rho|K+L|} \, K_1^2 +2 \sum_{ \substack{ L \in \mathscr{L}_{\mu,\delta,\rho} \\  L_1=0,L_2 \neq 0} } e^{-2\rho|K+L|} \, L_1^2, \label{DecompHighModesTerm22}
\end{align}
where the first sum can be bounded as the second term in \eqref{DecompHighModesTerm1}, while 
\begin{align}
\sum_{ \substack{ L \in \mathscr{L}_{\mu,\delta,\rho} \\  L_1 = 0,L_2 \neq 0} } e^{-2\rho|K+L|} L_1^2 &\leq e^{-2\rho|K|} \sum_{ \substack{  \ell \in \Z \setminus \{0\}   } } e^{-4 \rho |\ell|/\mu} \, \frac{\ell^2}{\mu^{2}} \leq 2 e^{-2\rho|K|} \int_1^{+\infty} e^{-4 \rho |\ell|/\mu} \, \frac{\ell^2}{\mu^{2}} \, \di \ell, \label{DecompHighModesTerm23}
\end{align}
where the last integral is exponentially small with respect to $\mu$. Similarly,
\begin{align}
\sum_{ \substack{ L \in \mathscr{L}_{\mu,\delta,\rho} \\  L_1 = 0,L_2 \neq 0} } e^{-2\rho|K+L|} L_2^2 &\leq e^{-2\rho|K|} \sum_{ \substack{  \ell \in \Z \setminus \{0\}   } } e^{-4 \rho |\ell|/\mu^\sigma} \, \frac{\ell^2}{\mu^{2\sigma}} \leq 2 e^{-2\rho|K|} \int_1^{+\infty} e^{-4 \rho |\ell|/\mu^\sigma} \, \frac{\ell^2}{\mu^{2\sigma}} \, \di \ell, \label{DecompHighModesTerm33}
\end{align}
where the last integral is exponentially small with respect to $\mu$. \\

On the other hand, for any $k$ such that $\kappa(k) = (\mu K_1,\mu^{\sigma} K_2)$ and $|K_1|+|K_2| \leq \frac{(2+\delta) |\log\mu|}{\rho}$
\begin{align}
&\left| \frac{\cE_\kappa}{\mu^{2}} - \frac{|\hat{\psi}_K|^2}{2} \right| \leq \left| \omega_k^2 - 1 \right| |\hat{q}_K|^2 + \frac{1}{2} \sum_{ \substack{ L=(L_1,L_2) \in \Z^2 \setminus \{0\} \\ \mu L_1, \mu^{\sigma} L_2 \in 2\Z} } \left(|\hat{p}_{K+L}|^2 + \omega_k^2 \left| \hat{q}_{K+L} \right|^2\right) \nonumber \\
&\stackrel{\eqref{EstFreq1DNLSr}}{\leq} (\mu^2 \pi^2 K_1^2 + \pi^2 \mu^{2\sigma} K_2^2) |\hat{p}_K|^2 \nonumber \\
&\; \; \; \; + \frac{1}{2} \sum_{ \substack{ L=(L_1,L_2) \in \Z^2 \setminus \{0\} \\ \mu L_1, \mu^{\sigma} L_2 \in 2\Z} } \left(|\hat{p}_{K+L}|^2 + |\hat{q}_{K+L}|^2 + \pi^2[\mu^2 (K_1+L_1)^2+\mu^{2\sigma}(K_2+L_2)^2] |\hat{q}_{K+L}|^2\right), \nonumber \\
&\leq \left( \pi^2 \, \mu^2 K_1^2 + \pi^2 \mu^{2\sigma} \, K_2^2 \right) |\hat{p}_K|^2 \nonumber \\
&\; \; \; \; + \|(\psi,\bar\psi)\|_{\cH^{\rho,0}}^2 \sum_{ \substack{ L=(L_1,L_2) \in \Z^2 \setminus \{0\} \\ \mu L_1, \mu^{\sigma} L_2 \in 2\Z} } e^{-2\rho |K+L|} [ 1+ \pi^2 \mu^2 (K_1+L_1)^2 + \pi^2 \mu^{2\sigma} (K_2+L_2)^2 ] \nonumber \\
&\leq \pi^2 \, \mu^2 \left( 1 + \mu^{2(\sigma-1)} \, \right) \frac{9 |\log\mu|^2}{\rho^2} \|(\psi,\bar\psi)\|_{\cH^{\rho,0}}^2  \label{EstLowModesTerm1} \\
&\; \; \; \; +\|(\psi,\bar\psi)\|_{\cH^{\rho,0}}^2 \sum_{ \substack{ L=(L_1,L_2) \in \Z^2 \setminus \{0\} \\ \mu L_1, \mu^{\sigma} L_2 \in 2\Z} } e^{-2\rho |K+L|} \label{EstLowModesTerm2} \\
&\; \; \; \; + \pi^2 \mu^2 \|(\psi,\bar\psi)\|_{\cH^{\rho,0}}^2 \sum_{ \substack{ L=(L_1,L_2) \in \Z^2 \setminus \{0\} \\ \mu L_1, \mu^{\sigma} L_2 \in 2\Z} } e^{-2\rho |K+L|}(K_1+L_1)^2 \label{EstLowModesTerm3} \\
&\; \; \; \; + \pi^2 \mu^{2\sigma} \|(\psi,\bar\psi)\|_{\cH^{\rho,0}}^2 \sum_{ \substack{ L=(L_1,L_2) \in \Z^2 \setminus \{0\} \\ \mu L_1, \mu^{\sigma} L_2 \in 2\Z} } e^{-2\rho |K+L|}(K_2+L_2)^2 \label{EstLowModesTerm4} 
\end{align}

and we can conclude by estimating \eqref{EstLowModesTerm1} by exploiting the fact that {$|\log\mu| \leq \mu^{-1/4}$}, while we can bound \eqref{EstLowModesTerm2}-\eqref{EstLowModesTerm3} by
\begin{align}
&\frac{\pi^2}{2} \, \|(\psi,\bar\psi)\|_{\cH^{\rho,0}}^2  \, \sum_{ \substack{ L=(L_1,L_2) \in \Z^2 \setminus \{0\} \\ \mu L_1, \mu^{\sigma} L_2 \in 2\Z} } \, [1+\mu^2(K_1+L_1)^2]\, e^{-2\rho |K+L|} \nonumber \\
&\leq \frac{\pi^2}{2} \, \|(\psi,\bar\psi)\|_{\cH^{\rho,0}}^2   \, \sum_{ \substack{ L=(L_1,L_2) \in \Z^2 \setminus \{0\} \\ \mu L_1, \mu^{\sigma} L_2 \in 2\Z} } \, (1+2 \, \mu^2 \, K_1^2 + 2 \, \mu^2 \, L_1^2) \, e^{-2\rho |K+L|} \nonumber \\
&\leq \frac{\pi^2}{2} \, \|(\psi,\bar\psi)\|_{\cH^{\rho,0}}^2 \,  \left[ (1+2 \, \mu^2 \, K_1^2) \, 2\pi \, \int_{2/\mu}^{+\infty} e^{-2\rho \ell} \, \ell \,  \di \ell + 4\pi \, \mu^2 \int_{2/\mu}^{+\infty} e^{-2\rho \ell} \, \ell^3 \,\di \ell \right] \nonumber \\
&=  \frac{\pi^2}{2} \, \|(\psi,\bar\psi)\|_{\cH^{\rho,0}}^2 \, \times \nonumber \\
&\; \; \; \left[ 2\pi \, \left(1 + 2\mu^2 \frac{9|\log\mu|^2}{\rho^2} \right)\,  e^{-4\rho/\mu} \, \frac{\mu+4\rho}{4\mu\rho^2} + 4\pi \, \mu^2 e^{-4\rho/\mu} \, \frac{3\mu^3+12\rho\mu^2+24\rho^2\mu+32\rho^3}{8\mu^3\rho^4} 
\right], \label{Est1LowModesTerm3}
\end{align}
and we can estimate \eqref{EstLowModesTerm4} by
\begin{align}
&\frac{\pi^2}{2} \, \|(\psi,\bar\psi)\|_{\cH^{\rho,0}}^2 \, \mu^{2(\sigma-1)} \, \sum_{ \substack{ L=(L_1,L_2) \in \Z^2 \setminus \{0\} \\ \mu L_1, \mu^{\sigma} L_2 \in 2\Z} } \, (K_2+L_2)^2\, e^{-2\rho |K+L|} \nonumber \\
&\leq \frac{\pi^2}{2} \, \|(\psi,\bar\psi)\|_{\cH^{\rho,0}}^2 \, \mu^{2(\sigma-1)}  \, \sum_{ \substack{ L=(L_1,L_2) \in \Z^2 \setminus \{0\} \\ \mu L_1, \mu^{\sigma} L_2 \in 2\Z} } \, (2 K_2^2 + 2 L_2^2) \, e^{-2\rho |K+L|} \nonumber \\
&\leq \frac{\pi^2}{2} \, \|(\psi,\bar\psi)\|_{\cH^{\rho,0}}^2 \, \mu^{2(\sigma-1)} \,  \left[ 2K_1^2 \, 2\pi \, \int_{2/\mu^\sigma}^{+\infty} e^{-2\rho \ell} \, \ell \di \ell + 4\pi \, \int_{2/\mu^\sigma}^{+\infty} e^{-2\rho \ell} \, \ell^3 \di \ell \right] \nonumber \\
&=  \frac{\pi^2}{2} \, \|(\psi,\bar\psi)\|_{\cH^{\rho,0}}^2 \, \mu^{2(\sigma-1)} \times \nonumber \\
&\; \; \; \left[ 4\pi \,  \frac{9|\log\mu|^2}{\rho^2} \,  e^{-4\rho/\mu^\sigma} \, \frac{\mu^\sigma+4\rho}{4\mu^\sigma\rho^2} + 4\pi \, e^{-4\rho/\mu^\sigma} \, \frac{3\mu^{3\sigma}+12\rho\mu^{2\sigma}+24\rho^2\mu^\sigma+32\rho^3}{8\mu^{3\sigma}\rho^4} 
\right]. \label{Est1LowModesTerm4}
\end{align}

\section{Proof of Proposition \ref{ApprProp1DNLS}} \label{ApprEstSec22}

The argument follows along the lines of Appendix C in \cite{bambusi2006metastability}.

Exploiting the canonical transformation found in Theorem \ref{gavthm}, we also define
\begin{align}
\zeta_a \, := \, (\psi_a,\bar\psi_a) &=  \cT_{\mu^2}( \widetilde{\psi_a},\bar{\widetilde{\psi_a}} ) \, = \, \tilde\zeta_a + \phi_a(\tilde\zeta_a),
\end{align}
where $\phi_a(\tilde\zeta_a):=( \phi_\xi(\tilde\zeta_a),\phi_\eta(\tilde\zeta_a) )$; by \eqref{CTthm} we have
\begin{align} \label{estRemThm3}
\sup_{\zeta \in \cB_{\rho,n}(R)} \|\phi_a(\zeta)\|_{\cH^{\rho,n}} &\leq C'_n \mu^2 \, R.
\end{align}

For convenience we define
\begin{align}
q_a(\tau,y) &:= \frac{1}{\sqrt{2}} \left[ e^{\ii\tau} \widetilde{\psi_a}(\tau,y_1,y_2) + e^{-\ii\tau} \bar{\widetilde{\psi_a}}(\tau,y_1,y_2)  \right] \label{qappr3} \\
p_a(\tau,y) &:= \frac{1}{\sqrt{2}\ii} \left[ e^{i\tau} \widetilde{\psi_a}(\tau,y_1,y_2) - e^{-\ii\tau} \bar{\widetilde{\psi_a}}(\tau,y_1,y_2)  \right], \label{pappr3} 
\end{align}

We observe that the pair $(q_a,p_a)$ satisfies

\begin{align*}
\mu (q_a)_t = \mu p_a + \mu^{ {\min(2\sigma+1,5)} } \cR_q, \; &\; \; \mu (p_a)_t = - \mu q_a + \mu \Delta_1 q_a- \mu^3 \, \beta \, \overline{\pi_{0}}q_a^3 + \mu^{ {\min(2\sigma+1,5)} } \cR_p,
\end{align*}
where the operator $\Delta_1$ acts on the variable $x$, $\overline{\pi_{0}}$ is the projector on the space of the functions with zero average, and the remainders are functions of the rescaled variables $\tau$ and $y$ which satisfy
\begin{align*}
\sup_{\cB_{\rho,n}(R)} \|\cR_q\|_{\ell^2_{\rho,0}} \leq C, \; &\; \; \sup_{\cB_{\rho,n}(R)} \|\cR_p\|_{\ell^2_{\rho,1}} \leq C.
\end{align*}
We now restrict the space variables to integer values; keeping in mind that $q_a$ and $p_a$ are periodic, we assume that $j \in \Z^2_{N_1,N_1^\sigma}$. 

For a finite sequence $Q=(Q_j)_{j \in  \Z^2_{N_1,N_1^\sigma} }$ we use the norm $\|Q\|_{\ell^2_{N_1,N_1^\sigma}}^2$ defined in \eqref{NormSeq}.

Now we consider the discrete model \eqref{2DETLeq}: we rewrite in the following form,
\begin{align}
\dot{Q}_j &= P_j \label{DiscrEq31} \\
\dot{P}_j &= - Q_j + (\Delta_1Q)_j - \, \beta \, \overline{\pi_{0}}Q_j^3 \label{DiscrEq32}
\end{align}
and we want to show that there exist two sequences $E=(E_j)_{j \in \Z^2_{N_1,N_1^\sigma} }$ and $F=(F_j)_{j \in \Z^2_{N_1,N_1^\sigma} }$ such that
\begin{align*}
Q \, = \, \mu \, q_a + \mu^{1+\gamma} E, &\; \; P \, = \, \mu p_a + \mu^{1+\gamma} F
\end{align*}
fulfills \eqref{DiscrEq31}-\eqref{DiscrEq32}, where $\gamma>0$ is a parameter we will fix later in the proof. Therefore, we have that
\begin{align}
\dot{E} &= F - \mu^{ {\min(2\sigma+1,5)} -1-\gamma} \cR_q  \label{EqSeq31} \\
\dot{F} &= - E + \Delta_1E  - \beta \overline{\pi_0}\, ( 3 \mu^{2} \, q_a^2 E + 3 \mu^{2+\gamma} \, q_a E^2 + \mu^{2+2\gamma} E^3) - \mu^{ {\min(2\sigma+1,5)} -1-\gamma}\cR_p, \label{EqSeq32}
\end{align}
where we impose initial conditions on $(E,F)$ such that $(\tilde{q},\tilde{p})$ has initial conditions corresponding to the ones of the true initial datum,
\begin{align}
\mu q_a(0,\mu j_1,\mu^\sigma j_2) + \mu^{1+\gamma} E_{0,j} &= Q_{0,j}, \nonumber \\
\mu p_a(0,\mu j_1,\mu^\sigma j_2) + \mu^{1+\gamma} F_{0,j} &= P_{0,j} . \nonumber
\end{align}

We now define the operator $\d_{i}$, $i=1,2$,  by $(\d_if)_j := f_j - f_{j-e_i}$ for each $f \in \ell^2_{N_1,N_1^\sigma}$. 

\begin{itemize}
\item Claim 1: Let $\sigma > 1$ and $\gamma >0$, we have
\begin{align*}
\|E_0\|_{\ell^2_{N_1,N_1^\sigma}} \leq C' \mu^{(3-2\gamma-\sigma)/2}, \; \|F_0\|_{\ell^2_{N_1,N_1^\sigma}} &\leq C' \mu^{(3-2\gamma-\sigma)/2}, \; \|\d_1E_0\|_{\ell^2_{N_1,N_1^\sigma}} \leq C' \mu^{(5-2\gamma-\sigma)/2}, \\
\|\d_2E_0\|_{\ell^2_{N_1,N_1^\sigma}} \leq C' \mu^{(3-2\gamma+\sigma)/2}, \; \|\d_1F_0\|_{\ell^2_{N_1,N_1^\sigma}} &\leq C' \mu^{(5-2\gamma-\sigma)/2}, \; \|\d_2F_0\|_{\ell^2_{N_1,N_1^\sigma}} \leq C' \mu^{(3-2\gamma+\sigma)/2}.
\end{align*}
\end{itemize}
To prove Claim 1 we observe that
\begin{align*}
E_0 &= \mu \frac{ \psi_a+\bar\psi_a-(\widetilde{\psi}_a+\bar{\widetilde{\psi}_a}) }{\sqrt{2} \mu^{1+\gamma} }\, = \, \mu^{-\gamma} \, \frac{ \phi_\xi+\phi_\eta }{\sqrt{2}}, \\
F_0 &= \mu \frac{ \psi_a-\bar\psi_a-(\widetilde{\psi}_a-\bar{\widetilde{\psi}_a}) ] }{\sqrt{2}\ii \, \mu^{1+\gamma} }\, = \, \mu^{-\gamma} \frac{ \phi_\xi-\phi_\eta }{\sqrt{2}\ii }, \\
\end{align*}
from which we can deduce
\begin{align*}
\|E_0\|_{\ell^2_{N_1,N_1^\sigma}}^2 &\leq \sum_{j \in \Z^2_{N_1,N_1^\sigma} } |E_{0,j}|^2 \,\leq \, C \, 4 N_1^{\sigma+1} \, (\mu^{2-\gamma})^2 \,  = \, C \, \mu^{3-2\gamma-\sigma}, \\
\|F_0\|_{\ell^2_{N_1,N_1^\sigma}}^2 &\leq \sum_{j \in \Z^2_{N_1,N_1^\sigma} } |F_{0,j}|^2 \,\leq \, C \, 4 N_1^{\sigma+1} \, (\mu^{2-\gamma})^2 \,  = \, C \, \mu^{3-2\gamma-\sigma}, \\
\|\d_1E_0\|_{\ell^2_{N_1,N_1^\sigma}}^2 &\leq \sum_{j \in \Z^2_{N_1,N_1^\sigma} } |\d_1E_{0,j}|^2 \, \leq \, \, C \, 4 N_1^{\sigma+1} \, (\mu^{2+1-\gamma})^2 \, \leq \, C \, \mu^{5-2\gamma-\sigma}, \\
\|\d_2E_0\|_{\ell^2_{N_1,N_1^\sigma}}^2 &\leq \sum_{j \in \Z^2_{N_1,N_1^\sigma} } |\d_2E_{0,j}|^2 \, \leq \,\, C \, 4 N_1^{\sigma+1} \, (\mu^{2+\sigma-\gamma})^2 \, = \, C \, \mu^{3-2\gamma+\sigma}, \\
\|\d_1F_0\|_{\ell^2_{N_1,N_1^\sigma}}^2 &\leq \sum_{j \in \Z^2_{N_1,N_1^\sigma} } |\d_1F_{0,j}|^2 \, \leq \, \, C \, 4 N_1^{\sigma+1} \, (\mu^{2+1-\gamma})^2 \, \leq \, C \, \mu^{5-2\gamma-\sigma}, \\
\|\d_2F_0\|_{\ell^2_{N_1,N_1^\sigma}}^2 &\leq \sum_{j \in \Z^2_{N_1,N_1^\sigma} } |\d_2F_{0,j}|^2 \, \leq \,\, C \, 4 N_1^{\sigma+1} \, (\mu^{2+\sigma-\gamma})^2 \, = \, C \, \mu^{3-2\gamma+\sigma},
\end{align*}
and this leads to the thesis.

\begin{itemize}
\item Claim 2: Fix $n \geq 0$, $T_0 > 0$  and $K_\ast >0$, then for any $\mu < \mu_s$ and for any $\sigma > 1$ and $\gamma > 0$ such that $\sigma+2\gamma < {\min(4\sigma-1,7)}$ we have
\begin{align}
\|E\|_{\ell^2_{N_1,N_1^\sigma}}^2 + \|F\|_{\ell^2_{N_1,N_1^\sigma}}^2 +  \|\d_1E_0\|_{\ell^2_{N_1,N_1^\sigma}}^2 + \|\d_2E_0\|_{\ell^2_{N_1,N_1^\sigma}}^2  &\leq K_\ast, \; \;{\color{black} \forall \, } |t| < \frac{T_0}{ \mu^{2} }.
\end{align}
\end{itemize}

To prove the claim, we define
\begin{align} \label{auxFuncClaim32}
\cG(E,F) &:= \sum_{j \in \Z^2_{N_1,N_1^\sigma} } \left(\frac{F_j^2 + E_j^2 + E_j (-\Delta_1E)_j}{2} + \frac{ 3\mu^2\beta q_a^2 E_j^2 + 3 \mu^{2+\gamma} \beta q_a E_j^3 }{2}\right) ,
\end{align}
and we remark that
\begin{align*}
\frac{1}{2} \cG(E,F) \, \leq \, \|E\|_{\ell^2_{N_1,N_1^\sigma}}^2  &+  \|\d_1F_0\|_{\ell^2_{N_1,N_1^\sigma}}^2 + \|\d_2F_0\|_{\ell^2_{N_1,N_1^\sigma}}^2 \, \leq \, 2 \cG(E,F).
\end{align*}

Now we compute the time derivative of $\cG$. Exploiting \eqref{EqSeq1}-\eqref{EqSeq2} {\color{black}and using a procedure analogous to Appendix \ref{ApprEstSec12}, we finally obtain}
as long as $\cG < 2 K_\ast$ 
\begin{align}
\left| \dot{\cG} \right| &\leq C \left[ \mu^{2+\gamma} \, K_\ast^{1/2} + \mu^{2+2\gamma} \, K_\ast + \mu^{3} + \mu^{2+\gamma} \, K_\ast^{1/2} + \mu^{ {\min(2\sigma+2,6)} -(1+\sigma)/2} + \mu^{2+\gamma} \, K_\ast^{1/2} \right] \cG \label{EstTimeDer31} \\
&\; \; \; + C \,\left[ 2 \mu^{ {\min(2\sigma,4)} -\gamma-(1+\sigma)/2} + \mu^{ {\min(2\sigma,4)} -\gamma+(1-\sigma)/2} + \mu^{ {\min(2\sigma+2,6)} -\gamma-(1+\sigma)/2} \right] K_\ast^{1/2} \label{EstTimeDer32} \\
&\stackrel{\sigma+2\gamma < {\min(4\sigma-1,7)} }{\leq} C \, \mu^2 \, (1+K_\ast^{1/2}) \, \cG + C \, \mu^{( {\min(4\sigma-1,7)} -2\gamma-\sigma)/2} \, K_\ast^{1/2} \label{EstTimeDer33}
\end{align}
and by applying Gronwall's lemma we get
\begin{align} \label{Gronwall3}
\cG(t) &\leq \cG(0) e^{C \,(1+K_\ast^{1/2}) \, \mu^2 t} + e^{C \, (1+K_\ast^{1/2}) \, \mu^2 t} \; C \, (1+K_\ast^{1/2}) \, \mu^2 t \; C \, \mu^{( {\min(4\sigma-1,7)} -2\gamma-\sigma)/2} \, K_{\ast}^{1/2},
\end{align}
from which we can deduce the thesis. \\

\paragraph{\emph{Acknowledgements}} The authors would like to thank Dario Bambusi, Alberto Maspero, Tiziano Penati and Antonio Ponno for useful comments and suggestions. We also thank the anonymous referees for suggesting improvements to the paper. S.P. acknowledges financial support from the Spanish "Ministerio de Ciencia, Innovaci\'on y Universidades", through the María de Maeztu Programme for Units of Excellence (2015-2019) and the Barcelona Graduate School of Mathematics, and partial support by the Spanish MINECO-FEDERGrant MTM2015-65715-P. M.G. acknowledges financial support from the MIUR-PRIN 2017 project MaQuMA cod.~2017ASFLJR, INdAM (GNFM) and also the Department of Mathematics at  "Universitat Polit\'ecnica de Catalunya" where part of this work was carried on. This project has received funding from the European Research Council (ERC) under the European Union's Horizon 2020 research and innovation programme (grant agreement No 757802). 

\end{appendix}

\bibliography{GP_ETL_2019}
\bibliographystyle{alpha}

\end{document}